\documentclass[10pt,a4paper,oneside]{article}
\usepackage[margin=1in]{geometry}
\usepackage[usenames,dvipsnames]{xcolor} 
\usepackage[pdftitle={},
pdfauthor={Jacopo Borga},
colorlinks=true,linkcolor=JungleGreen,urlcolor=OliveGreen,citecolor=PineGreen,bookmarks=true,bookmarksopen=true,bookmarksopenlevel=2,unicode=true,linktocpage]{hyperref}

\usepackage[english]{babel}
\usepackage{authblk}  
\usepackage{amsmath}  
\usepackage{graphicx}  
\usepackage{amsthm} 	
\usepackage[capitalize]{cleveref} 			
\usepackage{amsmath,bm}  
\usepackage{amsfonts}
\usepackage{mathtools} 
\usepackage{dsfont}
\usepackage{autonum}
\usepackage{tikz-cd}
\usepackage{amssymb}
\usepackage{stmaryrd}

\newtheorem{thm}{Theorem}[section]

\newtheorem{prop}[thm]{Proposition}
\newtheorem{lem}[thm]{Lemma}
\newtheorem{conj}[thm]{Conjecture}
\newtheorem{quest}[thm]{Question}

\theoremstyle{definition}
\newtheorem{defn}[thm]{Definition}

\theoremstyle{remark}
\newtheorem{rem}[thm]{Remark}

\usepackage{todonotes}

\def\R{\mathbb{R}}
\def\Z{\mathbb{Z}}
\def\E{\mathbb{E}}
\def\P{\mathbb{P}}

\DeclareMathOperator{\Var}{Var}
\DeclareMathOperator{\Cov}{Cov}
\DeclareMathOperator{\sgn}{sgn}
\DeclareMathOperator{\idf}{\mathds{1}}
\DeclareMathOperator{\Leb}{Leb}
\DeclareMathOperator{\Id}{Id}

\DeclareMathOperator{\Perm}{Perm}
\DeclareMathOperator{\occ}{occ}
\DeclareMathOperator{\pocc}{\widetilde{occ}}
\DeclareMathOperator{\stc}{stc}

\newcommand{\eps}{\varepsilon}
\newcommand{\bmwzm}{\bm W_{[0,m]}}
\newcommand{\dxyent}{\delta_2(x,y,\eps,n,t)}
\newcommand{\wzwb}{\widehat z+ \widehat b}
\newcommand{\sqen}{\sqrt{2\eps n}}
\newcommand{\cnz}{\conti Z}

\newcommand{\conti}[1]{{\bm{\mathcal #1}}}

\title{The skew Brownian permuton: a new universality class for random constrained permutations}

\date{  } 
\author[1]{Jacopo Borga\thanks{\href{mailto:jborga@stanford.edu}{jborga@stanford.edu}}}

\affil[1]{Department of Mathematics, Stanford University}

\makeatletter
\newcommand{\subjclass}[2][1991]{%
	\let\@oldtitle\@title%
	\gdef\@title{\@oldtitle\footnotetext{#1 \emph{Mathematics subject classification.} #2.}}%
}
\newcommand{\keywords}[1]{%
	\let\@@oldtitle\@title%
	\gdef\@title{\@@oldtitle\footnotetext{\emph{Key words and phrases.} #1.}}%
}
\makeatother



\begin{document}

\maketitle

\begin{abstract}
	We construct a new family of random permutons, called \emph{skew Brownian permuton}, which describes the limits of several models of random constrained permutations. This family is parametrized by two real parameters. 
	
	For a specific choice of the parameters, the skew Brownian permuton coincides with the Baxter permuton, i.e.\ the permuton limit of Baxter permutations. We prove that for another specific choice of the parameters, the skew Brownian permuton coincides with the biased Brownian separable permuton, a one-parameter family of permutons previously studied in the literature as the limit of uniform permutations in substitution-closed classes. This brings two different limiting objects under the same roof, identifying a new larger universality class.
	
	The skew Brownian permuton is constructed in terms of flows of solutions of certain stochastic differential equations (SDEs) driven by two-dimensional correlated Brownian excursions in the non-negative quadrant. We call these SDEs \emph{skew perturbed Tanaka equations} because they are a mixture of the perturbed Tanaka equations and the equations encoding skew Brownian motions. We prove existence and uniqueness of (strong) solutions for these new SDEs.
	
	In addition, we show that some natural permutons arising from Liouville quantum gravity spheres decorated with two Schramm-Loewner evolution curves are skew Brownian permutons and such permutons cover almost the whole range of possible parameters. Some connections between constrained permutations and decorated planar maps have been investigated in the literature at the discrete level; this paper establishes this connection directly at the continuum level.
	Proving the latter result, we also give an SDE interpretation of some quantities related to SLE-decorated Liouville quantum gravity spheres. 
\end{abstract}

\tableofcontents

\section{Introduction}

\subsection{Limits of random permutations}

The study of limits of random permutations is a classical topic in probability theory. The typical question is to determine the behavior of a large random permutation when its size tends to infinity. For many years, the main approach to answering this question has been to study the convergence of relevant statistics, such as the number of inversions, the length of the longest increasing subsequence, the number of cycles, and many others.
In the last decade, a more geometric approach has been investigated, mainly to study limits of \emph{non-uniform} models of random permutations. 
In this case, the goal is to directly determine the limit of the permutation itself from a global perspective. This theory goes under the name of \emph{permuton limits}. For a complete introduction to the theory of permuton limits, we refer the reader to \cite[Section 2.1]{borga2021random}. Here we only recall some basic definitions.

A Borel probability measure $\mu$ on the unit square $[0,1]^2$ is called a \emph{permuton} if its marginals are uniform, i.e.\ $\mu( [0,1]\times[x,y])=\mu([x,y]\times[0,1])=y-x$
for all $0\leq x\leq y\leq 1$. 

Presutti and
Stromquist \cite{MR2732835} first used permutons, calling them \emph{normalized measures}, to investigate the packing density of some specific patterns. A more theoretical approach to the study of deterministic permutons (again without using this terminology) was developed by Hoppen, Kohayakawa,
Moreira, Rath, and Sampaio~\cite{hoppen2013limits}. The word \emph{permuton} appeared for the first time in a work of Glebov, Grzesik, Klimošová, and Král~\cite{MR3279390}. They adopt this terminology by analogy with
\emph{graphon} in the theory of random graphs \cite{lovasz2012large}. 
Various results appearing in \cite{hoppen2013limits} were generalized to the case of random permutons by Bassino, Bouvel, F\'{e}ray, Gerin, Maazoun, and Pierrot~\cite{bassino2017universal}.

To a permutation $\sigma$ of size $n \ge 1$, it is possible to associate a natural permuton $\mu_\sigma$, given by the sum of Lebesgue area measures
\begin{equation}
	\mu_\sigma(A)
	= 
	n 
	\sum_{i=1}^n 
	\Leb
	\big([(i-1)/n, i/n]
	\times
	[(\sigma(i)-1)/n,\sigma(i)/n]
	\cap 
	A\big),
\end{equation}
where $A$ is a Borel measurable set of $[0,1]^2$. Note that $\mu_\sigma$ corresponds to the normalized diagram of $\sigma$,
where each dot has been replaced with a square of dimension $1/n \times 1/n$ and total mass $1/n$. See the example given in \cref{fig:iubewduqeubd}.

\begin{figure}[htbp]
	\begin{minipage}[c]{0.54\textwidth}
		\centering
		\includegraphics[scale=1.5]{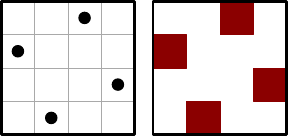}
	\end{minipage}
	\begin{minipage}[c]{0.453\textwidth}
		\caption{\textbf{Left:} The diagram of the permutation $\sigma=3142$. \textbf{Right:} The corresponding permuton. Every small square of the diagram containing a dot is endowed with Lebesgue measure of total mass $1/4$. With this procedure we obtain a probability measure on the unit square with uniform marginals, that is, the permuton $\mu_\sigma$.\label{fig:iubewduqeubd}}
	\end{minipage}
\end{figure}

Let $\mathcal M$ be the set of permutons endowed with the weak topology.
A sequence of permutons $(\mu_n)_{n\in\Z_{>0}}$ converges weakly to $\mu$, and we write $\mu_n \to \mu$, if 
\begin{equation}
	\int_{[0,1]^2} f d\mu_n 
	\rightarrow 
	\int_{[0,1]^2} f d\mu,
\end{equation}
for every continuous function $f: [0,1]^2 \to \mathbb{R}$. When a sequence of permutons $(\mu_n)_{n\in\Z_{>0}}$ converges weakly to $\mu$ we will often say that $(\mu_n)_{n\in\Z_{>0}}$ converges to $\mu$ in the permuton sense.

\medskip

In the past years, permuton limits of various models of non-uniform random permutations have been investigated in the literature. These models can be divided into two main different classes: 
\begin{itemize}
	\item The models that exhibit a \emph{random fractal limiting permuton}: here the classical examples are pattern-avoiding permutations (see below for a more detailed discussion).
	\item The models that exhibit a \emph{deterministic limiting permuton}: here some examples are Erd\"{o}s–Szekeres permutations \cite{MR2266895}, Mallows permutations \cite{starr2009thermodynamic}, random sorting networks \cite{dauvergne2018archimedean}, square\footnote{To be precise, we remark that the permuton limits of square and almost square permutations with a finite number of internal points are random, but their randomness can be simply expressed in terms of a beta distributed random variable. On the contrary, almost square permutations with an infinite number of internal points exhibit the same deterministic permuton limit. For these reasons we prefer to classify square and almost square permutations as permutations that exhibit a deterministic permuton limit.} and almost square permutations \cite{MR4149526,borga2021almost}, and permutations sorted with the \emph{runsort} algorithm \cite{alon2021runsort}.
\end{itemize}
The present paper focuses on models that exhibit a limiting random fractal permuton. Our main goal is to introduce a new family of limiting permutons that unifies the various instances of random fractal limiting permutons that appear in the literature and some new ones.

\bigskip

We now quickly review the literature on models of \emph{pattern-avoiding permutations}  that exhibit a random fractal permuton in the limit. If the reader is not familiar with the terminology related to pattern-avoiding permutations, he/she can find a quick introduction in \cite[Section 1.6.1]{borga2021random}.

\begin{figure}[htbp]
	\centering
	\includegraphics[scale=0.1]{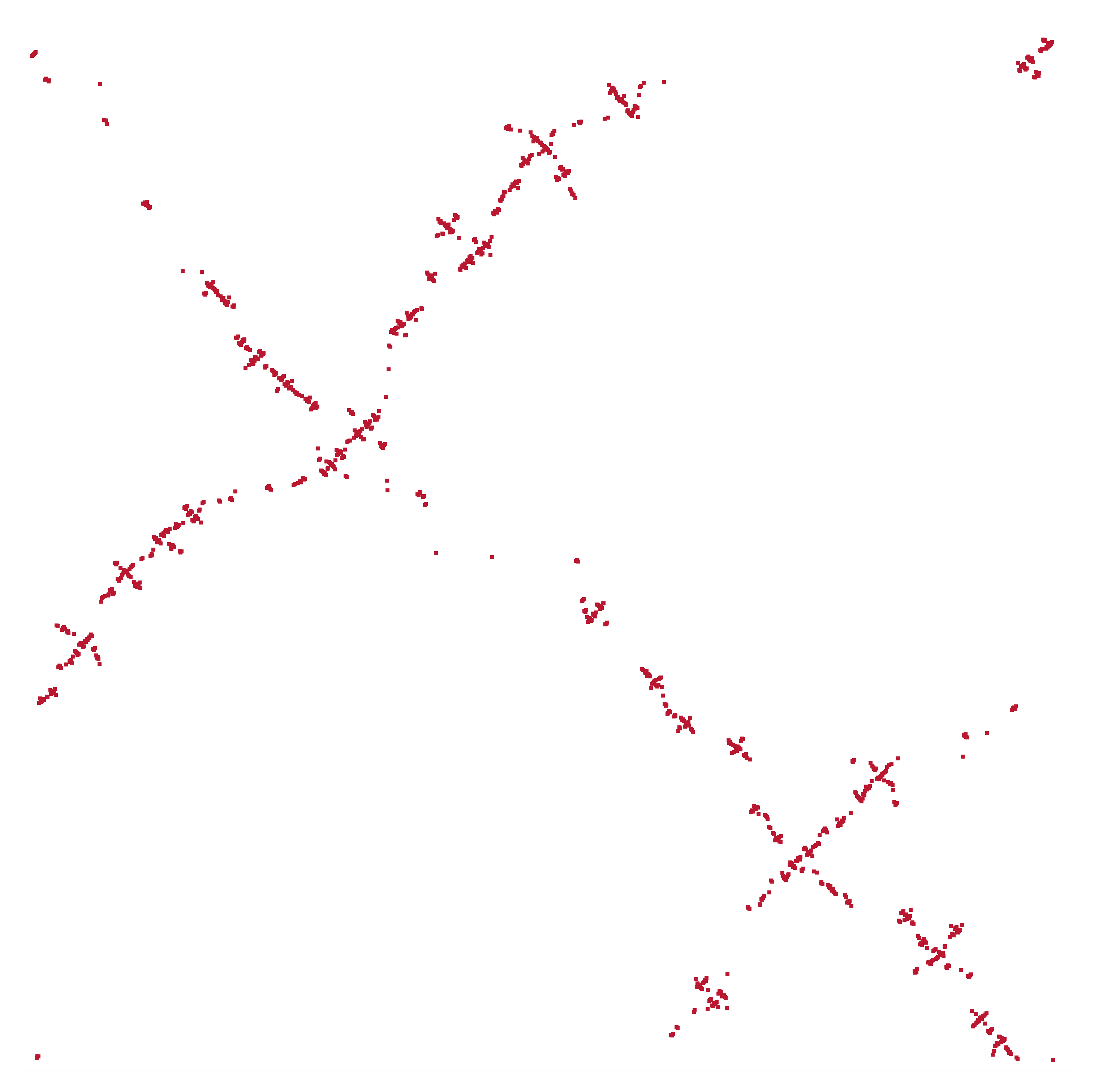}
	\includegraphics[scale=0.1]{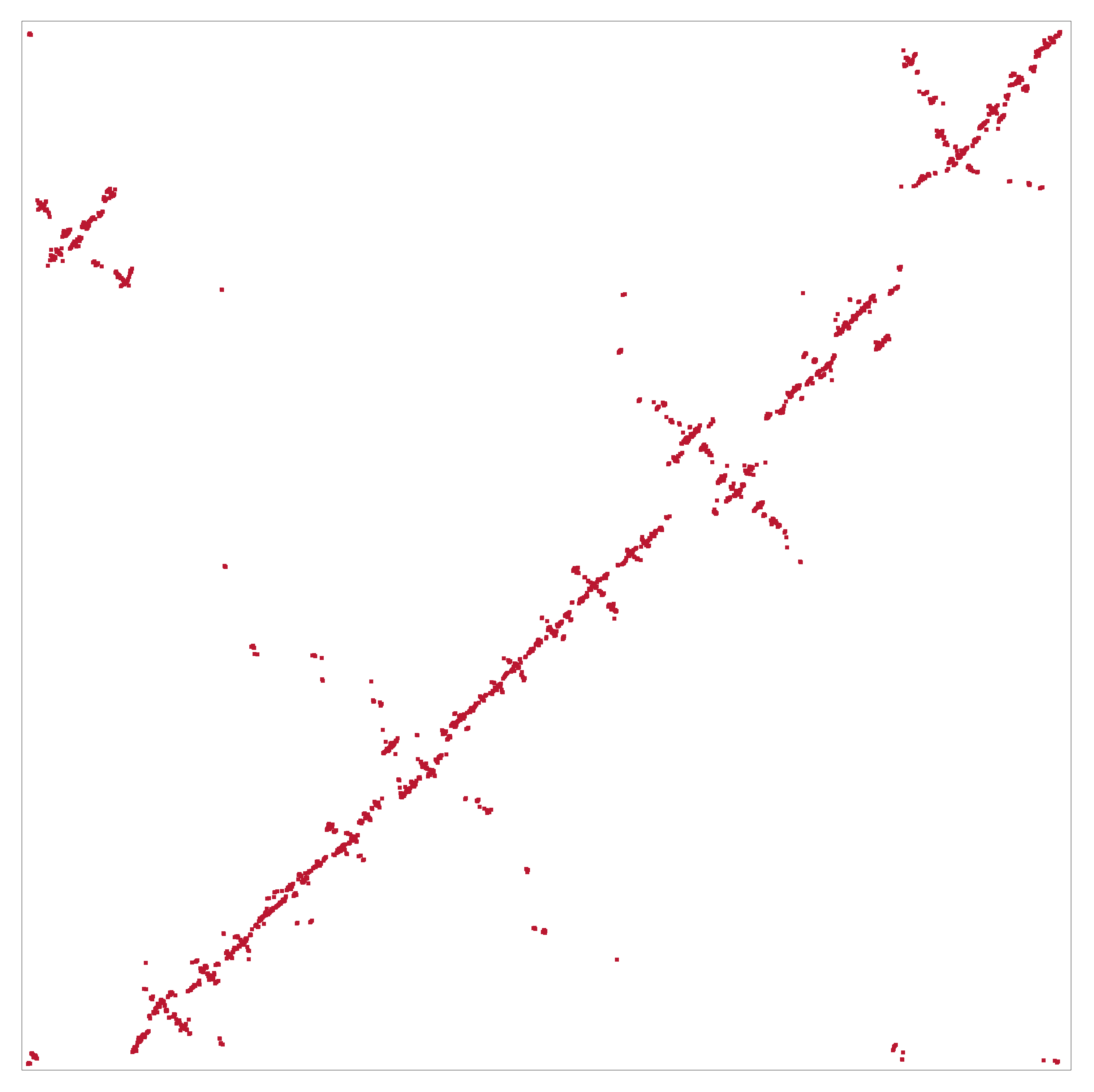}
	\includegraphics[scale=0.091]{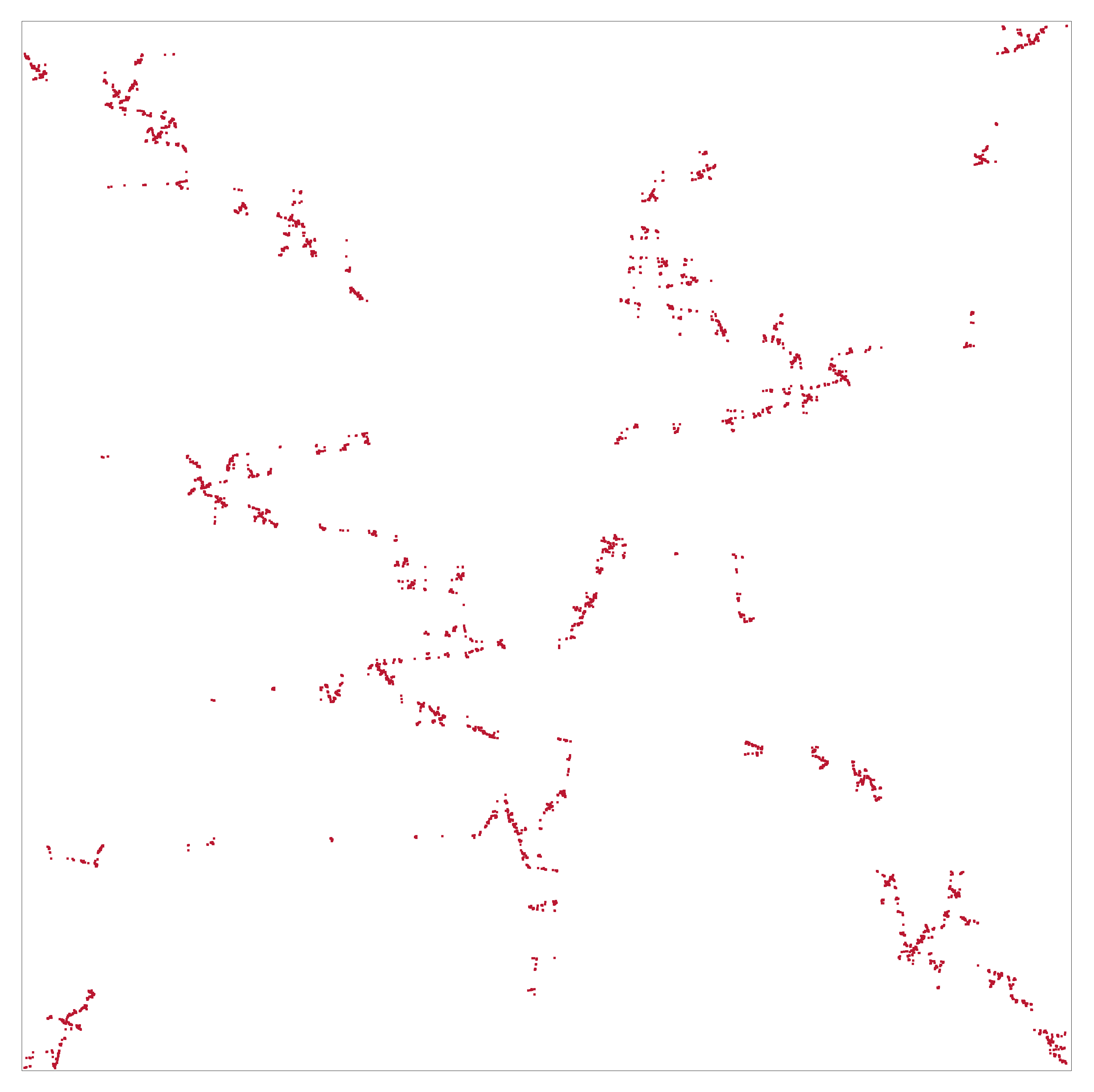}
	\caption{\textbf{Left:} An instance of the Brownian separable permuton. \textbf{Middle:} An instance of the biased Brownian separable permuton.  
		\textbf{Right:} An instance of the Baxter permuton.\break
		More explanations on these limiting permutons will be given in the next sections of the introduction.
		\label{fig:perm_sub}
	}
\end{figure}

\begin{itemize}
	\item In \cite{bassino2018separable}, Bassino, Bouvel, F\'eray, Gerin and Pierrot prove that a sequence of uniform random separable permutations\footnote{Separable permutations are the permutations avoiding the patterns $2413$ and $3142$. They are one of the most studied families of pattern-avoiding permutations, see for instance \cite{MR1620935,MR624050,MR1093199,MR3359905}.} converges to the \emph{Brownian separable permuton} (see the picture on the left-hand side of \cref{fig:perm_sub}). This is the first work where a random fractal limiting permuton appears in the literature.
	
	\item In a second work \cite{bassino2017universal} (see also \cite{MR4115736}) the authors prove that the Brownian separable permuton is \emph{universal}: they consider uniform random permutations in proper substitution-closed classes\footnote{For an introduction to proper substitution-closed classes we refer the reader to \cite[Section 3.2]{borga2021random}.} and show that their limits in the permuton sense are a one-parameter deformation of the Brownian separable permuton, called \emph{biased Brownian separable permuton} (see the picture in the middle of \cref{fig:perm_sub}). 
	
	\item These new universal permutons are later investigated by Maazoun \cite{maazoun}. 
	
	\item In  \cite{bassino2019scaling} the authors investigate permuton limits for permutations in classes having a finite combinatorial specification for the substitution decomposition. The limit depends on the structure of the specification restricted to families with the largest growth rate. When the specification is strongly connected, two cases occur. If the associated system of equations
	is linear, the limiting permuton is a deterministic $X$-shape. Otherwise, the limiting permuton
	is the biased Brownian separable permuton.
	
	\item Finally, in  \cite{borga2020scaling}, Maazoun and the author of the present paper show that the permuton limit of Baxter permutations\footnote{Baxter permutations were introduced by Glen Baxter in 1964 \cite{MR0184217} to study fixed points of commuting functions.
		Baxter permutations are permutations avoiding the vincular patterns $2\underbracket[.5pt][1pt]{41}3$ and $3\underbracket[.5pt][1pt]{14}2$, i.e.\ permutations $\sigma$ such that there are no indices $i < j < k$ such that $\sigma(j+1) < \sigma(i) < \sigma(k) < \sigma(j)$ or $\sigma(j) < \sigma(k) < \sigma(i) < \sigma(j+1)$. These permutations are deeply studied in combinatorics, see for instance \cite{MR0250516,MR491652,MR555815,MR2763051,MR3882946} and references therein.} is a new random fractal limiting permuton, called \emph{the Baxter permuton} (see the picture in the right-hand side of \cref{fig:perm_sub}), that is not included in the biased Brownian separable permuton universality class.
\end{itemize}
The next two sections review the constructions of the biased Brownian separable permuton and the Baxter permuton. These constructions are fundamental to understand later in \cref{sect:feihbweifb} our definition of the new family of limiting permutons mentioned before.

\subsection{The biased Brownian separable permuton}\label{sect:fvuuvf}
We introduce the biased Brownian separable permuton following \cite[Sections 1.3-4]{maazoun}. Consider a one-dimensional Brownian excursion\footnote{Here and throughout the paper we denote random quantities using \textbf{bold} characters.} $(\bm e(t))_{t\in [0,1]}$ on $[0,1]$ and a parameter $p\in[0,1]$. Conditional on $\bm e$,  consider an i.i.d.\ sequence $(\bm s(\ell))_{\ell}\in\{+1,-1\}^{\Z_{>0}}$ indexed by the local
minima\footnote{For the technicalities involved in indexing an i.i.d.\ sequence by
	this random countable set, see \cite[Section 2.2]{maazoun}.} of $\bm e$ and with distribution $\P(\bm s(\ell)=+1)=p=1-\P(\bm s(\ell)=-1)$. We denote by $(\widetilde{\bm e},p)$ the pair $(\bm e,(\bm s(\ell))_{\ell})$. We define the following random relation $\vartriangleleft_{\widetilde{\bm e},p}$:
conditional on $\bm e$, if $x,y\in[0,1]$ and $x<y$ and $\min_{[x,y]}\bm e$ is reached at a unique point which is a strict local minimum $\ell\in[x,y]$ then

\begin{equation}\label{eq:exc_to_perm}
	\begin{cases}
		x\vartriangleleft_{\widetilde{\bm e},p} y, &\quad\text{if}\quad \bm s(\ell)=+1,\\
		y\vartriangleleft_{\widetilde{\bm e},p} x, &\quad\text{if}\quad \bm s(\ell)=-1.\\
	\end{cases}
\end{equation}
Maazoun showed that there exists a random set $\bm A \subset [0,1]^2$ of a.s.\ zero Lebesgue measure, i.e.\ $\P(\Leb(\bm A)=0)=1$, such that for every $x,y\in[0,1]^2\setminus \bm A$ with $x<y$ then $\min_{[x,y]}\bm e$ is reached at a unique point which is a strict local minimum. In particular, the restriction of $\vartriangleleft_{\widetilde{\bm e},p}$ to $[0,1]^2\setminus \bm A$ is a total order.
Setting
\begin{equation}\label{eq:level_function_sep}
	\psi_{\widetilde{\bm e},p}(t)\coloneqq\Leb\left( \big\{x\in[0,1]|x \vartriangleleft_{\widetilde{\bm e},p} t\big\}\right),\quad t\in[0,1],
\end{equation}
then the biased Brownian separable permuton is defined (see \cite[Theorem 1.3]{maazoun}) as the push-forward of the Lebesgue measure on $[0,1]$ via the mapping $(\Id,\psi_{\widetilde{\bm e},p})$, that is
$$\bm \mu^S_p(\cdot)\coloneqq (\Id,\psi_{\widetilde{\bm e},p})_{*}\Leb(\cdot)=\Leb\left(\{t\in[0,1]|(t,\psi_{\widetilde{\bm e},p}(t))\in \cdot \,\}\right).$$
Heuristically, $\psi_{\widetilde{\bm e},p}$ is the \textquotedblleft\emph{continuum permutation}" of the elements in the interval $[0,1]$ induced by the order
$\vartriangleleft_{\widetilde{\bm e},p}$ and $\bm \mu^S_p$ is the diagram of $\psi_{\widetilde{\bm e},p}$.

\subsection{The Baxter permuton} 
We now introduce the Baxter permuton following \cite{borga2020scaling}. 
To do that, we first define the \emph{continuous coalescent-walk process} driven by a two-dimensional Brownian excursion. 

\medskip

We start by recalling that a \emph{two-dimensional Brownian motion} \emph{of correlation} $\rho\in[-1,1]$, denoted $(\conti W_{\rho}(t))_{t\in \R_{\geq 0}}=(\conti X_{\rho}(t),\conti Y_{\rho}(t)))_{t\in \R_{\geq 0}}$, is a continuous two-dimensional Gaussian process such that the components $\conti X_{\rho}$ and $\conti Y_{\rho} $ are standard one-dimensional Brownian motions, and $\mathrm{Cov}(\conti X_{\rho}(t),\conti Y_{\rho}(s)) = \rho \cdot \min\{t,s\}$. We also recall that a  \emph{two-dimensional Brownian excursion} $(\conti E_{\rho}(t))_{t\in [0,1]}$ \emph{of correlation}\footnote{We highlight that we excluded the case $\rho=-1$ in the definition of the two-dimensional Brownian excursion. Indeed, if $\rho=-1$ it is not meaningful to condition a two-dimensional Brownian motion of correlation $\rho=-1$ to stay in the non-negative quadrant.} $\rho\in(-1,1]$ \emph{in the non-negative quadrant} (here simply called a  \emph{two-dimensional Brownian excursion of correlation} $\rho$) is a two-dimensional Brownian motion of correlation $\rho$ conditioned to stay in the non-negative quadrant $\R_{\geq 0}^2$ and to end at the origin, i.e.\ $\conti E_{\rho}(1)=(0,0)$. The latter process was formally constructed in various works (see for instance \cite[Section 3]{MR4010949} and \cite{MR4102254}).

\medskip

Let $(\conti E_{-1/2}(t))_{t\in [0,1]}$be a two-dimensional Brownian excursion of correlation $-1/2$.
Consider the (strong) solutions -- which exist and are unique thanks to Theorem 4.6 in \cite{borga2020scaling}  -- of the following family of stochastic differential equations (SDEs) indexed by $u\in [0,1]$ and driven by $\conti E_{-1/2}(t)= (\conti X_{-1/2}(t),\conti Y_{-1/2}(t))$:
\begin{equation}\label{eq:flow_SDE}
	\begin{cases}
		d\cnz^{(u)}(t) = \idf_{\{\cnz^{(u)}(t)> 0\}} d\conti Y_{-1/2}(t) - \idf_{\{\cnz^{(u)}(t)\leq 0\}} d \conti X_{-1/2}(t),& t\in (u,1),\\
		\cnz^{(u)}(t)=0,&  t\in [0,u].
	\end{cases} 
\end{equation}

\begin{defn}\label{defn:cont_coal_proc}
	The \emph{continuous coalescent-walk process driven by} $\conti E_{-1/2}$ is the collection of stochastic processes $\left\{\cnz^{(u)}\right\}_{u\in[0,1]}$ defined by the SDEs in \cref{eq:flow_SDE}.  
\end{defn}

We can now formally introduce the Baxter permuton (a heuristic explanation is given after \cref{defn:Baxter_perm}). 
Consider the following stochastic process: 
\begin{equation}
	\varphi_{\cnz}(t)
	\coloneqq
	\Leb\left( \big\{x\in[0,t)|\cnz^{(x)}(t)<0\big\} \cup \big\{x\in[t,1]|\cnz^{(t)}(x)\geq0\big\} \right),\quad t\in[0,1],
\end{equation}
where $\Leb(\cdot)$ denotes the one-dimensional Lebesgue measure.

\begin{defn}\label{defn:Baxter_perm}
	The \emph{Baxter permuton} $\bm \mu^B$ is the push-forward of the Lebesgue measure on $[0,1]$ via the mapping $(\Id,\varphi_{\cnz})$, that is,
	$$
	\bm \mu^B(\cdot)\coloneqq(\Id,\varphi_{\cnz})_{*}\Leb (\cdot)= \Leb\left(\{t\in[0,1]|(t,\varphi_{\cnz}(t))\in \cdot \,\}\right).
	$$
\end{defn}
The Baxter permuton $\bm \mu^B$ is a random measure on the unit square $[0,1]^2$ and it has uniform marginals (Lemma 5.5 in \cite{borga2020scaling}), hence it is a permuton.
Informally, as in the case of the biased Brownian separable permuton, $\varphi_{\cnz}$ is the \textquotedblleft\emph{continuum permutation}" of the elements of $[0,1]$ induced by the order
\begin{equation}
    t\preccurlyeq_{\cnz} s \qquad\text{if and only if}\qquad\cnz^{(t)}(s)<0,
\end{equation} 
where $0\leq t<s \leq 1$, and $\bm \mu^B$ is the diagram of $\varphi_{\cnz}$. More details are given in \cref{sect:welldef}.

\medskip

\subsection{A new family of universal permutons: the skew Brownian permuton}\label{sect:feihbweifb}

\emph{Note added in revision:} In this paper, we called the family of permutons introduced in \cref{defn:Baxter_perm_gen} the \emph{skew Brownian permuton}, but we now believe that it would be more appropriate to refer to them as the \emph{skew Brownian permutons}, as done for instance in \cite{borga2022meanders}.

\medskip

In the previous two sections, we introduced the biased Brownian separable permuton and the Baxter permuton. A natural question is to explain what is the connection between the two. The main goal of the present paper is to answer this question by constructing a new family of permutons, called \emph{skew Brownian permuton}, that includes both the biased Brownian separable permuton and the Baxter permuton. This brings two different limiting objects under the same roof, identifying a new larger universality class.

\subsubsection{Definitions and construction}

Let $(\conti E_{\rho}(t))_{t\in [0,1]}$ be a two-dimensional Brownian excursion of correlation $\rho\in(-1,1]$ and  let $q\in[0,1]$ be a further parameter.
Consider the solutions (see \cref{sect:wievviwevf} for a discussion on existence and uniqueness) to the following family of SDEs indexed by $u\in [0,1]$ and driven by $\conti E_{\rho} = (\conti X_{\rho},\conti Y_{\rho})$:
\begin{equation}\label{eq:flow_SDE_gen}
\begin{cases}
d\cnz_{\rho,q}^{(u)}(t) = \idf_{\{\cnz_{\rho,q}^{(u)}(t)> 0\}} d\conti Y_{\rho}(t) - \idf_{\{\cnz_{\rho,q}^{(u)}(t)\leq 0\}} d \conti X_{\rho}(t)+(2q-1)\cdot d\conti L^{\cnz_{\rho,q}^{(u)}}(t),& t\in (u,1)\\
\cnz_{\rho,q}^{(u)}(t)=0,&  t\in [0,u],
\end{cases} 
\end{equation}
where $\conti L^{\cnz_{\rho,q}^{(u)}}(t)$ denotes the symmetric local-time process at zero of the process $\cnz_{\rho,q}^{(u)}$, i.e.\ 
$$\conti L^{\cnz^{(u)}_{\rho,q}}(t)=\lim_{\varepsilon\to 0}\frac{1}{2\varepsilon}\int_u^t\idf_{\{\cnz^{(u)}_{\rho,q}(s)\in[-\varepsilon,\varepsilon]\}}ds.$$
\begin{defn}\label{def:contcoalproc}
	For all $\rho\in(-1,1]$ and $q\in[0,1]$, we call \emph{continuous coalescent-walk process driven by} $(\conti E_{\rho},q)$ the collection of stochastic processes $\cnz_{\rho,q}=\left\{\cnz^{(u)}_{\rho,q}\right\}_{u\in[0,1]}$.
\end{defn}
We then consider the following stochastic process: 
\begin{equation}\label{eq:random_skew_function}
\varphi_{\cnz_{\rho,q}}(t)\coloneqq
\Leb\left( \big\{x\in[0,t)|\cnz_{\rho,q}^{(x)}(t)<0\big\} \cup \big\{x\in[t,1]|\cnz_{\rho,q}^{(t)}(x)\geq0\big\} \right), \quad t\in[0,1].
\end{equation}

\begin{defn}\label{defn:Baxter_perm_gen}
	Fix\footnote{For a discussion on the possible construction of the skew Brownian permuton when $\rho=-1$ we refer the reader to the open problems in \cref{sect:open_problems}.} $\rho\in(-1,1]$ and $q\in[0,1]$. The \emph{skew Brownian permuton} of parameters $\rho, q$, denoted $\bm \mu_{\rho,q}$, is the push-forward of the Lebesgue measure on $[0,1]$ via the mapping $(\Id,\varphi_{\cnz_{\rho,q}})$, that is
	\begin{equation}
		\bm \mu_{\rho,q}(\cdot)\coloneqq(\Id,\varphi_{\cnz_{\rho,q}})_{*}\Leb (\cdot)= \Leb\left(\{t\in[0,1]|(t,\varphi_{\cnz_{\rho,q}}(t))\in \cdot \,\}\right).
	\end{equation} 
\end{defn}

\begin{rem}\label{rem:friv}
	Note that when $\rho=-1/2$ and $q=1/2$ the SDEs in \cref{eq:flow_SDE_gen} are exactly the SDEs considered in \cref{eq:flow_SDE}. Therefore the Baxter permuton $\bm \mu^B(\cdot)$ coincides with the skew Brownian permuton $\bm \mu_{-1/2,1/2}$.
\end{rem}

\begin{rem}
	We highlight that for different values of $u\in[0,1]$, the SDEs in \cref{eq:flow_SDE_gen} are defined using \emph{the same} two-dimensional Brownian excursion $(\conti E_{\rho}(t))_{t\in [0,1]}$.
	Therefore the coupling of $\cnz_{\rho,q}^{(u)}$
	for different values of $u\in[0,1]$ is highly nontrivial.
\end{rem}

To guarantee that the skew Brownian permuton is well-defined for all $\rho\in(-1,1]$ and $q\in[0,1]$, we need  first to discuss existence and uniqueness of solutions to the SDEs in \cref{eq:flow_SDE_gen} and then to check that $\bm \mu_{\rho,q}$ is indeed a permuton, that is, it has uniform marginals. 

\subsubsection{The skew perturbed Tanaka equations and correctness of the definition}\label{sect:wievviwevf}

We now focus on the SDEs in \cref{eq:flow_SDE_gen}, which we call \emph{skew perturbed Tanaka equations}. The terminology is inspired by the names of some SDEs already studied in the literature. Indeed, when $q=1/2$, i.e.\ when the local time term cancels in \cref{eq:flow_SDE_gen}, we obtain the  \emph{perturbed Tanaka equations} studied in\footnote{To be precise the articles \cite{MR3098074,MR3882190} study the perturbed Tanaka equations driven by a correlated Brownian motion instead of a correlated Brownian excursion as in \cref{eq:flow_SDE_gen}. In \cite{MR3098074,MR3882190} it is proved that there exists a unique (strong) solution to the perturbed Tanaka equation driven by a correlated Brownian motion. Then in \cite[Theorem 4.6]{borga2020scaling} pathwise uniqueness and existence of a strong solution to the SDE in \cref{eq:flow_SDE_gen} was deduced from the results in \cite{MR3098074,MR3882190} using absolute continuity arguments.}
\cite{MR3098074,MR3882190}. Since here we are adding a local time term, we adopt the terminology \emph{skew perturbed Tanaka equations} by analogy with the case of skew Brownian motions: A \emph{skew Brownian motion} of parameter $q\in[0,1]$ is a standard one-dimensional Brownian motion where each excursion is flipped independently to the positive side with probability $q$ (see for instance \cite[Theorem 6]{MR2280299}). It is shown in \cite{MR606993} that the unique strong solution to the SDE
\begin{equation}
	\cnz(t)=\conti B (t)+(2q-1)d\conti L^{\cnz}(t),\qquad t\in\R_{\geq 0},
\end{equation}
where $\conti B (t)$ is a standard one-dimensional Brownian motion, is a skew Brownian motion of parameter $q$. 

It is also convenient to distinguish the following two cases:
\begin{itemize}
	\item when $\rho\in(-1,1)$ (and $q\in [0,1]$), we will refer to the SDEs in \cref{eq:flow_SDE_gen} as \emph{skew pure perturbed Tanaka equations};
	\item	when $\rho=1$ (and $q\in [0,1]$), we will refer to the SDEs in \cref{eq:flow_SDE_gen} simply as \emph{skew Tanaka equations}.
\end{itemize}

The skew pure perturbed Tanaka equations have not been investigated so far in the literature. Our first result guarantees existence and uniqueness of a strong solution.

\begin{thm}\label{thm:conj_rho_gen}
	Fix $\rho \in (-1,1)$ and $q\in [0,1]$.
	Pathwise uniqueness and existence of a strong solution to the SDEs in \cref{eq:flow_SDE_gen} hold for all  $u\in(0,1)$.
	
	\smallskip
	
	More precisely, denoting by $\P_{\conti E_\rho}$ the law of $\conti E_\rho$, we consider the sigma-algebra $\mathcal F^{(u)}_t$  generated by $\conti E_\rho(s) - \conti E_\rho(u)$  for $ s\in[u,t]$ and completed by negligible events of $\P_{\conti E_\rho}$.
	For every $u\in(0,1)$, there exists a continuous $\mathcal F^{(u)}_t$-adapted stochastic process $\cnz_{\rho,q}^{(u)}$ on $[u,1)$, such that
		\begin{enumerate}
			\item  the function $(\omega,u)\mapsto \cnz_{\rho,q}^{(u)}$ is jointly measurable. 
			\item For every $u,r\in(0,1)$, $u<r$, $\cnz_{\rho,q}^{(u)}$ satisfies  \cref{eq:flow_SDE_gen} a.s.\ on the interval $[u,r]$.
			\item If $u,r\in(0,1)$, $u<r$, and $\widetilde {\cnz}$ is a continuous $\mathcal F^{(u)}_t$-adapted stochastic process that satisfies \cref{eq:flow_SDE_gen} a.s.\ on the interval $[u,r]$, then $\widetilde  {\cnz} =  \cnz_{\rho,q}^{(u)}$ a.s.\ on $[u,r]$. 
		\end{enumerate}
\end{thm}

\cref{thm:conj_rho_gen}  guarantees that the continuous coalescent-walk process $\cnz_{\rho,q}=\left\{\cnz^{(u)}_{\rho,q}\right\}_{u\in[0,1]}$ introduced in \cref{def:contcoalproc} is well-defined for all $(\rho,q)\in(-1,1)\times[0,1]$ in the following sense: By Tonelli's theorem, we have that for almost every $\omega,$ $\cnz^{(u)}_{\rho,q}$ is a solution for almost every $u$. 
Since the definition in \cref{eq:random_skew_function} does not depend on a negligible subset of $[0,1]$, the latter issue causes no problems in defining the skew Brownian permuton (\cref{defn:Baxter_perm_gen}).

\begin{rem}
	We remark that \cref{thm:conj_rho_gen} is also a fundamental tool to prove permuton convergence towards the skew Brownian permuton for various models of random constrained permutations, as shown in \cite{borga2021strongBaxter}. For more details see for instance the proof of \cite[Theorem 4.4]{borga2021strongBaxter}.
\end{rem}

We highlight that the case $\rho =1$ is excluded in \cref{thm:conj_rho_gen}. Indeed, when $\rho =1$, existence of a strong solution to the SDEs in \cref{eq:flow_SDE_gen} fails (see \cref{sect:skew} below for a precise discussion).  Note that when $\rho =1$, the two-dimensional Brownian excursion $(\conti E_{1}(t))_{t\in [0,1]}$ rewrites as $\conti E_{1}(t)=(\bm e(t),\bm e(t))$, where $(\bm e(t))_{t\in [0,1]}$ is a one-dimensional Brownian excursion on $[0,1]$. Therefore the SDEs in \cref{eq:flow_SDE_gen} take the simplified form
\begin{equation}\label{eq:Tanaka}
	\begin{cases}
		d\cnz^{(u)}(t) = \sgn(\cnz^{(u)}(t)) d \bm e(t)+(2q-1)\cdot d\conti L^{\cnz^{(u)}}(t),& t\in [u,1),\\
		\cnz^{(u)}(t)=0,&  t\in [0,u],
	\end{cases} 
\end{equation}
where $\sgn(x)\coloneqq \idf_{\{x>0\}}-\idf_{\{x\leq0\}}$.

Solutions to the SDEs in \cref{eq:Tanaka} are not measurable functions of the driving process $(\bm e(t))_{t\in [0,1]}$ but can be constructed using some external randomness as follows.
Conditional on $\bm e$,  consider an i.i.d.\ sequence of random variables $(\bm s(\ell))_{\ell}$ indexed by the local
minima\footnote{For the technicalities involved in indexing an i.i.d.\ sequence by
	this random countable set, see again \cite[Section 2.2]{maazoun}.} of $\bm e$ and with distribution $\P(\bm s(\ell)=+1)=q=1-\P(\bm s(\ell)=-1)$.
For $0\leq u \leq t \leq 1$, set $\bm m^{(u)}(t) \coloneqq \inf_{[u,t]} \bm
e$ and $\bm \varepsilon_q^{(u)}(t) \coloneqq \bm s\left(\sup\{r\in[u,t]: \bm e(r) = \bm m^{(u)}(t)\}\right)$. 
Then, define for $0\leq u \leq t \leq 1$,
\begin{equation}\label{eq:def_process2}
	\cnz_q^{(u)}(t) \coloneqq (\bm e(t) - \bm m^{(u)}(t))\cdot\bm \varepsilon_q^{(u)}(t).
\end{equation}

\begin{prop}\label{thm:fvuwevifview2}
	The family $\left\{\cnz_q^{(u)}(t)\right\}_{u\in[0,1]}$ defined in \cref{eq:def_process2} is a family of solutions to the SDEs in \cref{eq:Tanaka}.
\end{prop}

\begin{rem}
	The construction in \cref{eq:def_process2} is inspired by the works of Hajri \cite{MR2835247} and Le Jan and
	Raimond \cite{MR2060298,MR4112725} on \textit{stochastic flow of maps}. We will discuss this connection in \cref{sect:skew}. 
	For a discussion on the uniqueness of the solution in \cref{thm:fvuwevifview2}, we refer the reader to \cref{rem:uniqueness}.
\end{rem}

\cref{thm:conj_rho_gen} and \cref{thm:fvuwevifview2} therefore imply that the continuous coalescent-walk process $\cnz_{\rho,q}=\left\{\cnz^{(u)}_{\rho,q}\right\}_{u\in[0,1]}$ introduced in \cref{def:contcoalproc} is well-defined for all $(\rho,q)\in(-1,1]\times[0,1]$. Checking that the marginals of skew Brownian permuton $\bm \mu_{\rho,q}$ are uniform (as done in \cref{sect:welldef}), we obtain the following result.

\begin{thm}\label{thm:perm_is_ok}
	The skew Brownian permuton $\bm \mu_{\rho,q}$ is well-defined for all $(\rho,q)\in(-1,1]\times[0,1]$. That is,  $\bm \mu_{\rho,q}$ is a permuton for all $(\rho,q)\in(-1,1]\times[0,1]$.
\end{thm}

Simulations of the skew Brownian permuton $\bm \mu_{\rho,q}$ for various values of $(\rho,q)\in(-1,1]\times[0,1]$ can be found in \cref{fig:uyievievbeee}. How these simulations were obtained is explained in \cref{sect:simul}.
\begin{figure}[htbp]
	\begin{minipage}[c]{0.06\textwidth}
		\centering
		\footnotesize$\rho=$\\
		\footnotesize$-0.995$
	\end{minipage}
	\begin{minipage}[c]{0.149\textwidth}
		\centering
		\footnotesize{$q \approx 0$}
		\includegraphics[scale=0.18]{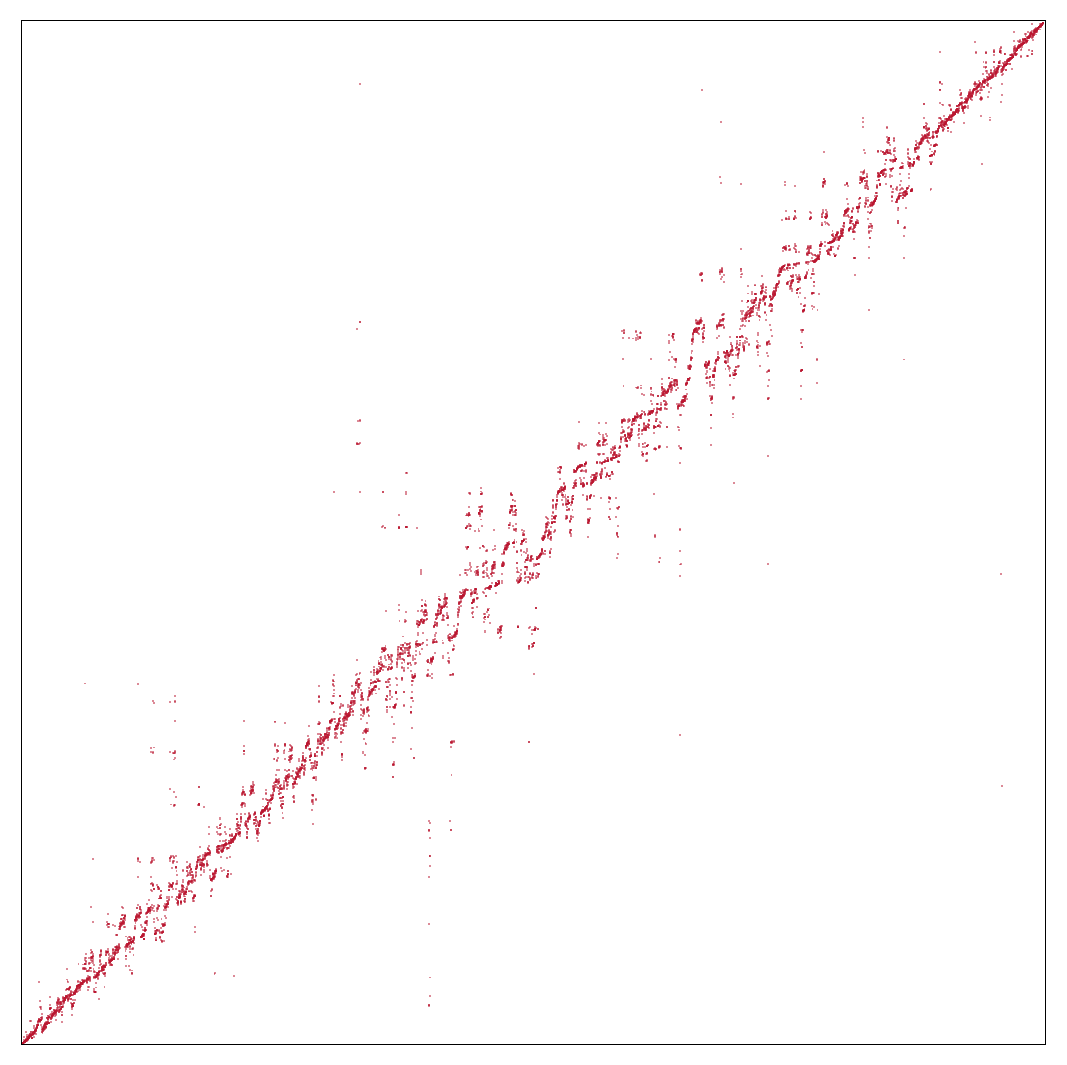}
	\end{minipage}
	\begin{minipage}[c]{0.149\textwidth}
		\centering
		\footnotesize{$q < 0.5$}
		\includegraphics[scale=0.18]{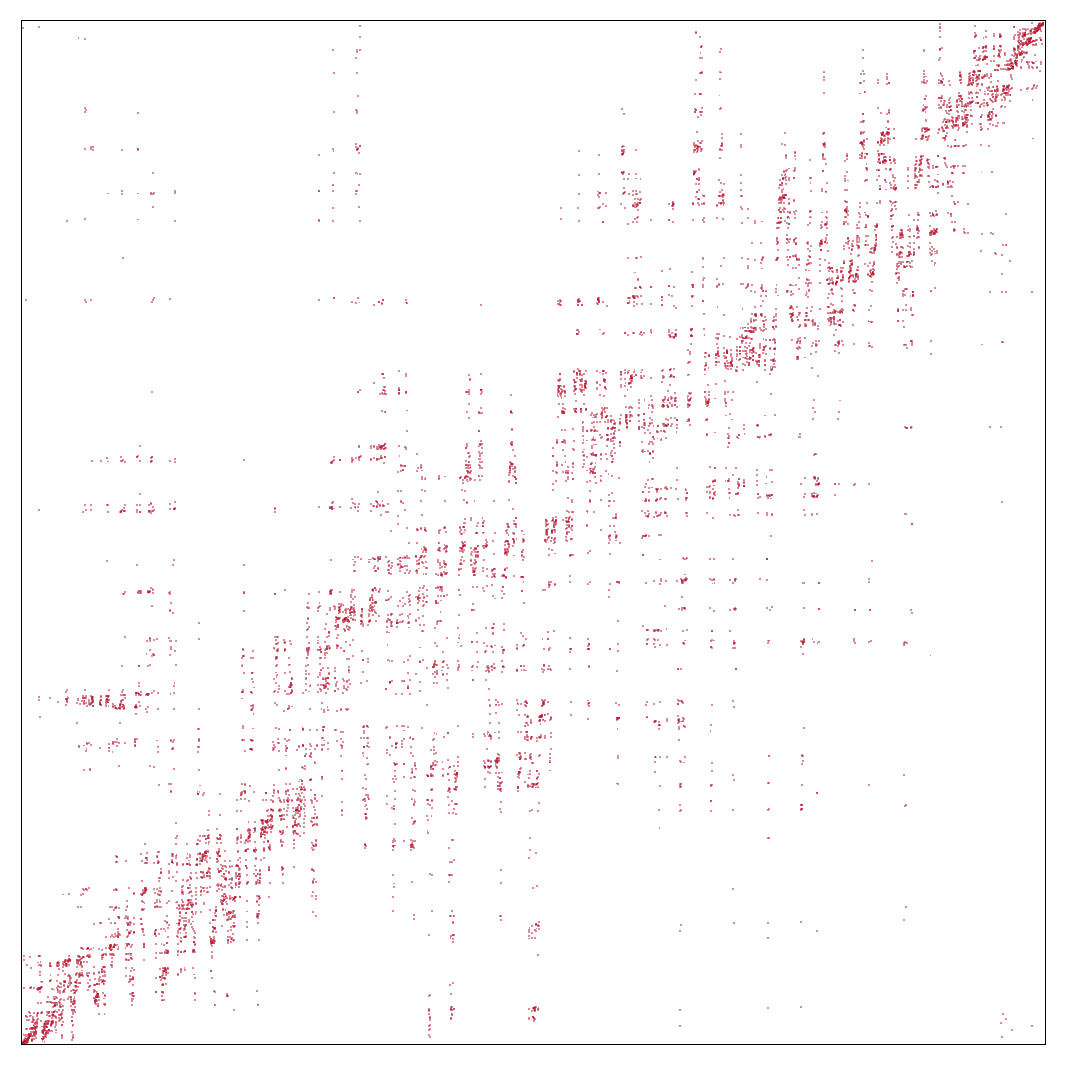}
	\end{minipage}
	\begin{minipage}[c]{0.149\textwidth}
		\centering
		\footnotesize{$q= 0.5$}
		\includegraphics[scale=0.18]{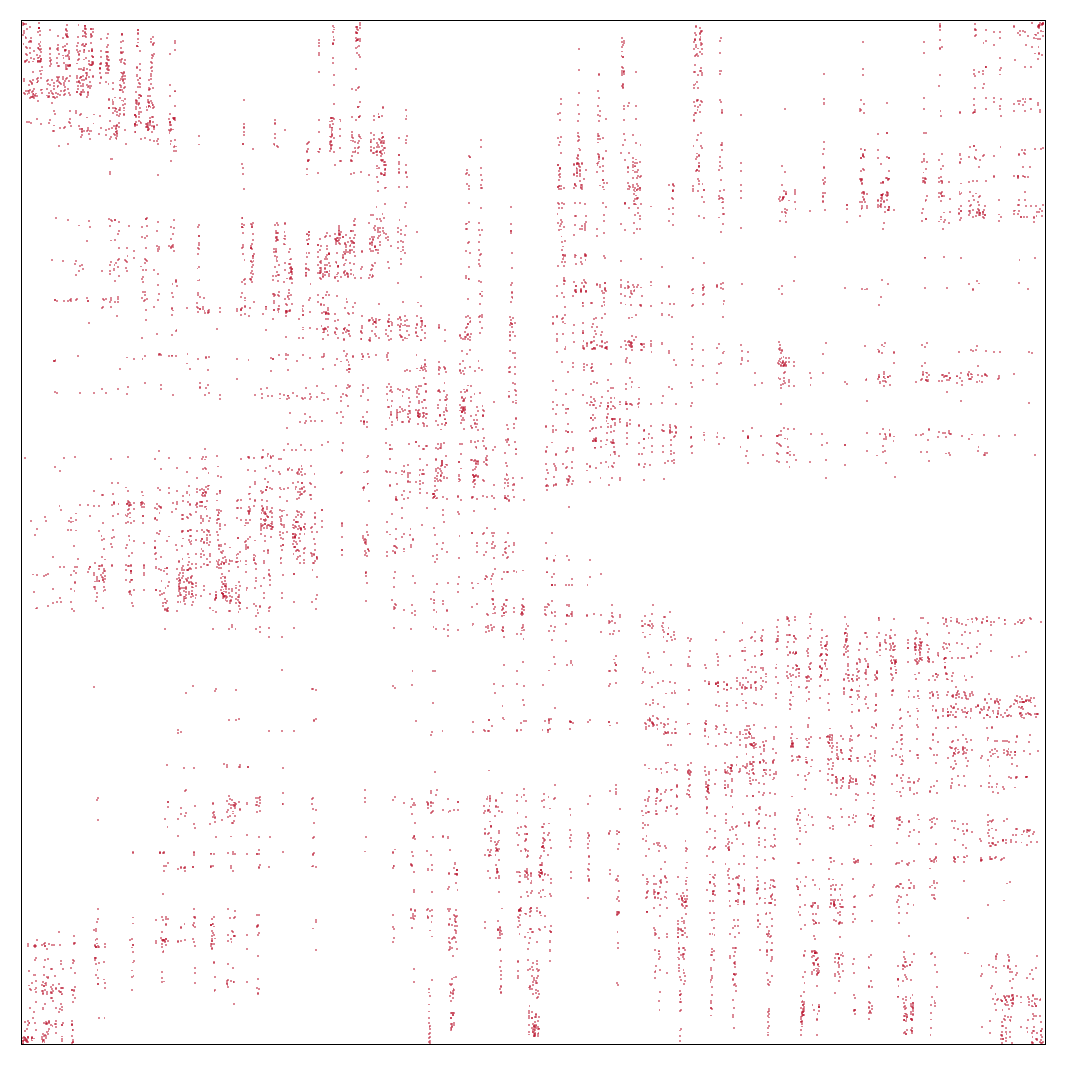}
	\end{minipage}
	\begin{minipage}[c]{0.149\textwidth}
		\centering
		\footnotesize{$q > 0.5$}
		\includegraphics[scale=0.18]{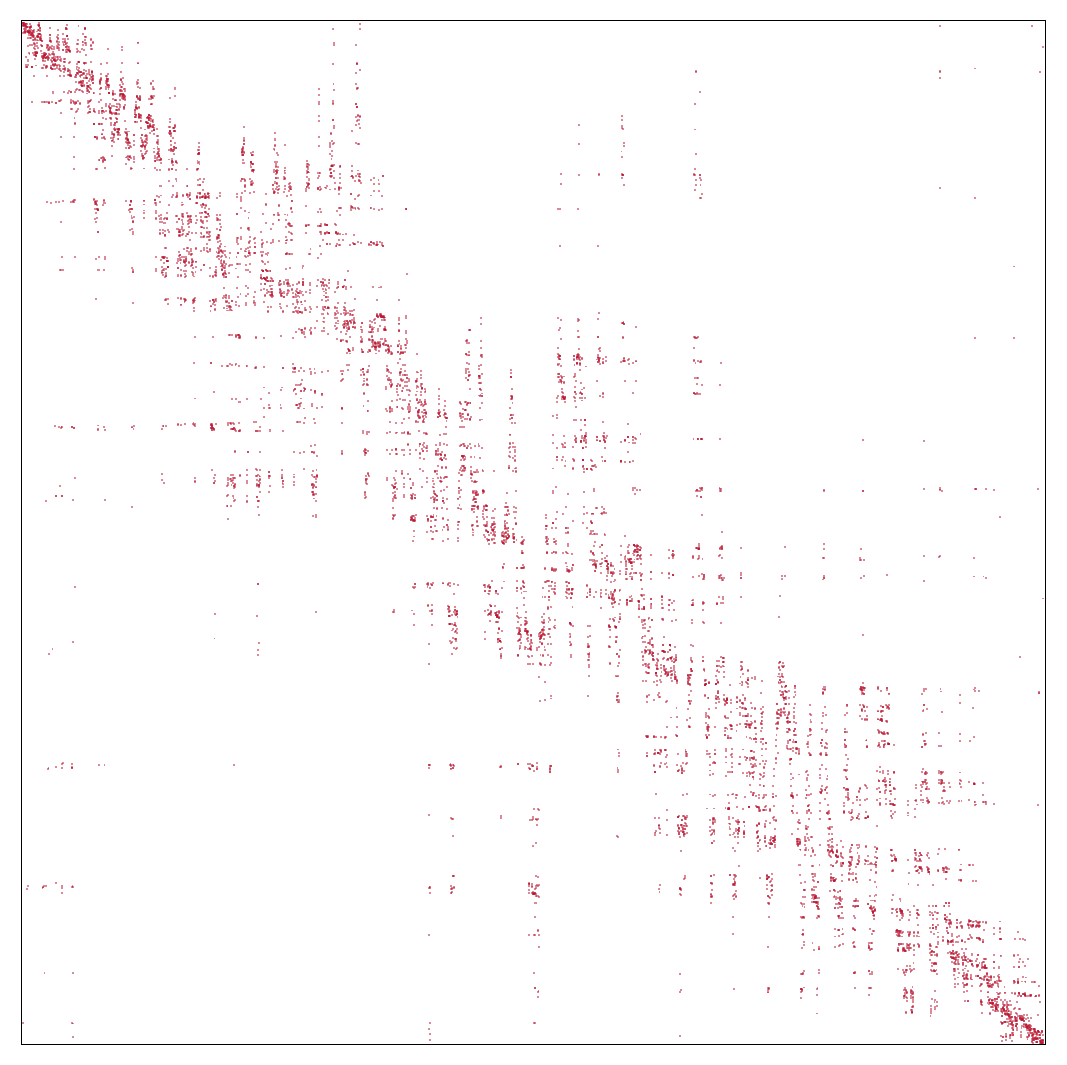}
	\end{minipage}
	\begin{minipage}[c]{0.149\textwidth}
		\centering
		\footnotesize{$q\approx 1$}
		\includegraphics[scale=0.18]{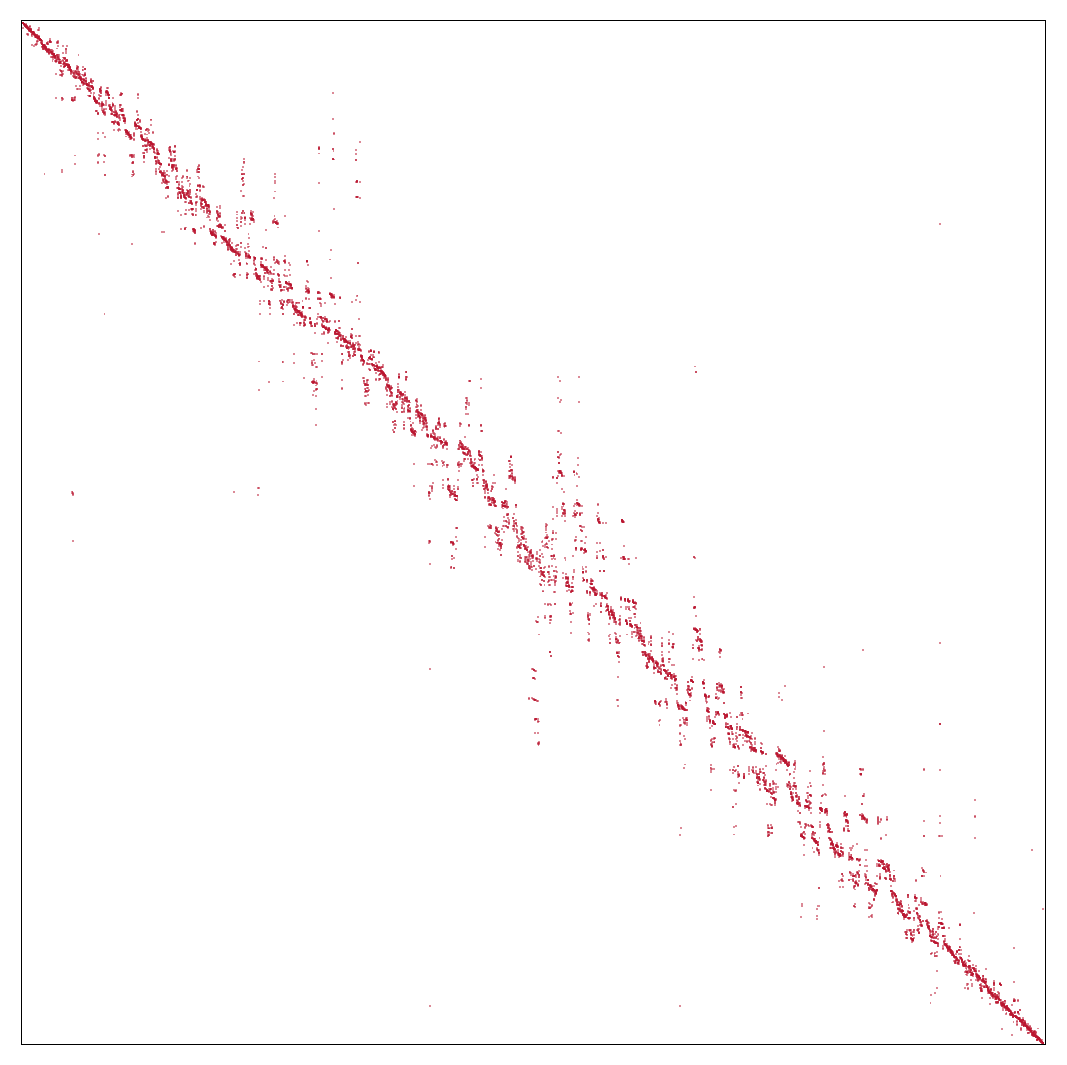}
	\end{minipage}
	\hspace{0.1 cm}
	\begin{minipage}[c]{0.149\textwidth}
		\centering
		\footnotesize{$\conti E_{\rho} = (\conti X_{\rho},\conti Y_{\rho})$}
		\includegraphics[scale=0.225]{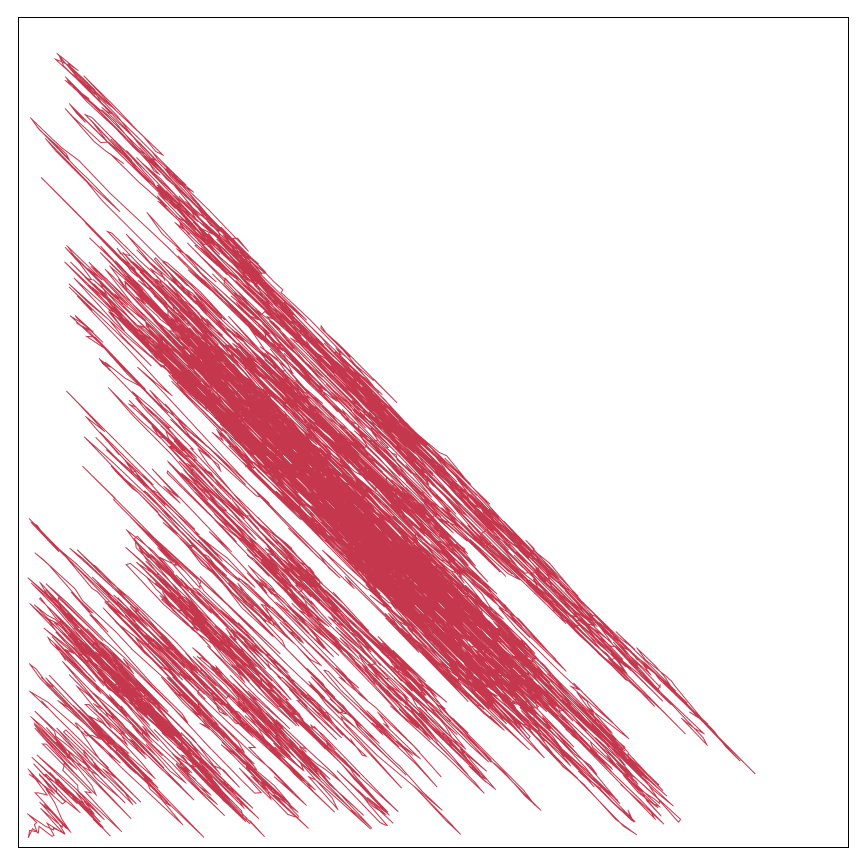}
	\end{minipage}
	
	\begin{minipage}[c]{0.06\textwidth}
		\centering
		\footnotesize$\rho=$\\
		\footnotesize$-0.5$
	\end{minipage}
	\begin{minipage}[c]{0.149\textwidth}
		\centering
		\includegraphics[scale=0.18]{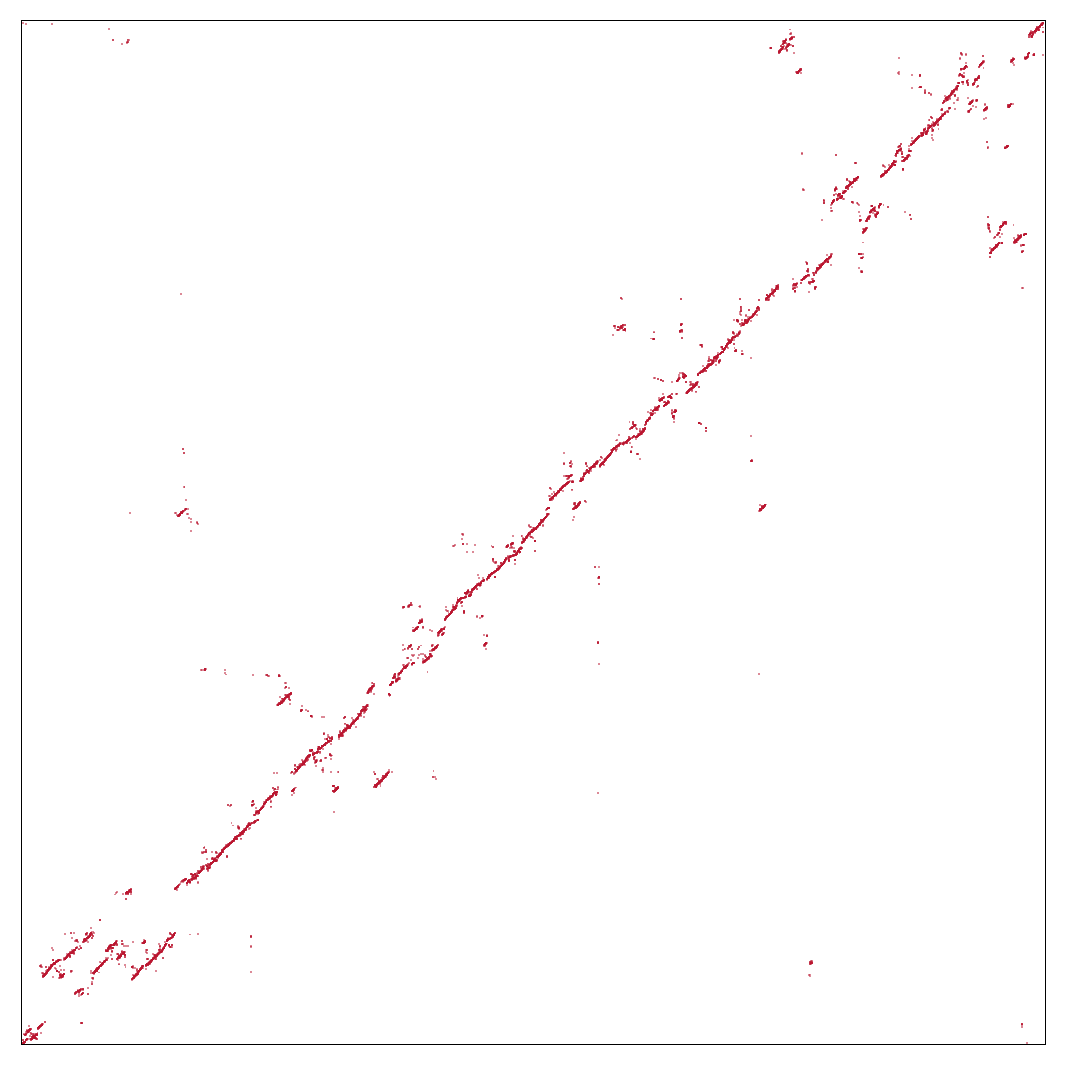}
	\end{minipage}
	\begin{minipage}[c]{0.149\textwidth}
		\centering
		\includegraphics[scale=0.18]{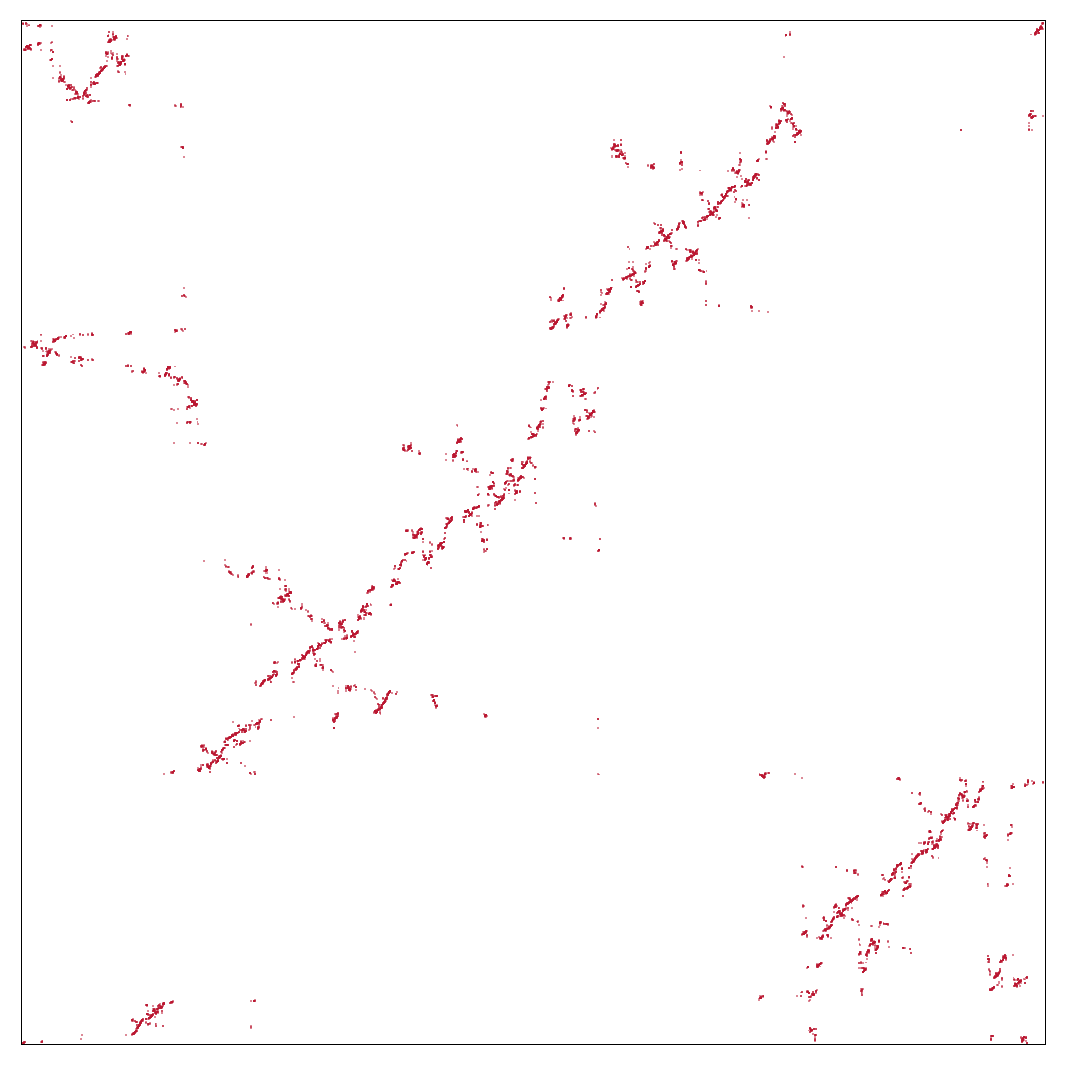}
	\end{minipage}
	\begin{minipage}[c]{0.149\textwidth}
		\centering
		\includegraphics[scale=0.18]{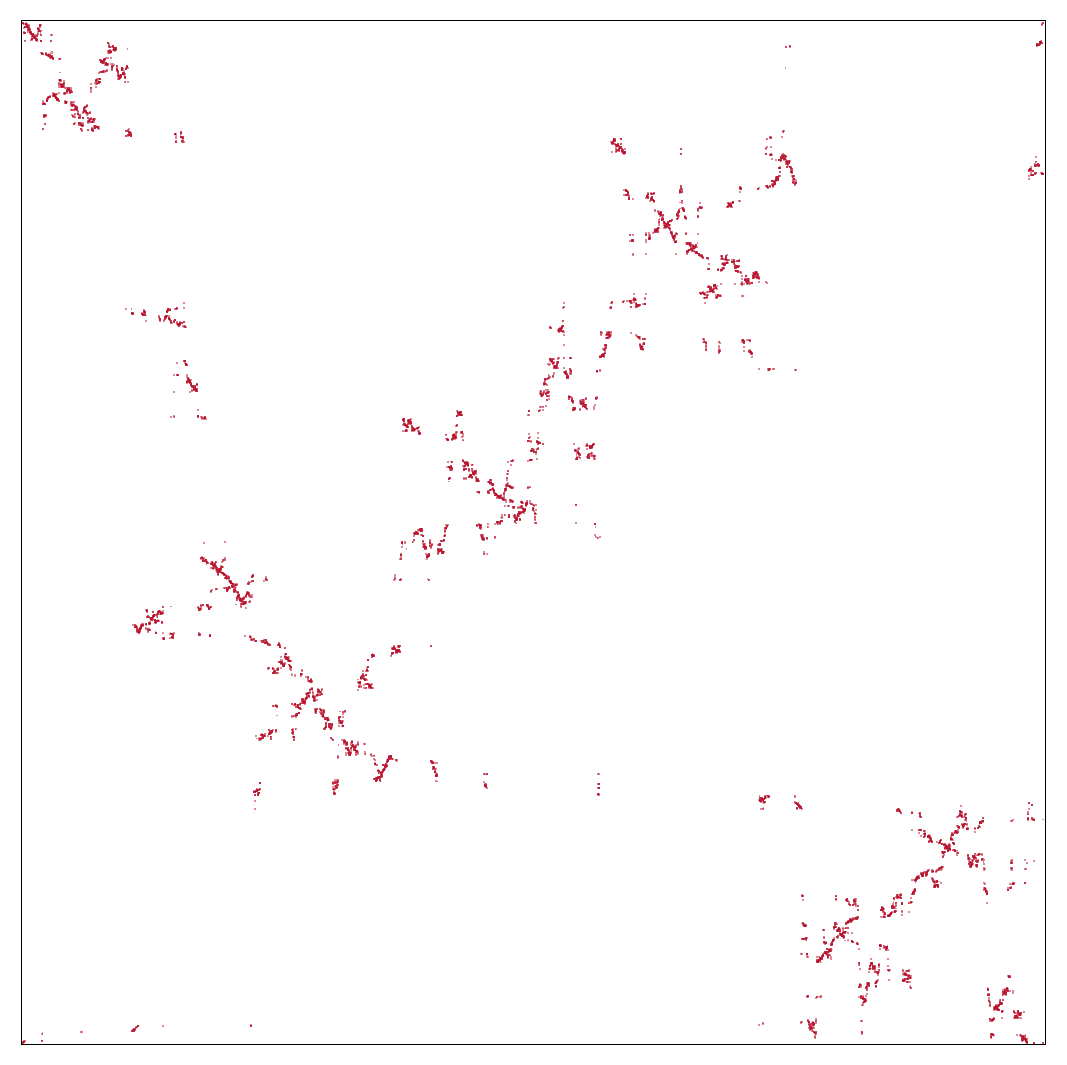}
	\end{minipage}
	\begin{minipage}[c]{0.149\textwidth}
		\centering
		\includegraphics[scale=0.18]{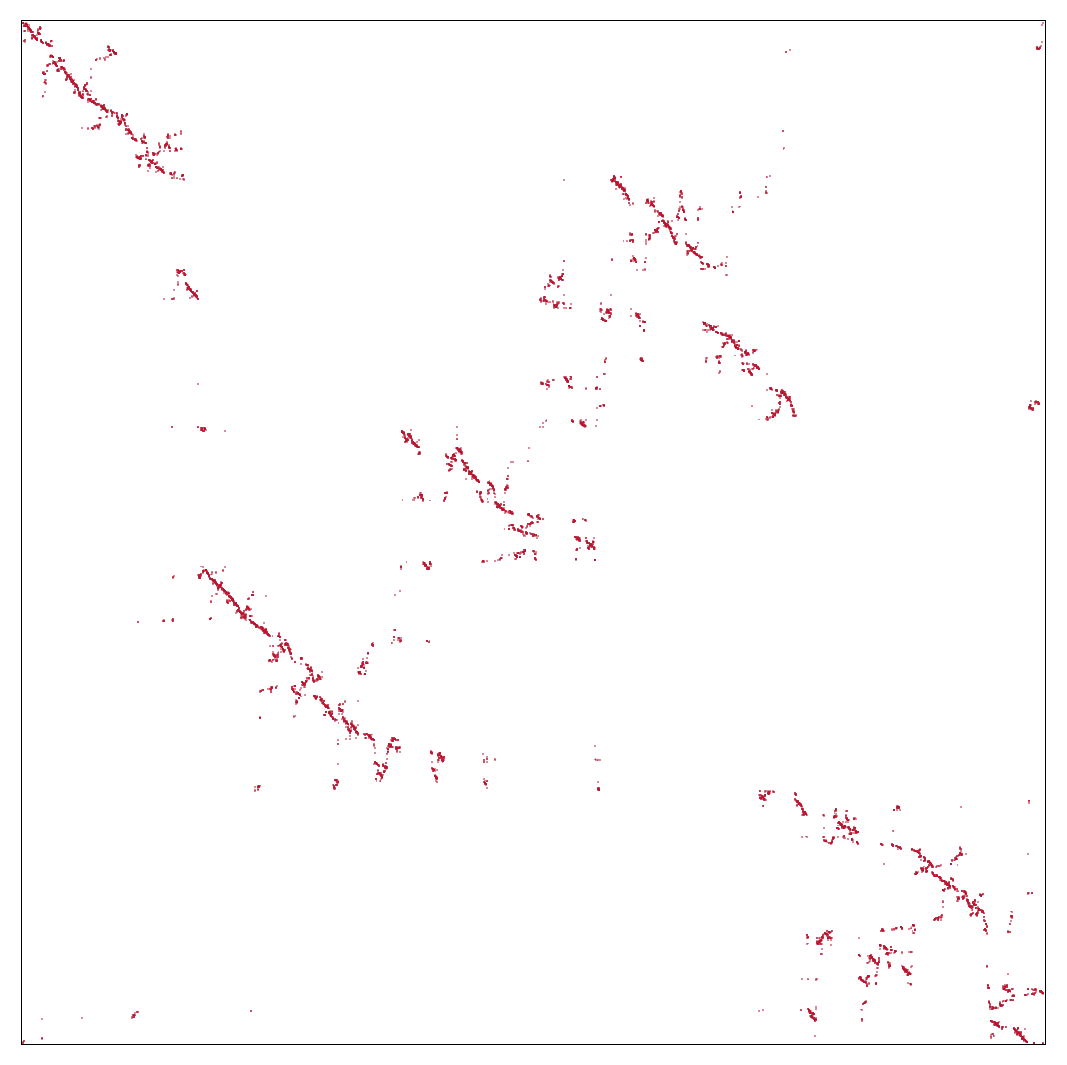}
	\end{minipage}
	\begin{minipage}[c]{0.149\textwidth}
		\centering
		\includegraphics[scale=0.18]{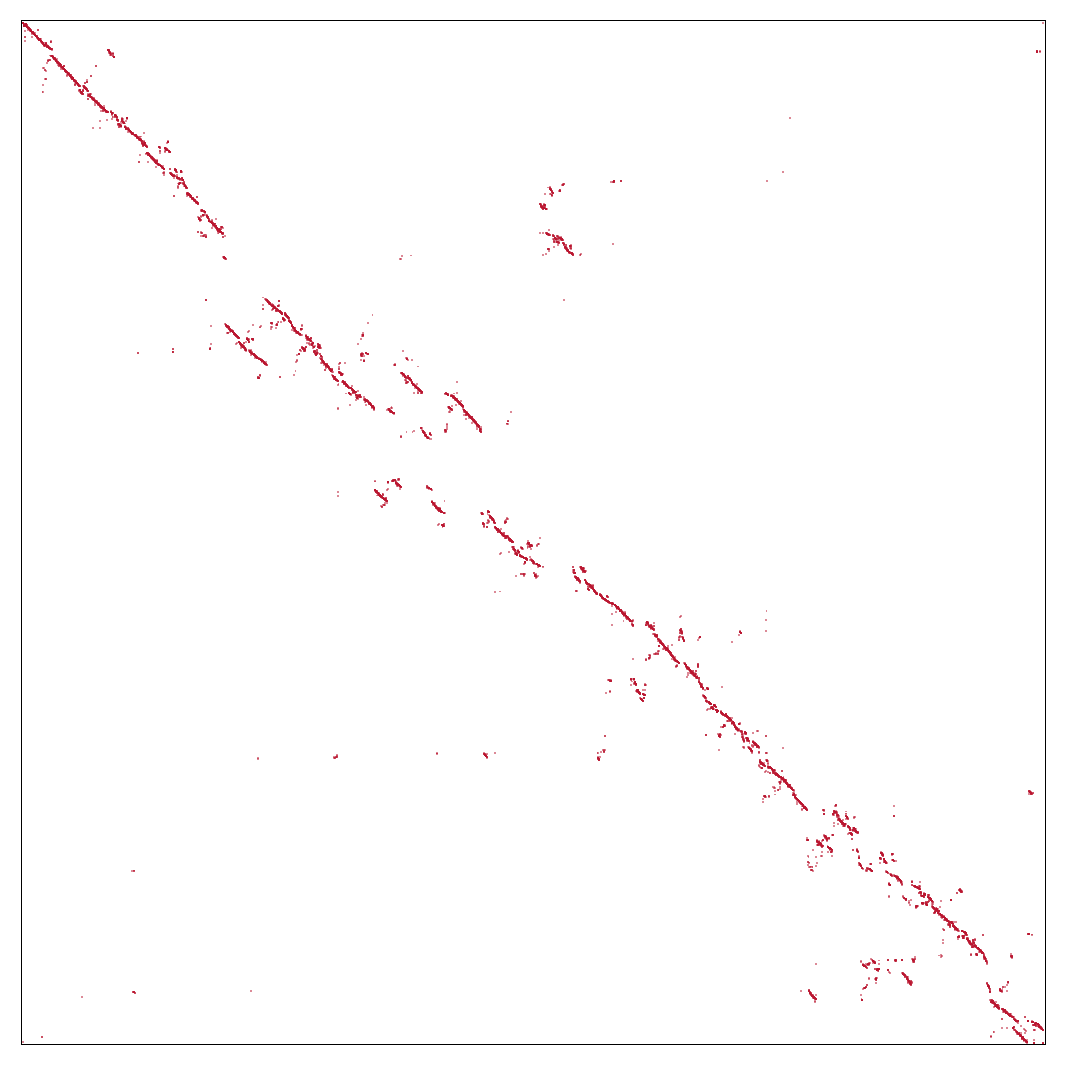}
	\end{minipage}
	\hspace{0.1 cm}
	\begin{minipage}[c]{0.149\textwidth}
		\centering
		\includegraphics[scale=0.225]{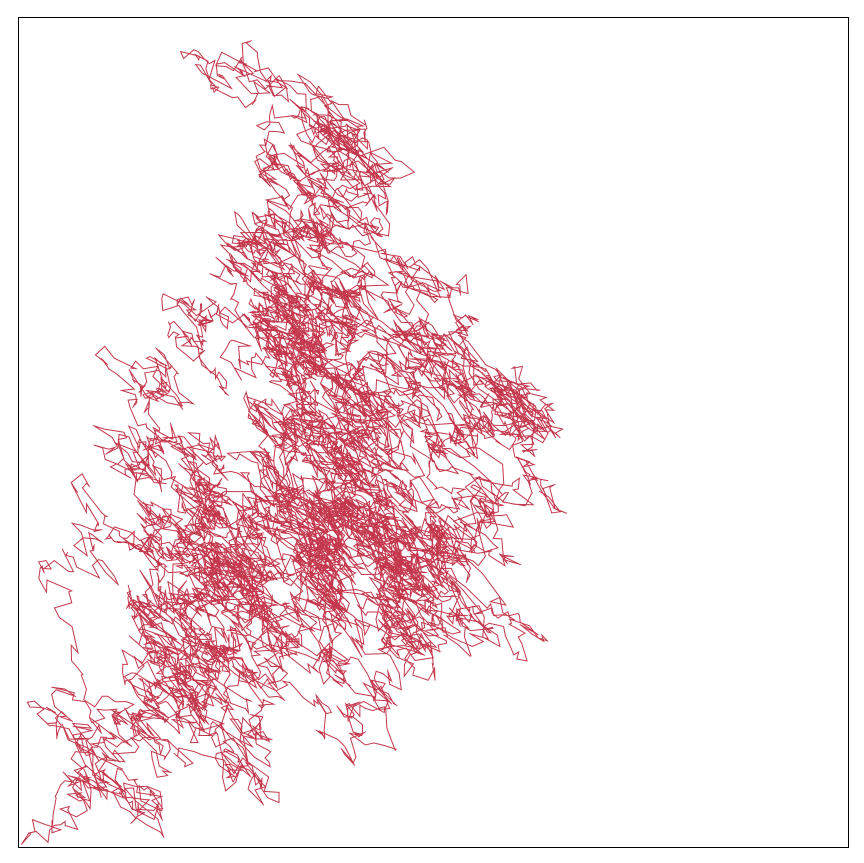}
	\end{minipage}
	
	\begin{minipage}[c]{0.06\textwidth}
		\centering
		\footnotesize$\rho=0$
	\end{minipage}
	\begin{minipage}[c]{0.149\textwidth}
		\centering
		\includegraphics[scale=0.18]{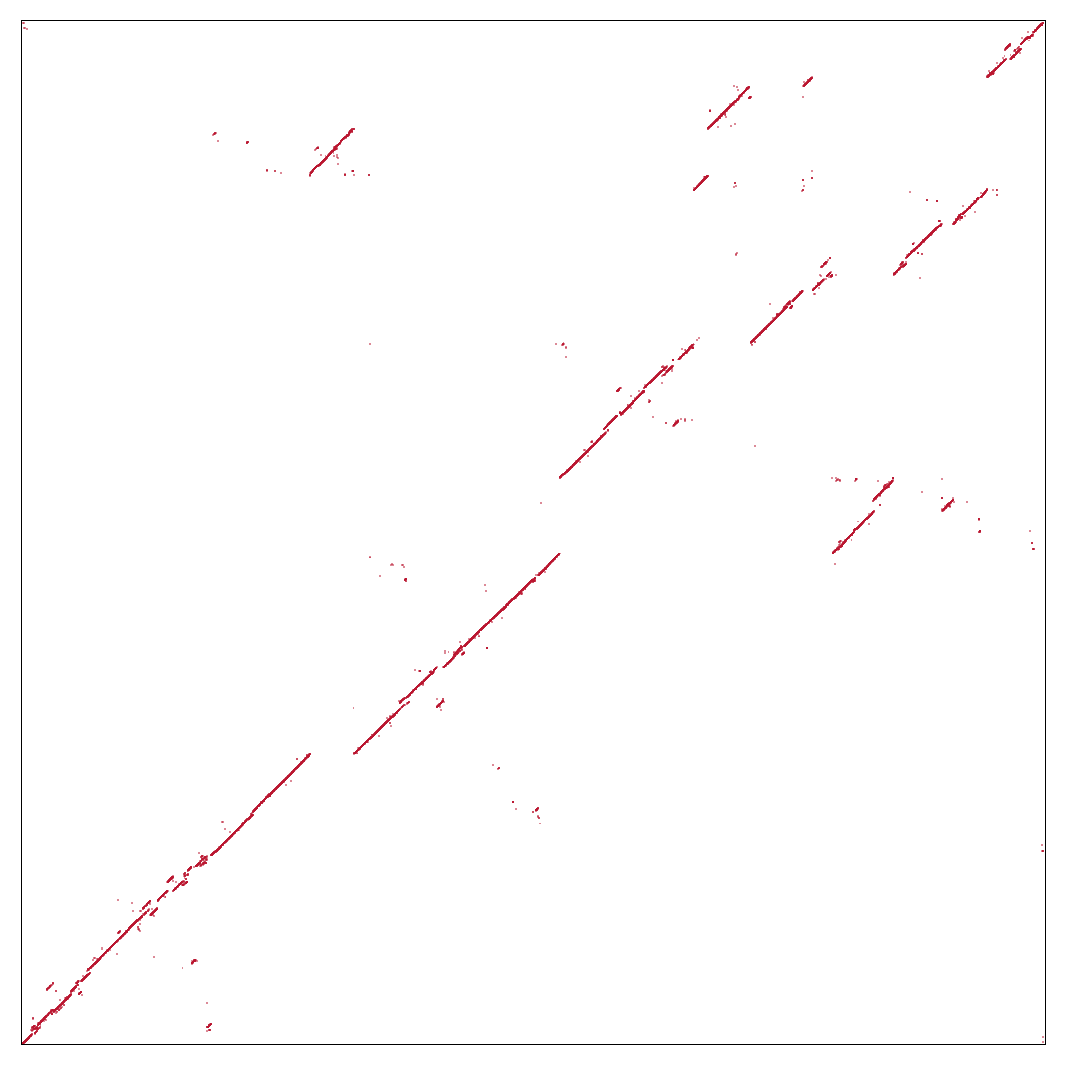}
	\end{minipage}
	\begin{minipage}[c]{0.149\textwidth}
		\centering
		\includegraphics[scale=0.18]{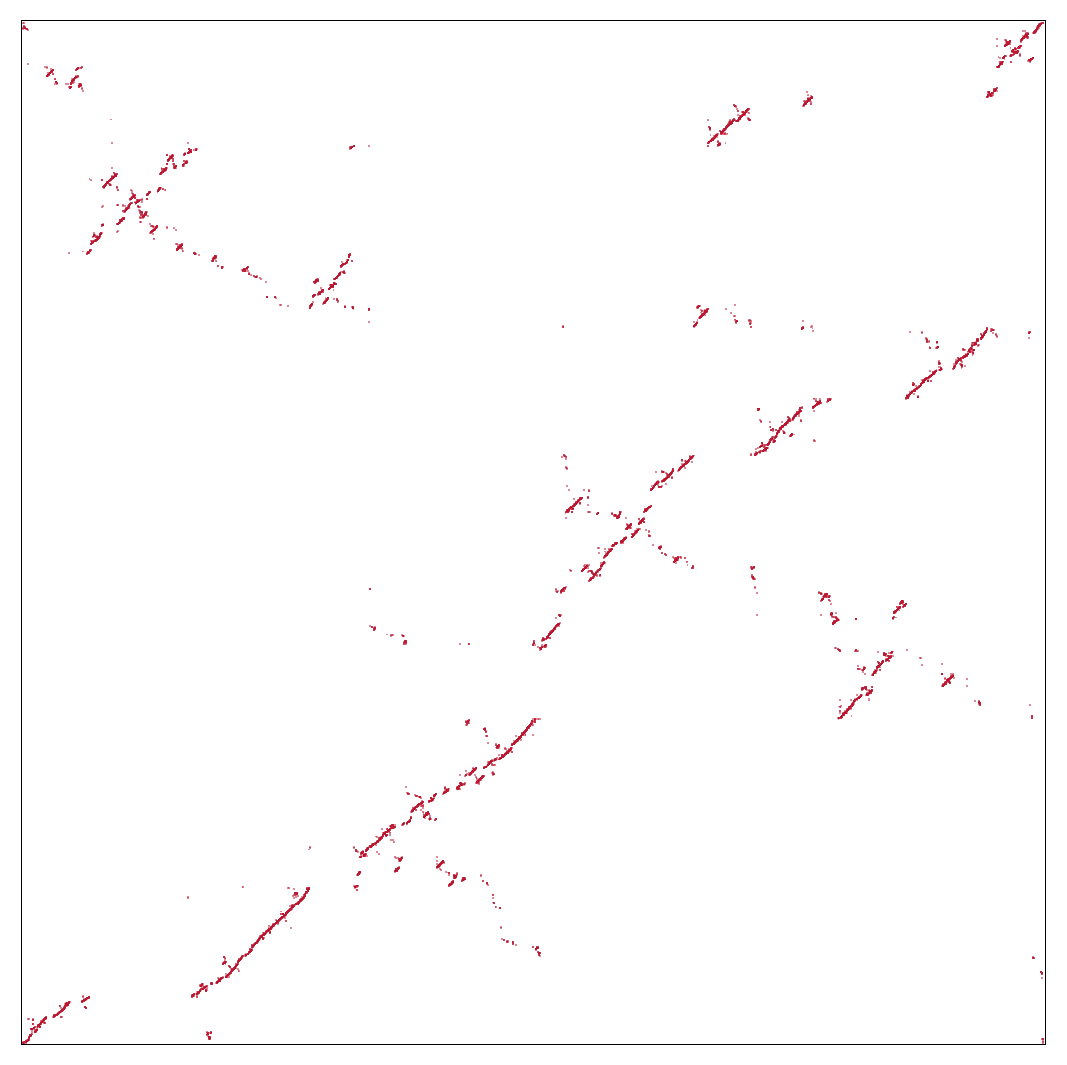}
	\end{minipage}
	\begin{minipage}[c]{0.149\textwidth}
		\centering
		\includegraphics[scale=0.18]{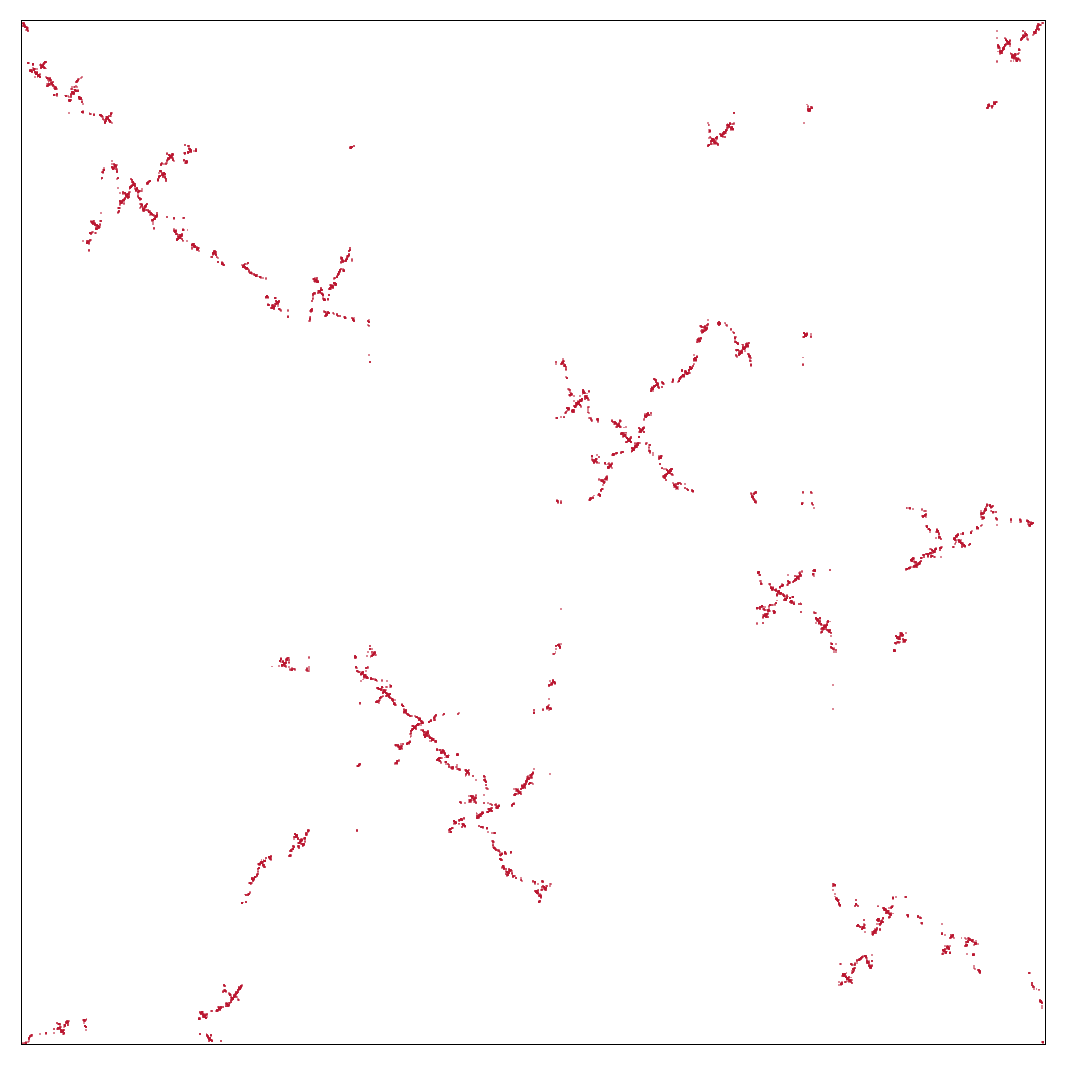}
	\end{minipage}
	\begin{minipage}[c]{0.149\textwidth}
		\centering
		\includegraphics[scale=0.18]{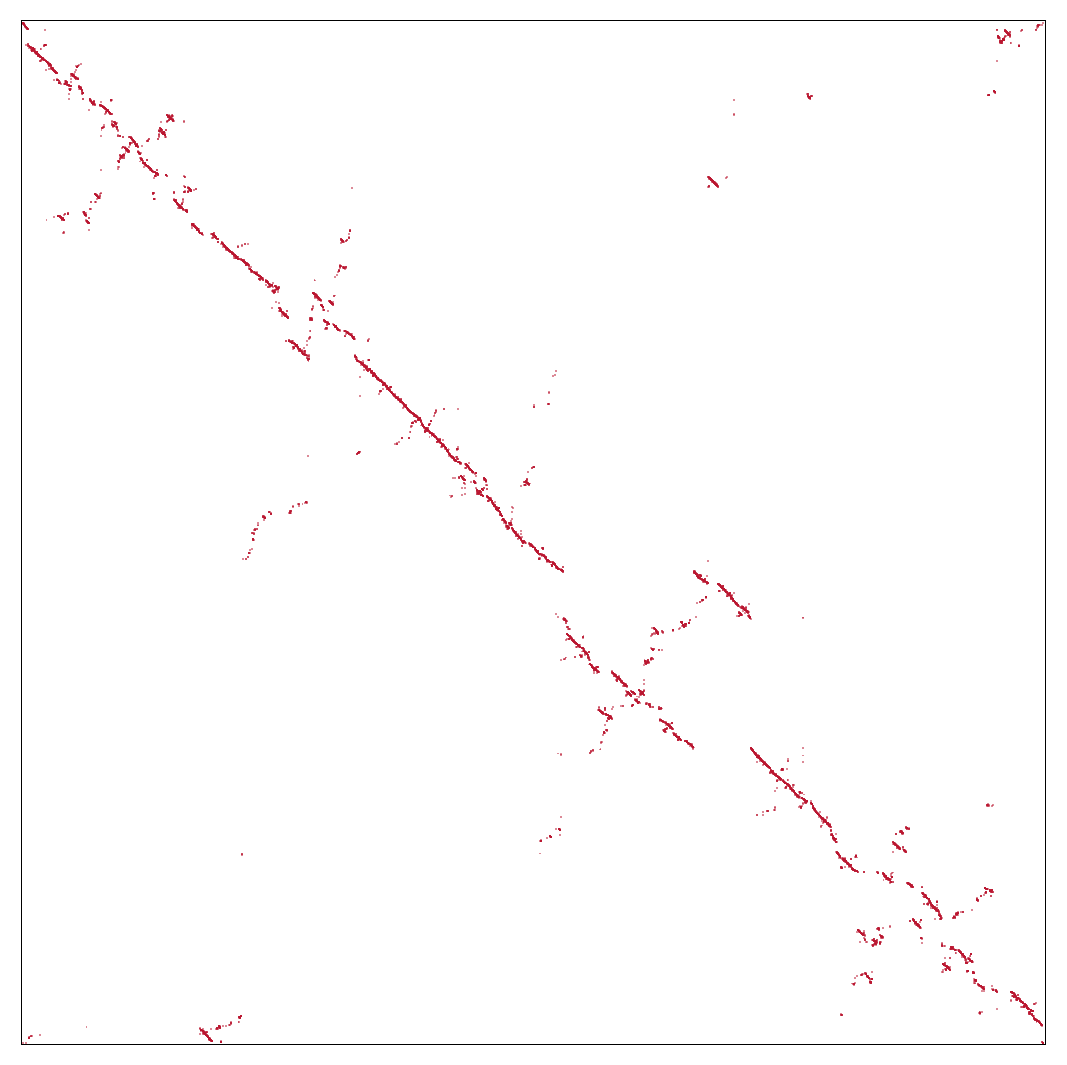}
	\end{minipage}
	\begin{minipage}[c]{0.149\textwidth}
		\centering
		\includegraphics[scale=0.18]{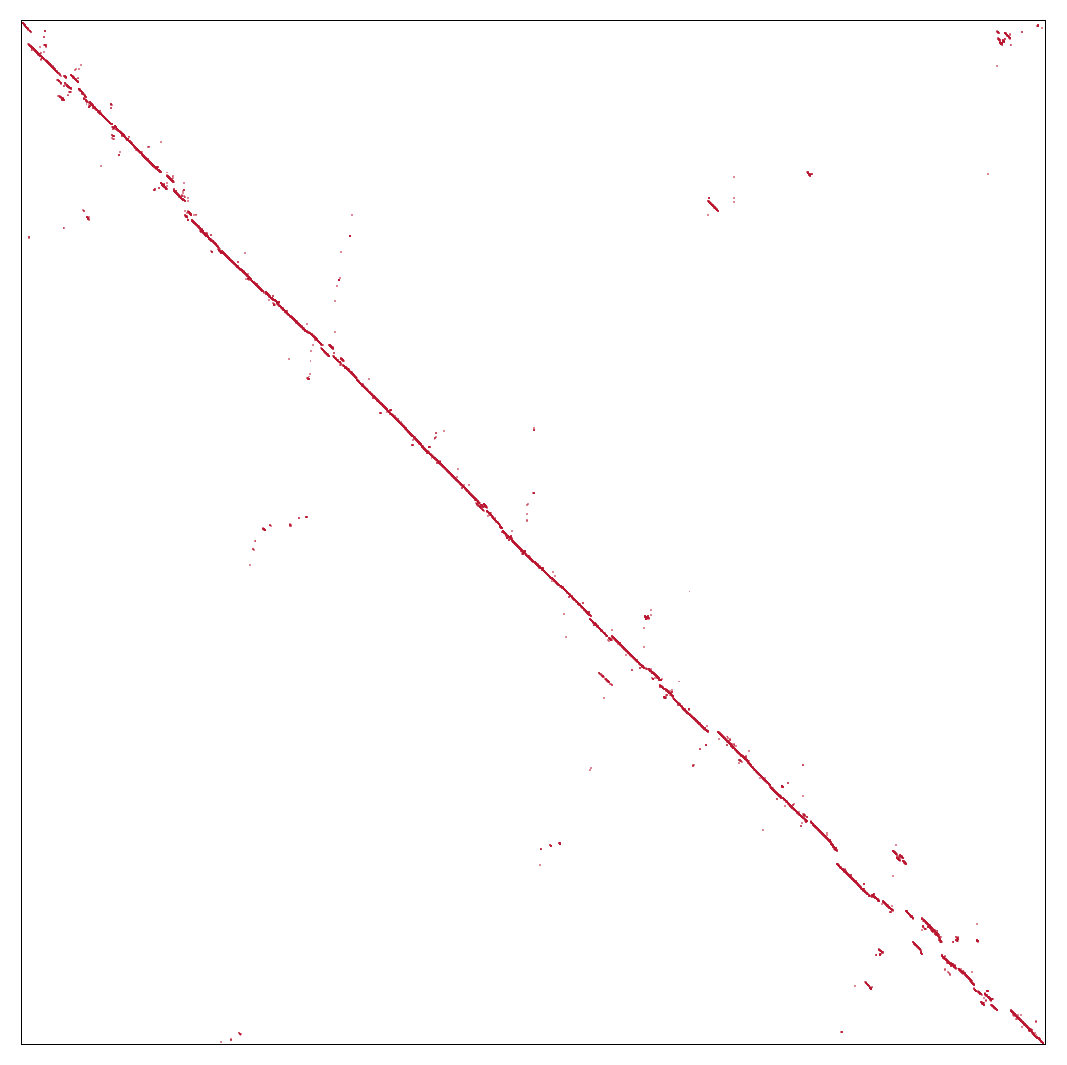}
	\end{minipage}
	\hspace{0.1 cm}
	\begin{minipage}[c]{0.149\textwidth}
		\centering
		\includegraphics[scale=0.225]{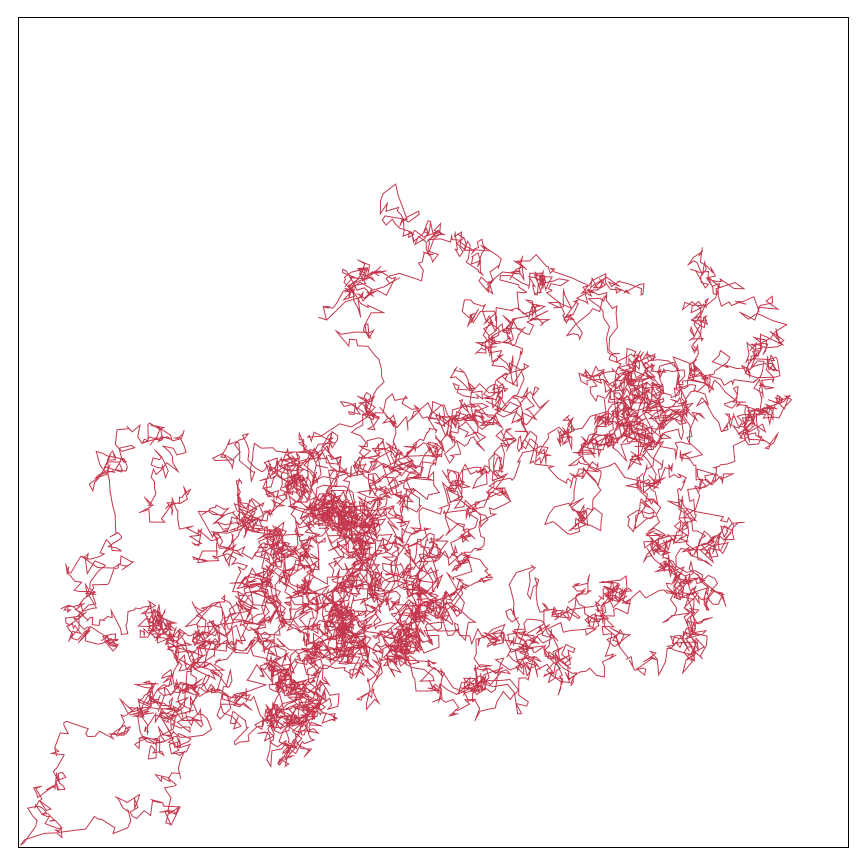}
	\end{minipage}
	
	\begin{minipage}[c]{0.06\textwidth}
		\centering
		\footnotesize$\rho=$\\
		\footnotesize$0.5$
	\end{minipage}
	\begin{minipage}[c]{0.149\textwidth}
		\centering
		\includegraphics[scale=0.18]{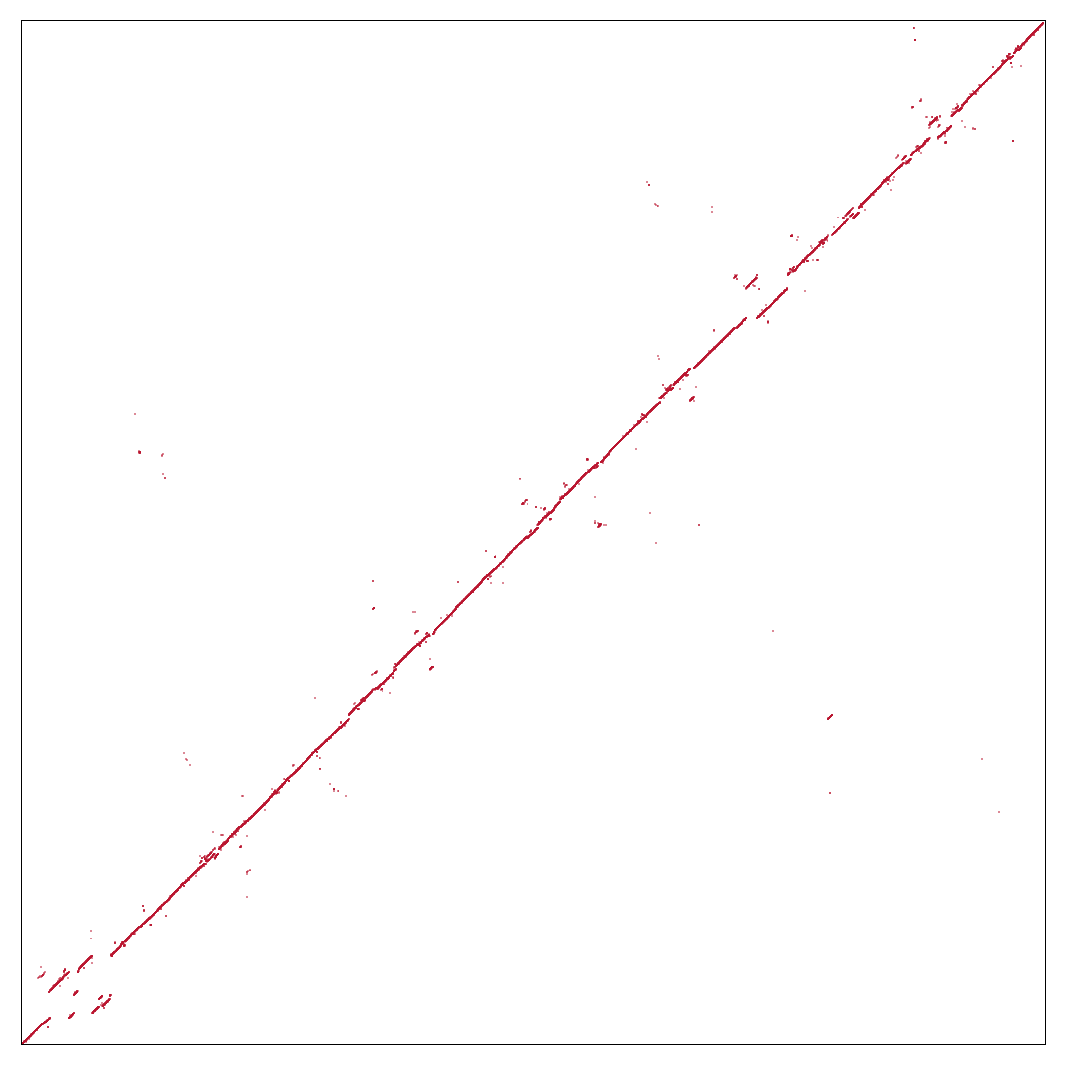}
	\end{minipage}
	\begin{minipage}[c]{0.149\textwidth}
		\centering
		\includegraphics[scale=0.18]{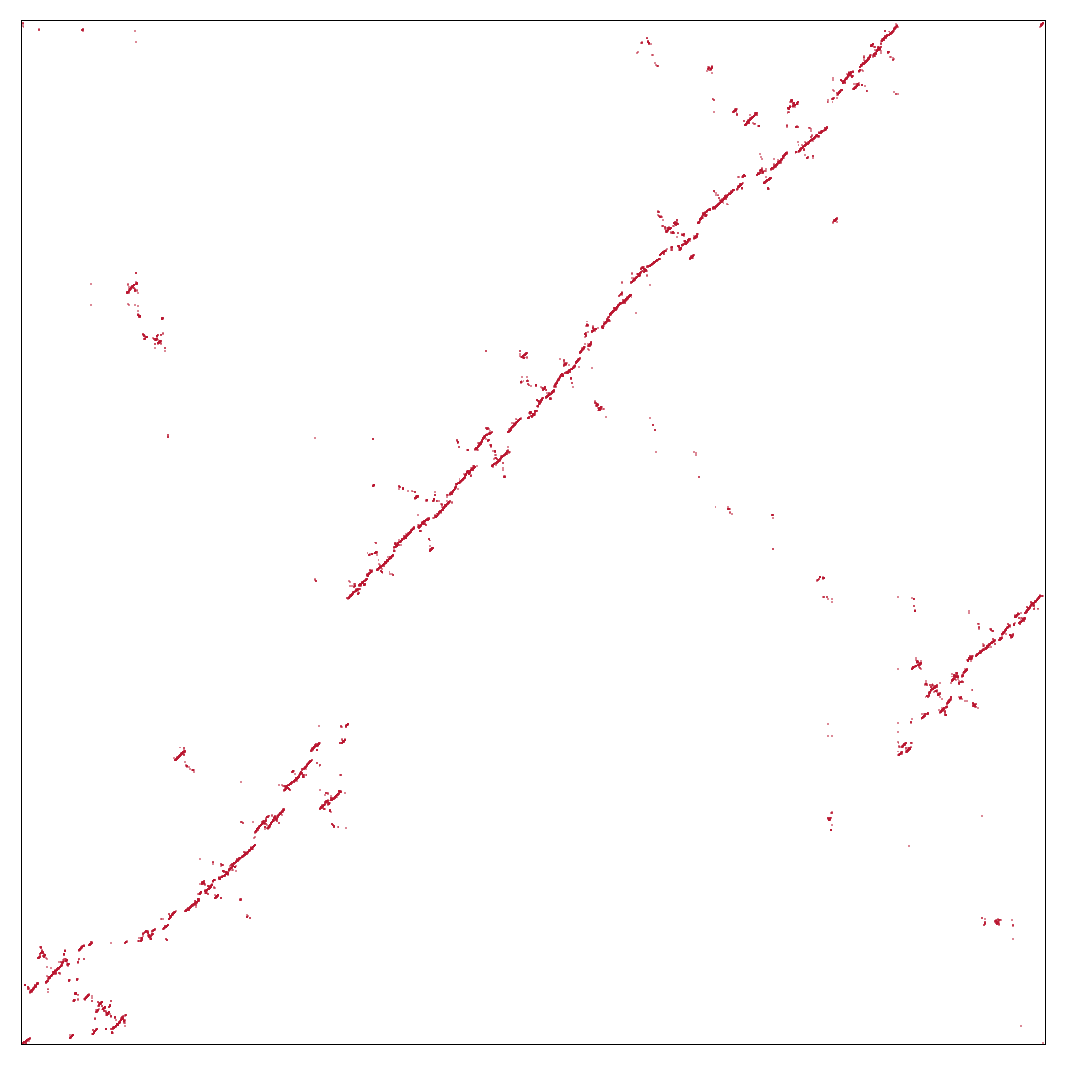}
	\end{minipage}
	\begin{minipage}[c]{0.149\textwidth}
		\centering
		\includegraphics[scale=0.18]{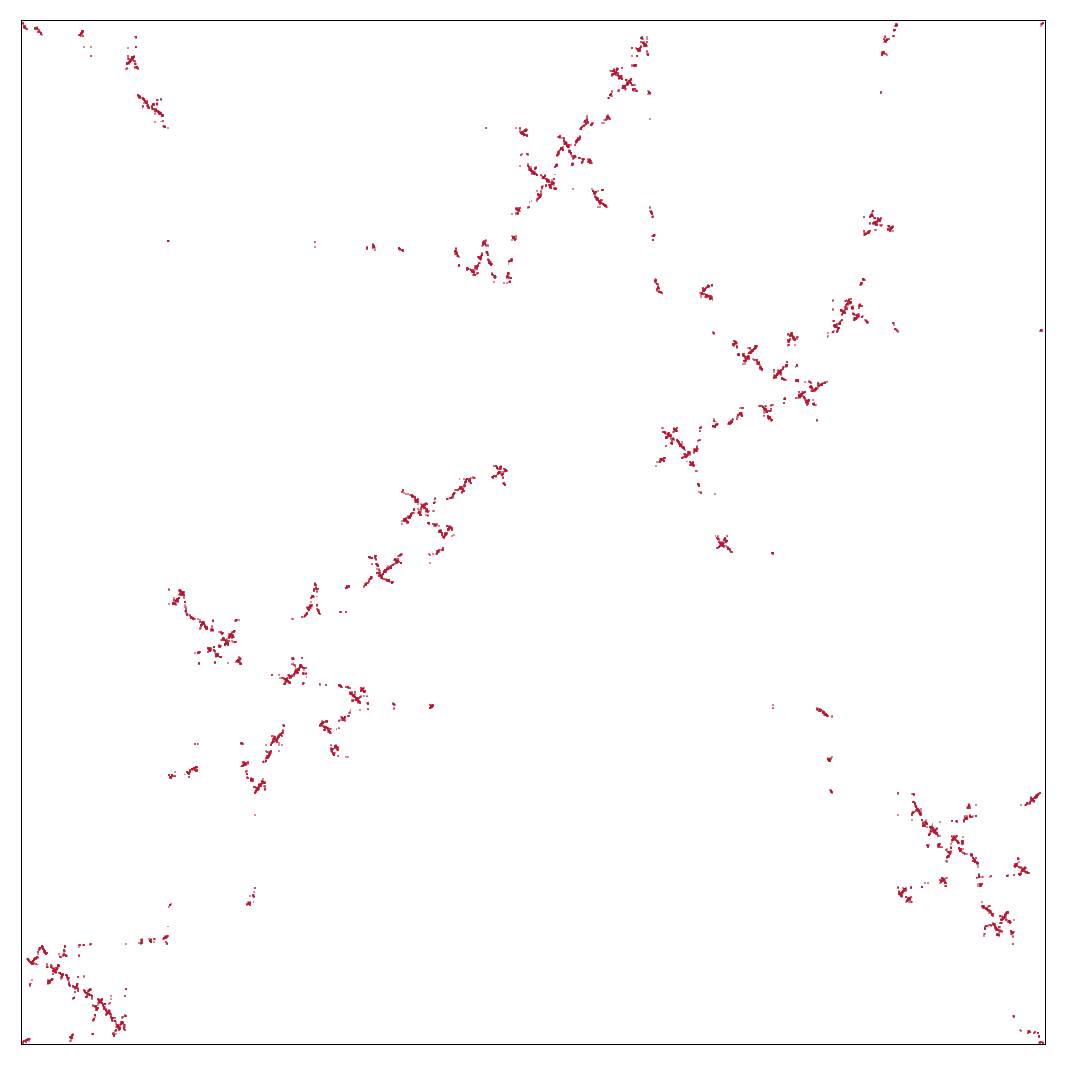}
	\end{minipage}
	\begin{minipage}[c]{0.149\textwidth}
		\centering
		\includegraphics[scale=0.18]{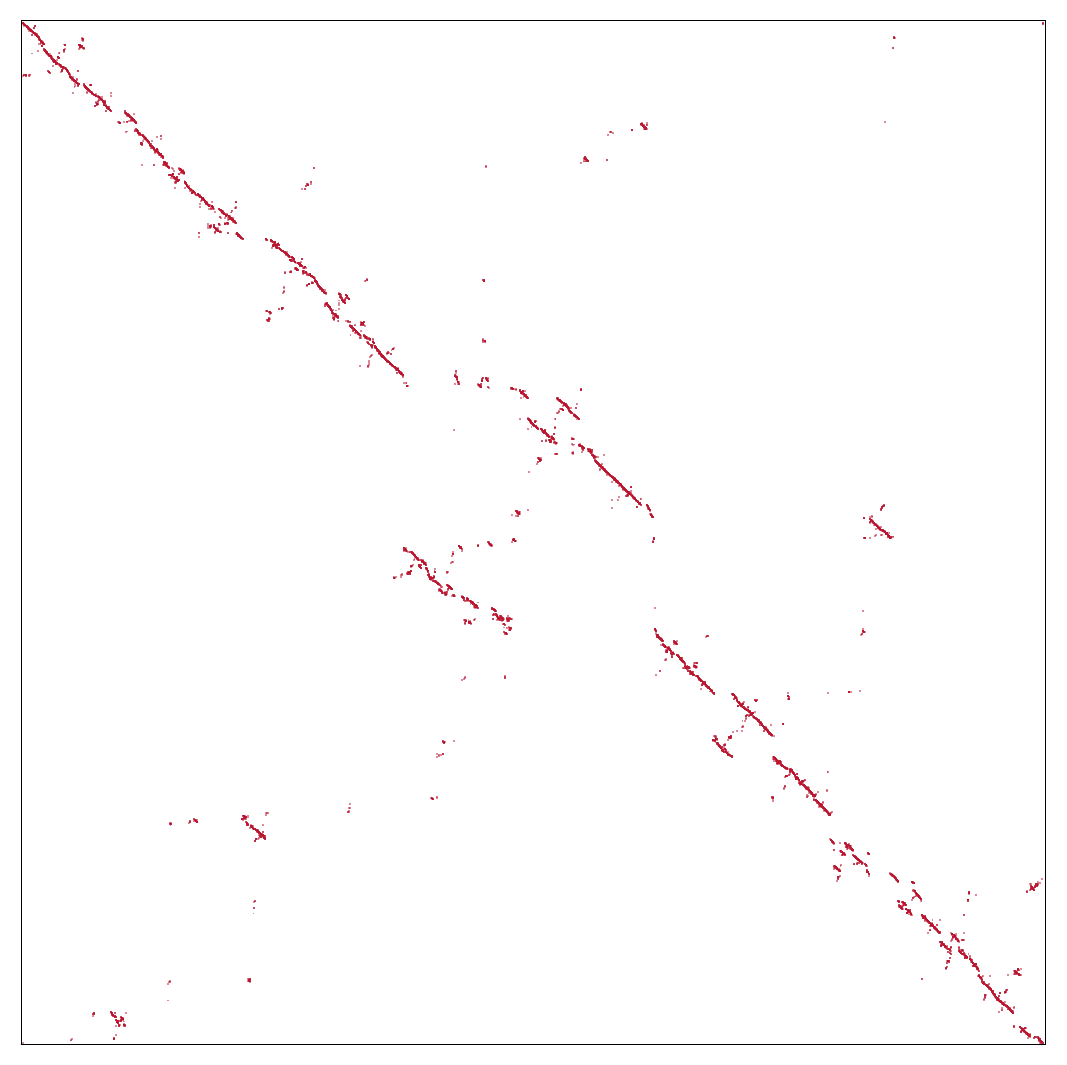}
	\end{minipage}
	\begin{minipage}[c]{0.149\textwidth}
		\centering
		\includegraphics[scale=0.18]{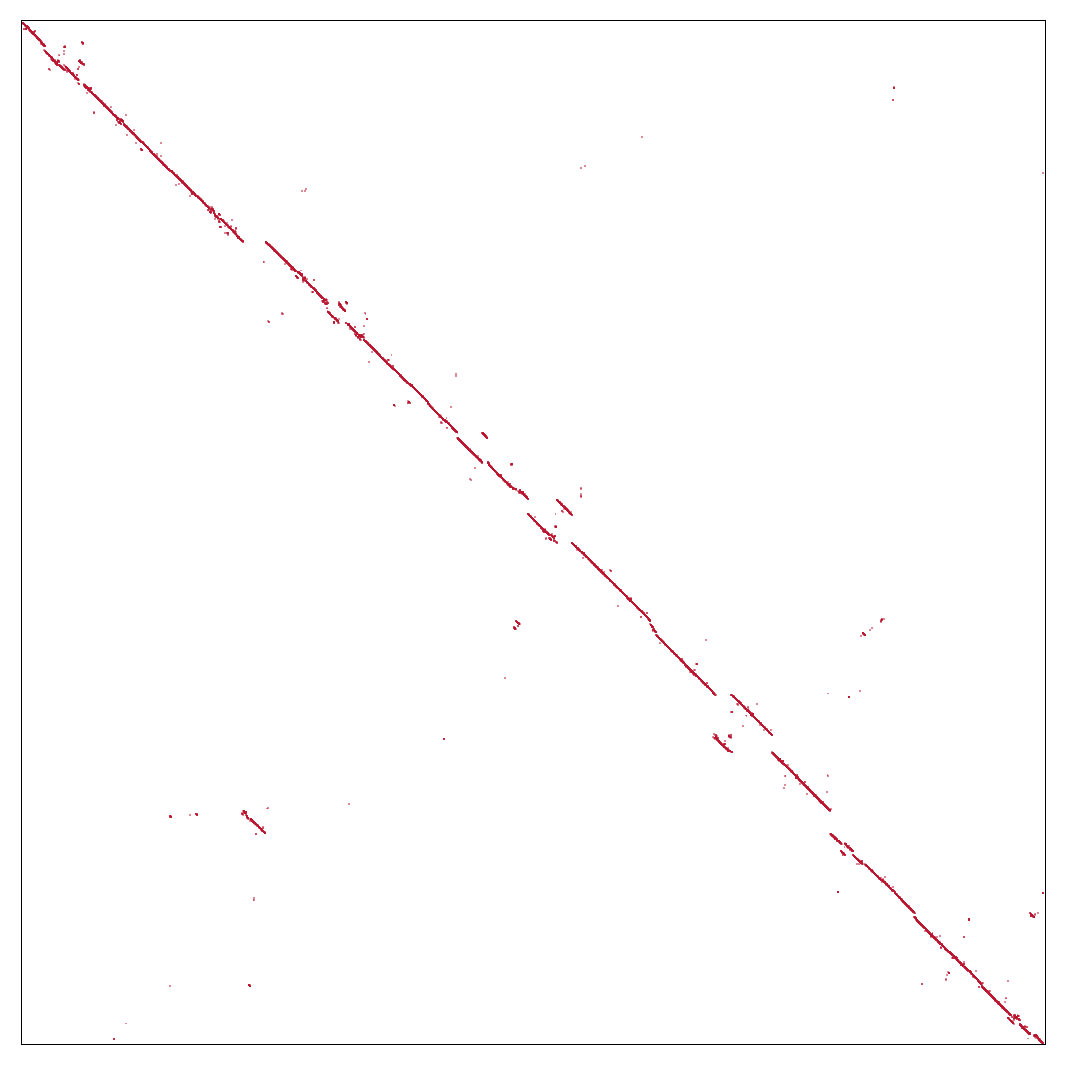}
	\end{minipage}
	\hspace{0.1 cm}
	\begin{minipage}[c]{0.149\textwidth}
		\centering
		\includegraphics[scale=0.225]{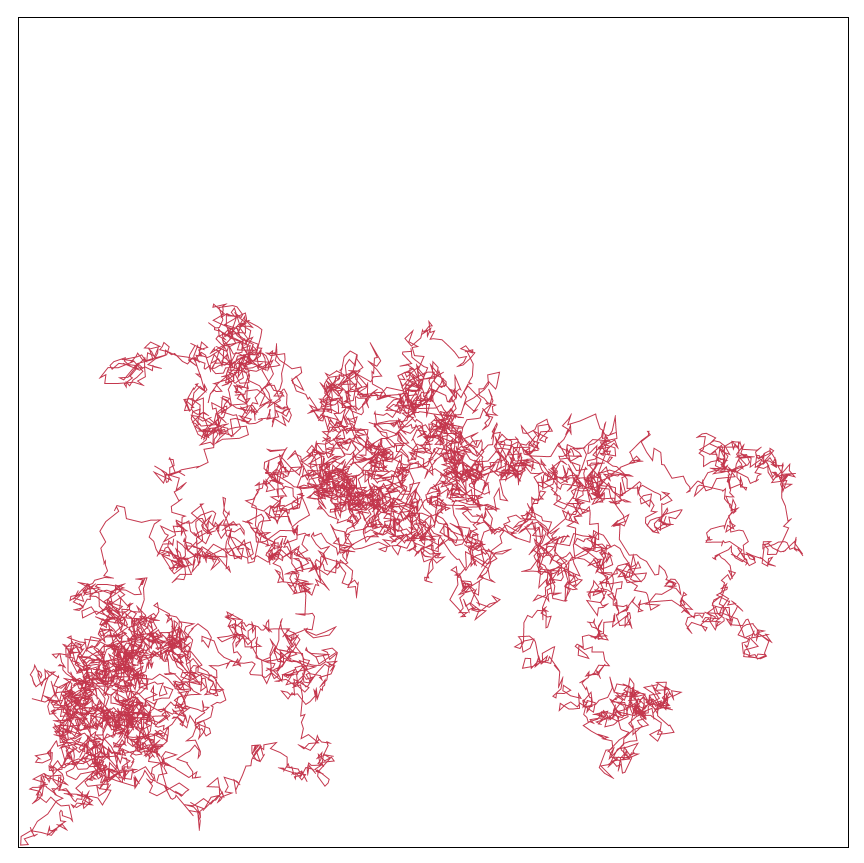}
	\end{minipage}
	
	\begin{minipage}[c]{0.06\textwidth}
		\centering
		\footnotesize$\rho=$\\
		\footnotesize$0.995$
	\end{minipage}
	\begin{minipage}[c]{0.149\textwidth}
		\centering
		\includegraphics[scale=0.18]{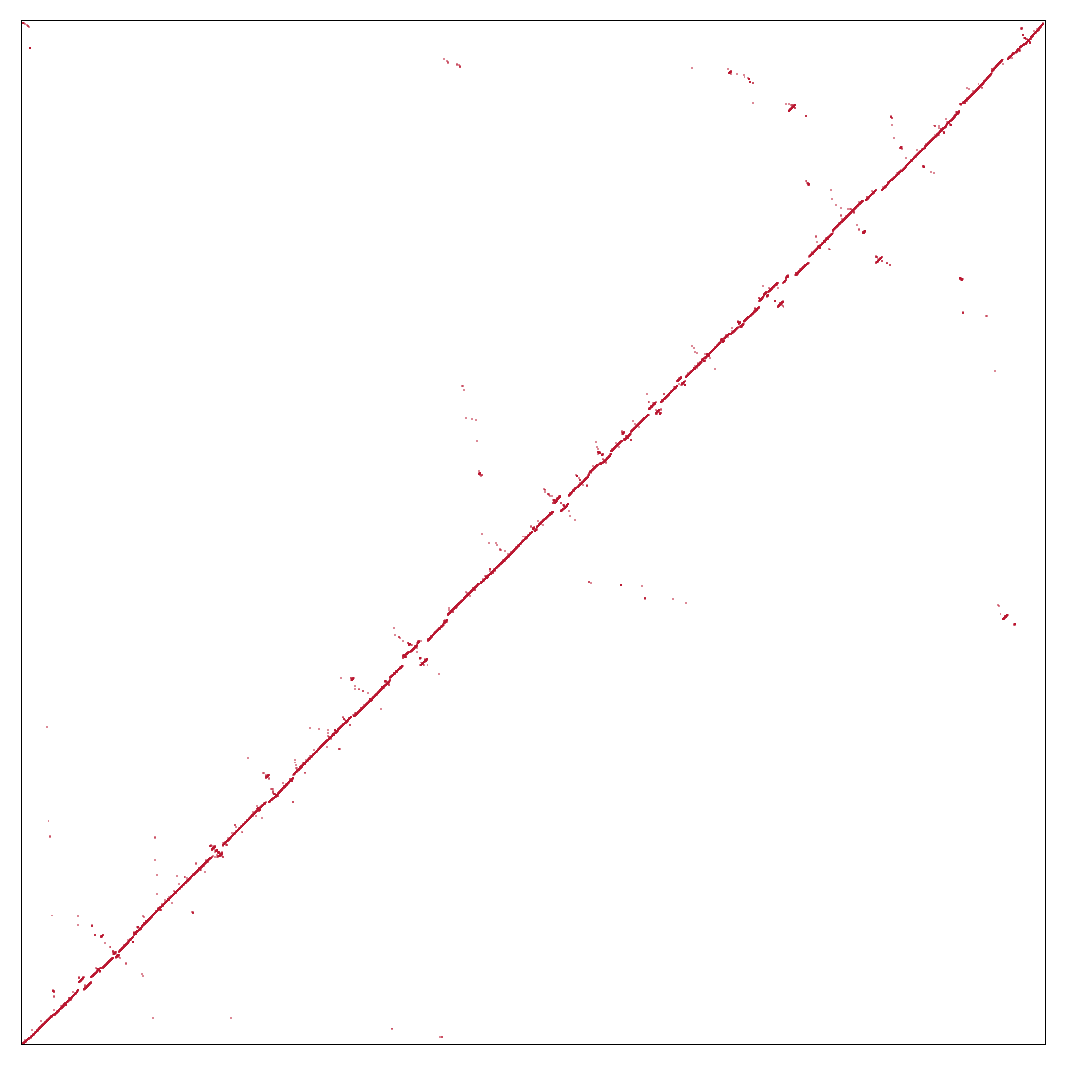}
	\end{minipage}
	\begin{minipage}[c]{0.149\textwidth}
		\centering
		\includegraphics[scale=0.18]{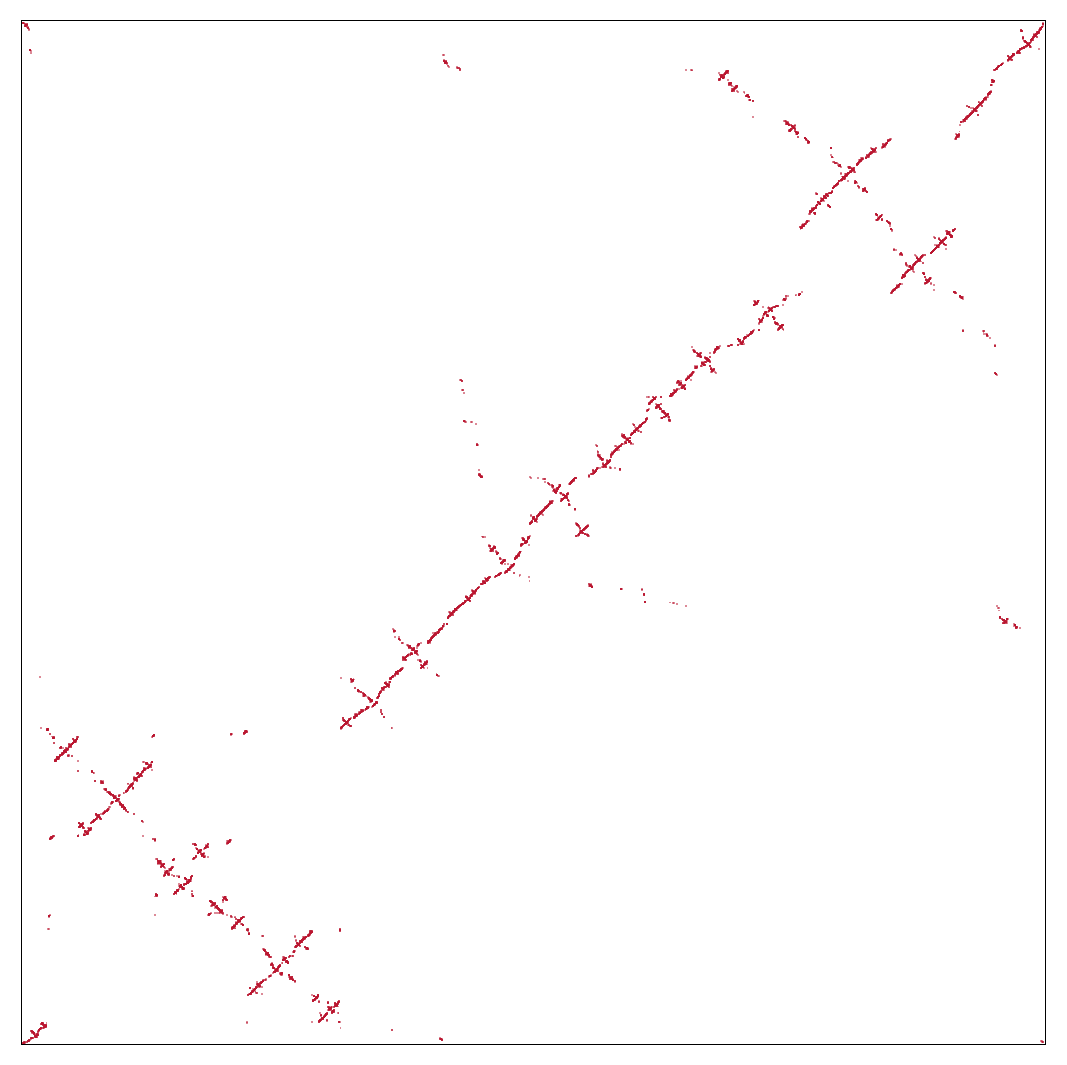}
	\end{minipage}
	\begin{minipage}[c]{0.149\textwidth}
		\centering
		\includegraphics[scale=0.18]{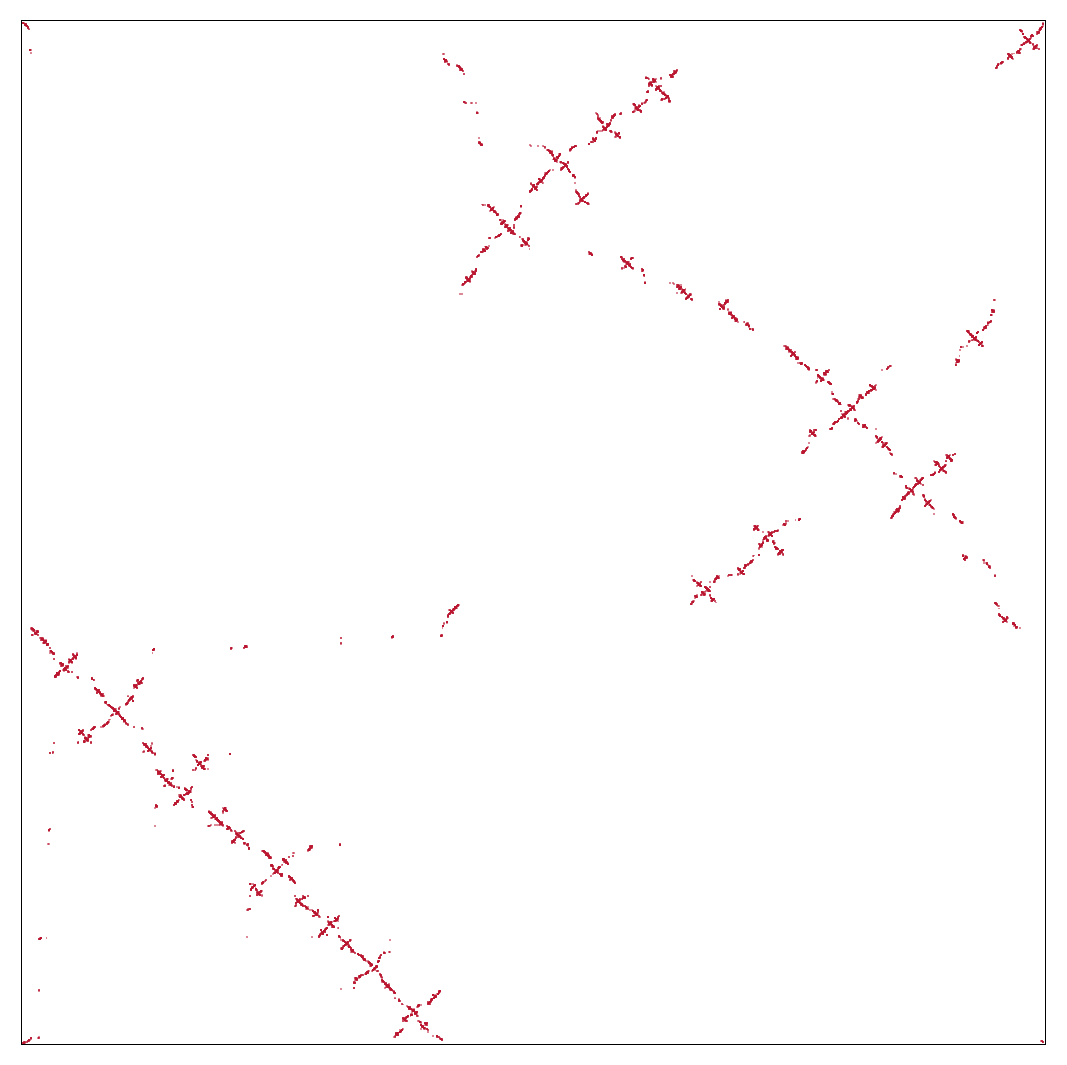}
	\end{minipage}
	\begin{minipage}[c]{0.149\textwidth}
		\centering
		\includegraphics[scale=0.18]{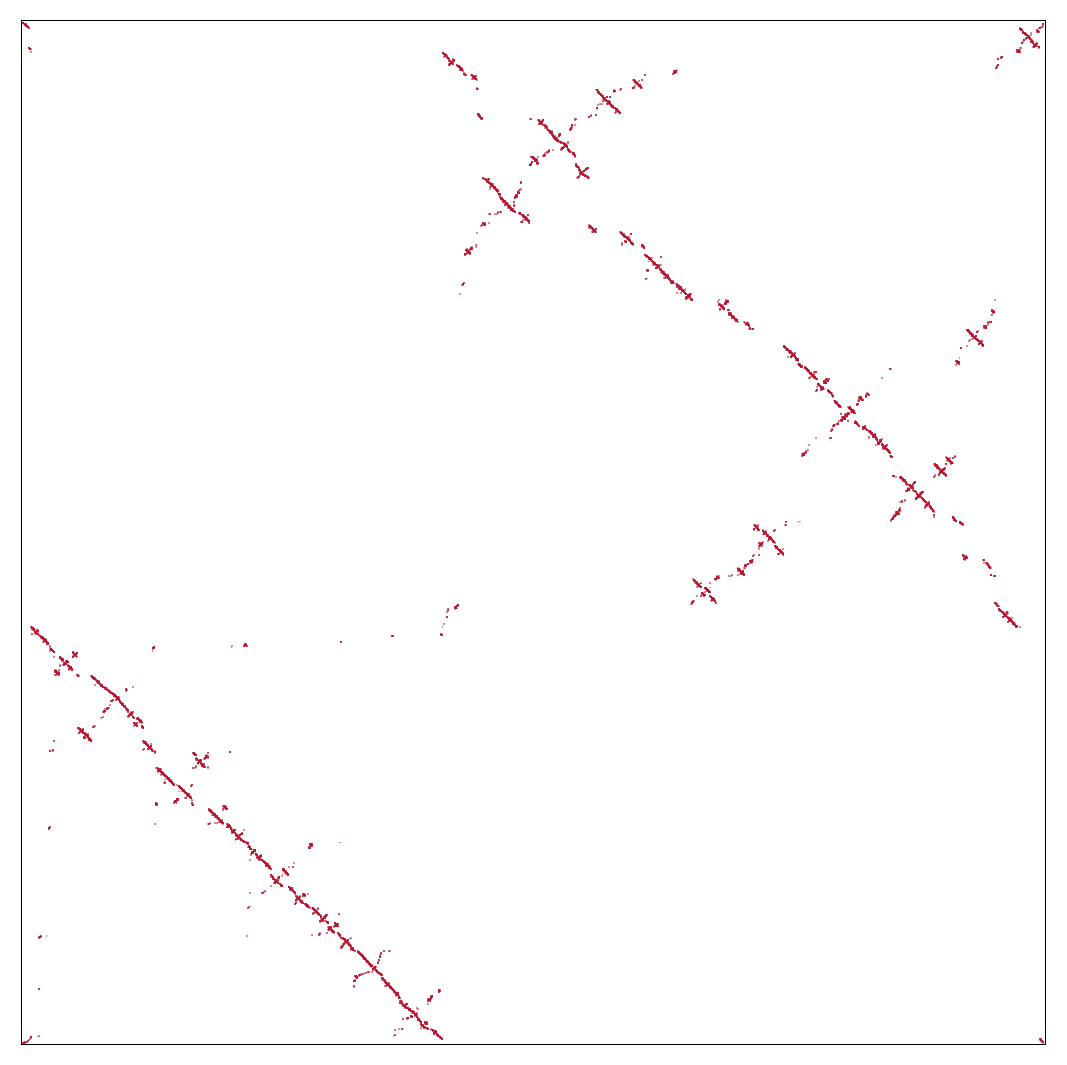}
	\end{minipage}
	\begin{minipage}[c]{0.149\textwidth}
		\centering
		\includegraphics[scale=0.18]{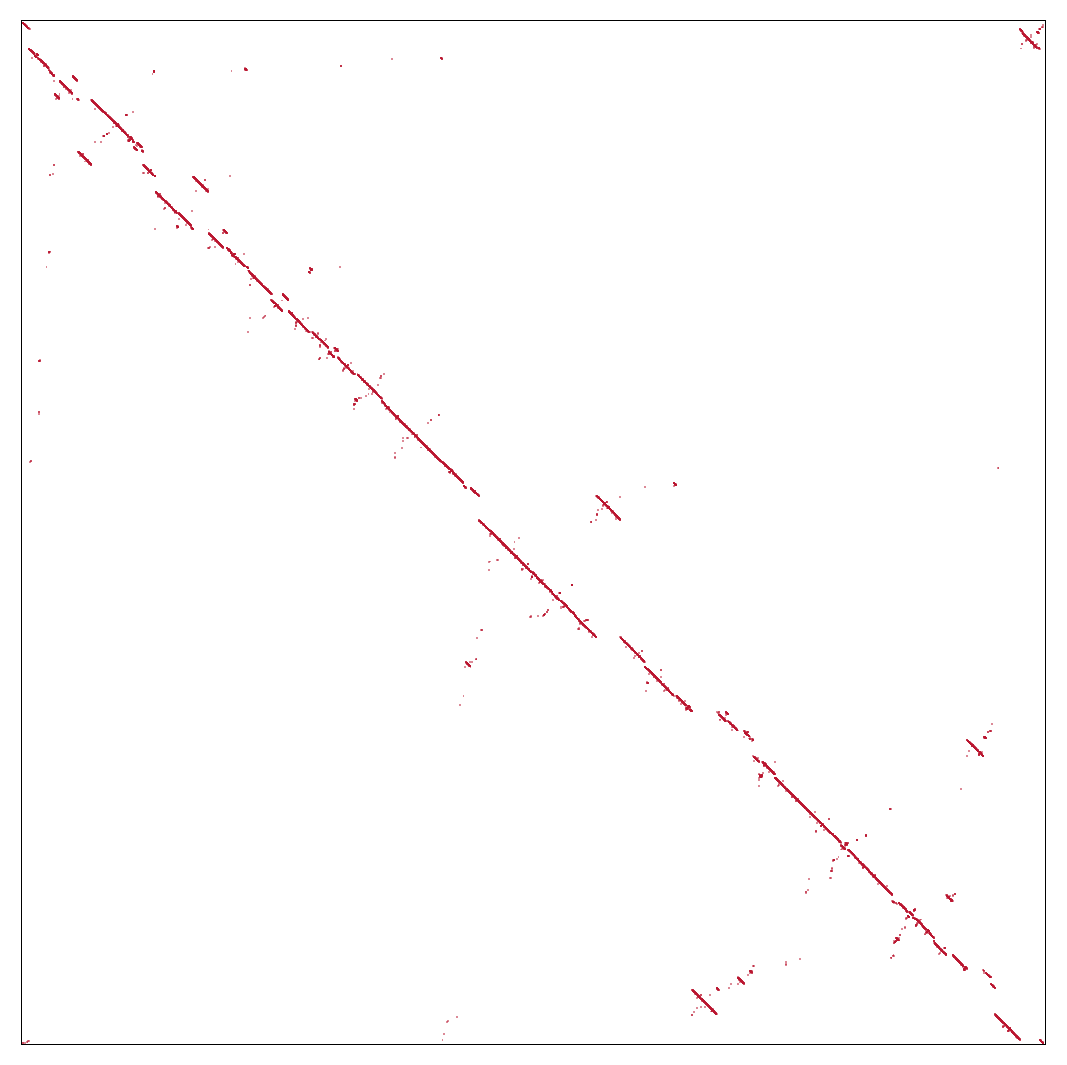}
	\end{minipage}
	\hspace{0.1 cm}
	\begin{minipage}[c]{0.149\textwidth}
		\centering
		\includegraphics[scale=0.225]{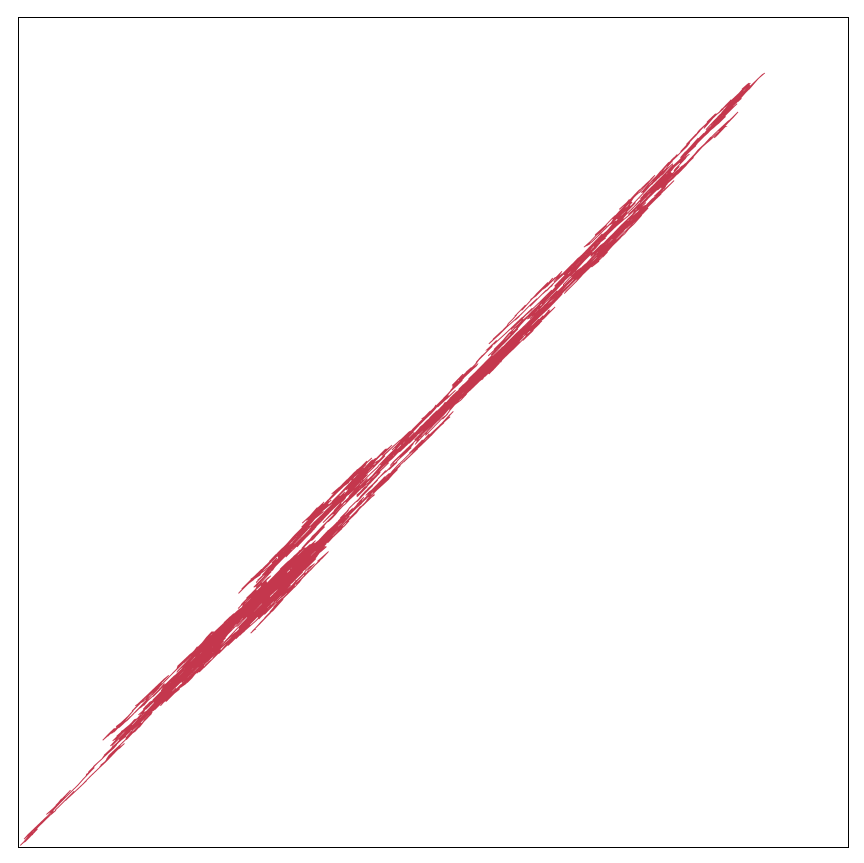}
	\end{minipage}
	
	\begin{minipage}[c]{0.06\textwidth}
		\centering
		\footnotesize$\rho=1$
	\end{minipage}
	\begin{minipage}[c]{0.149\textwidth}
		\centering
		\includegraphics[scale=0.18]{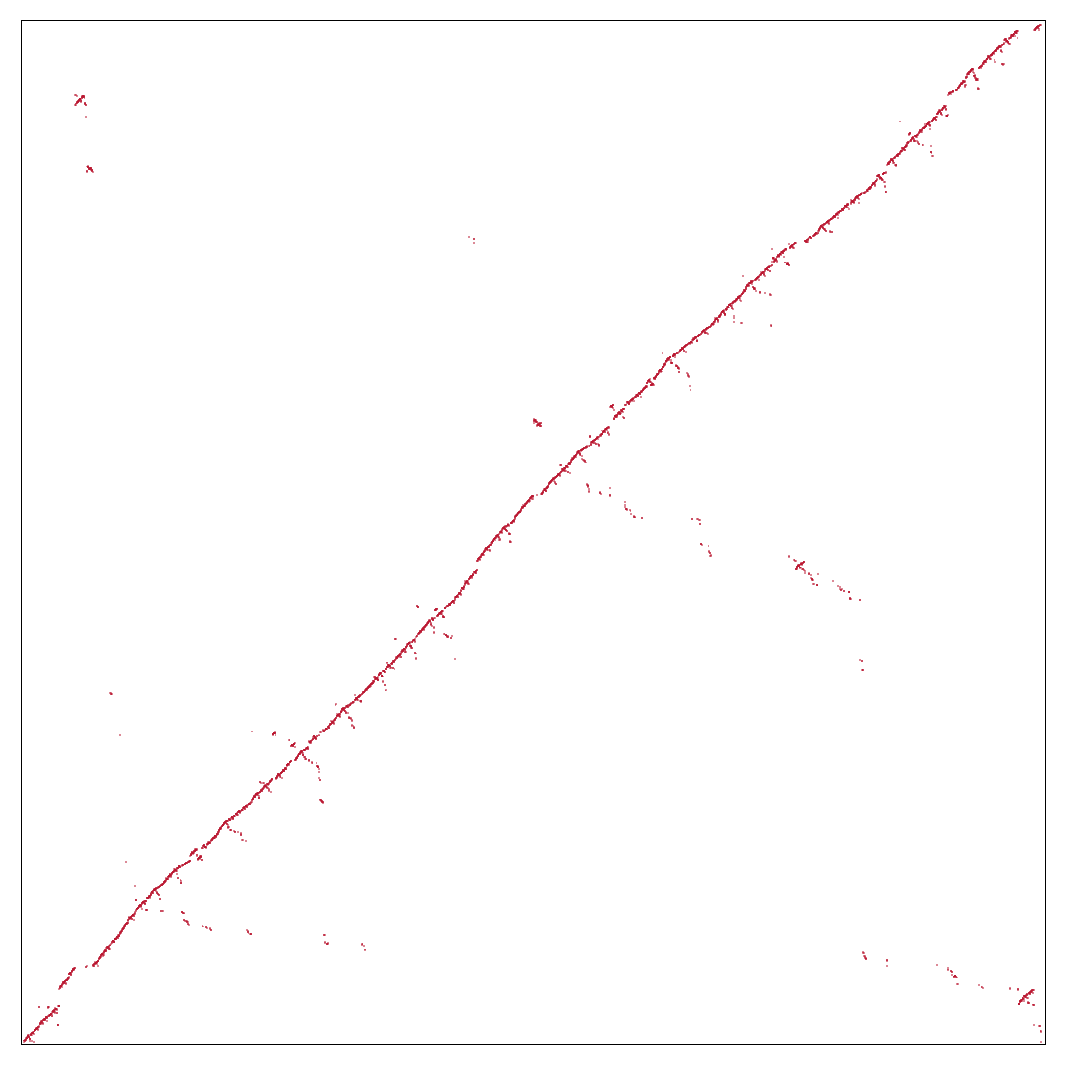}
	\end{minipage}
	\begin{minipage}[c]{0.149\textwidth}
		\centering
		\includegraphics[scale=0.18]{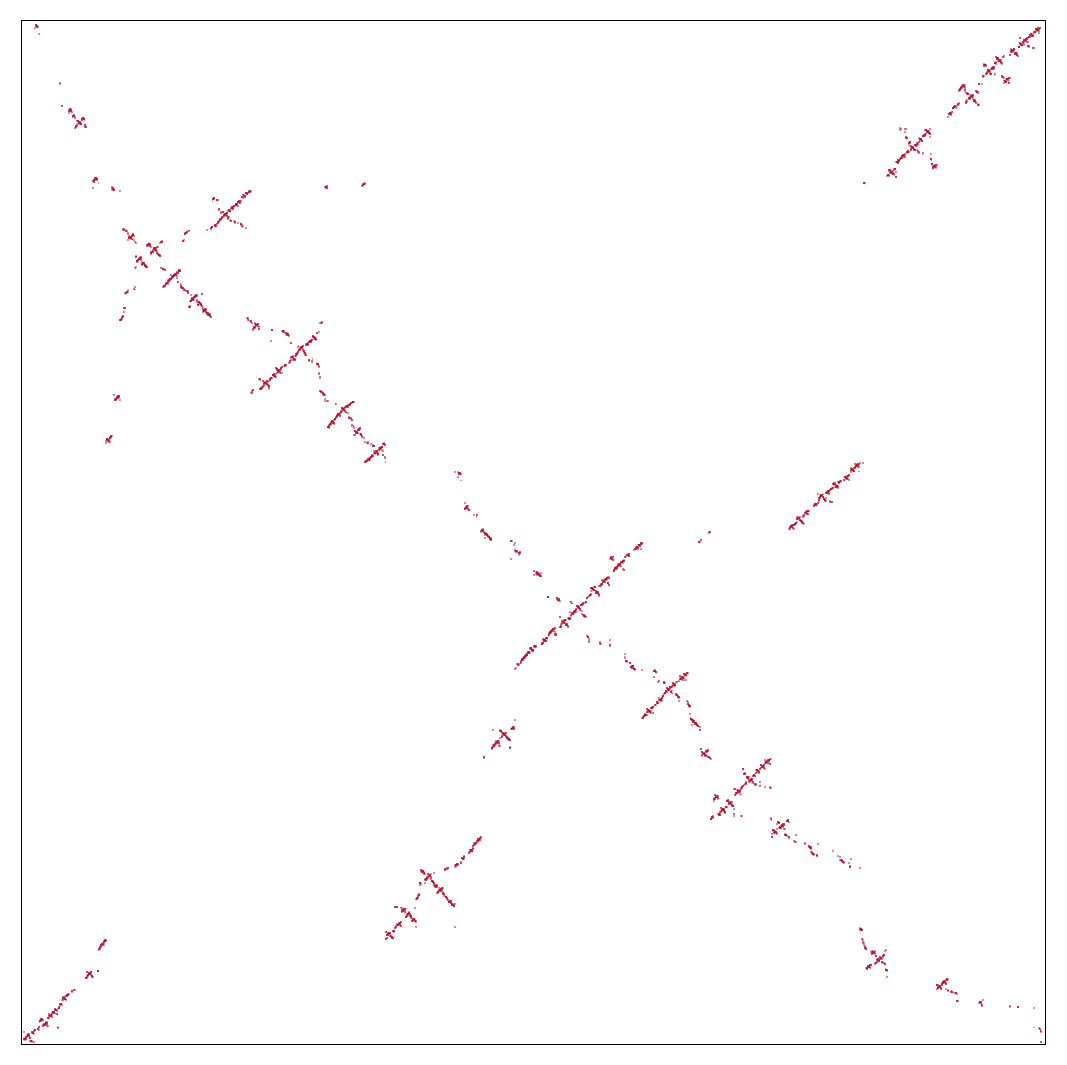}
	\end{minipage}
	\begin{minipage}[c]{0.149\textwidth}
		\centering
		\includegraphics[scale=0.18]{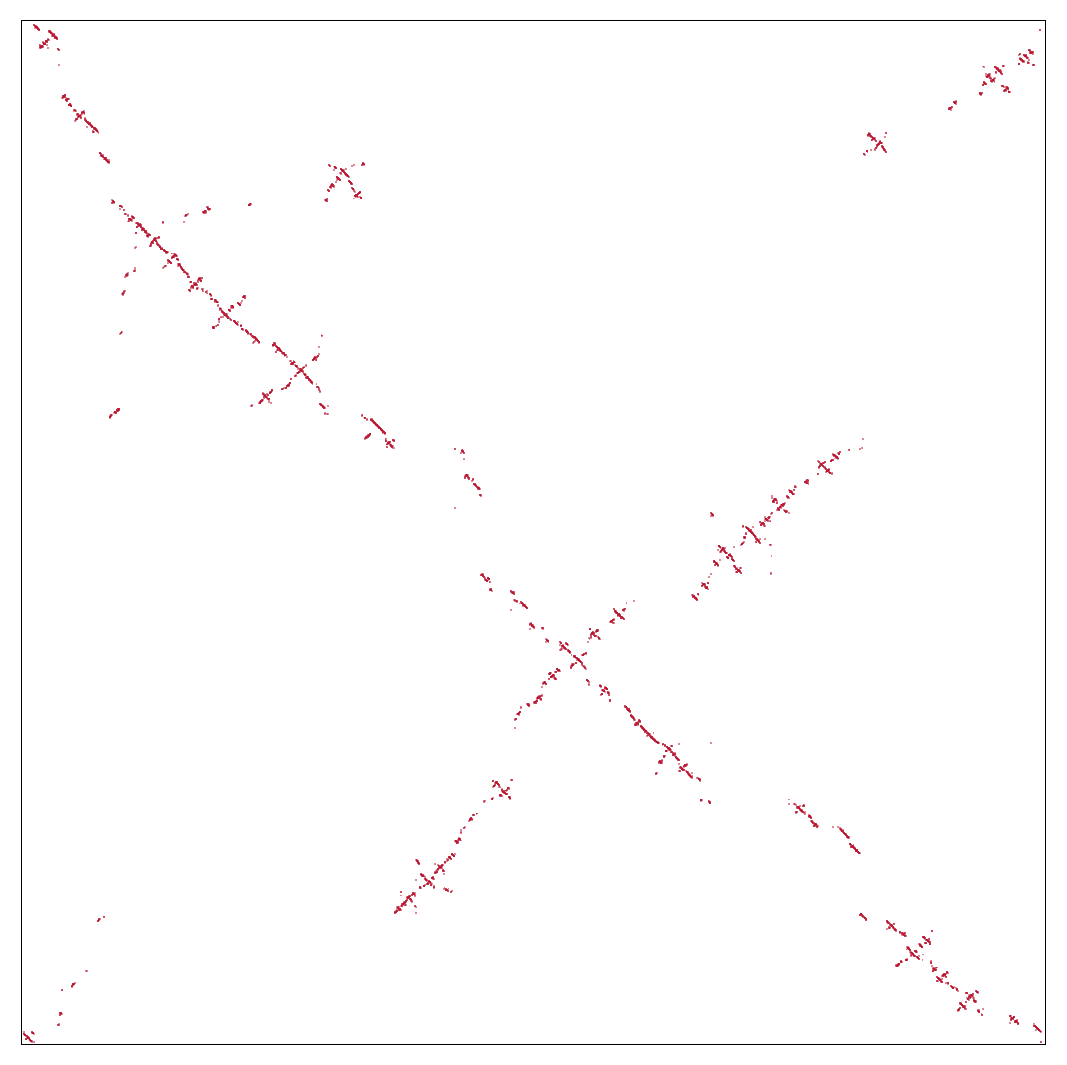}
	\end{minipage}
	\begin{minipage}[c]{0.149\textwidth}
		\centering
		\includegraphics[scale=0.18]{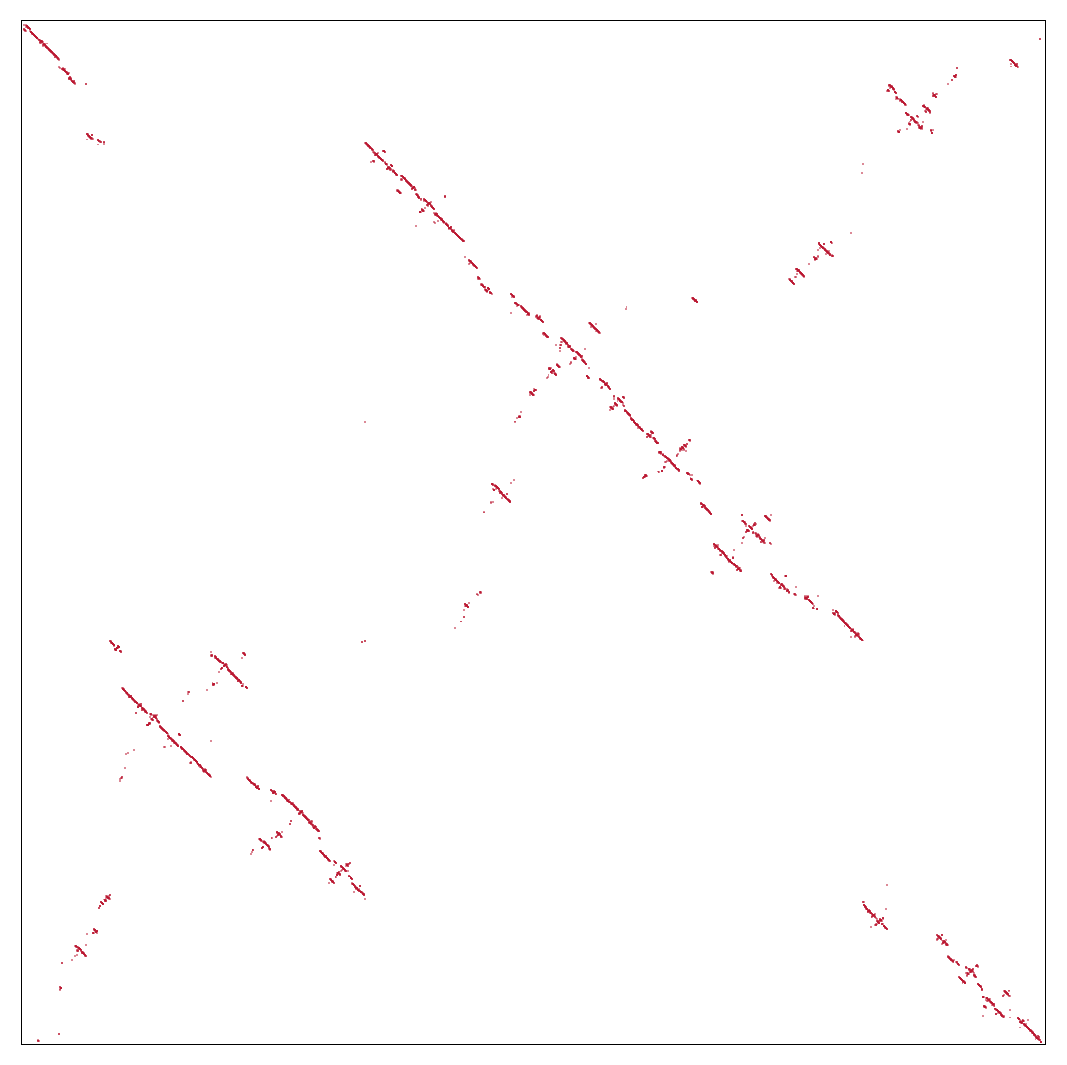}
	\end{minipage}
	\begin{minipage}[c]{0.149\textwidth}
		\centering
		\includegraphics[scale=0.18]{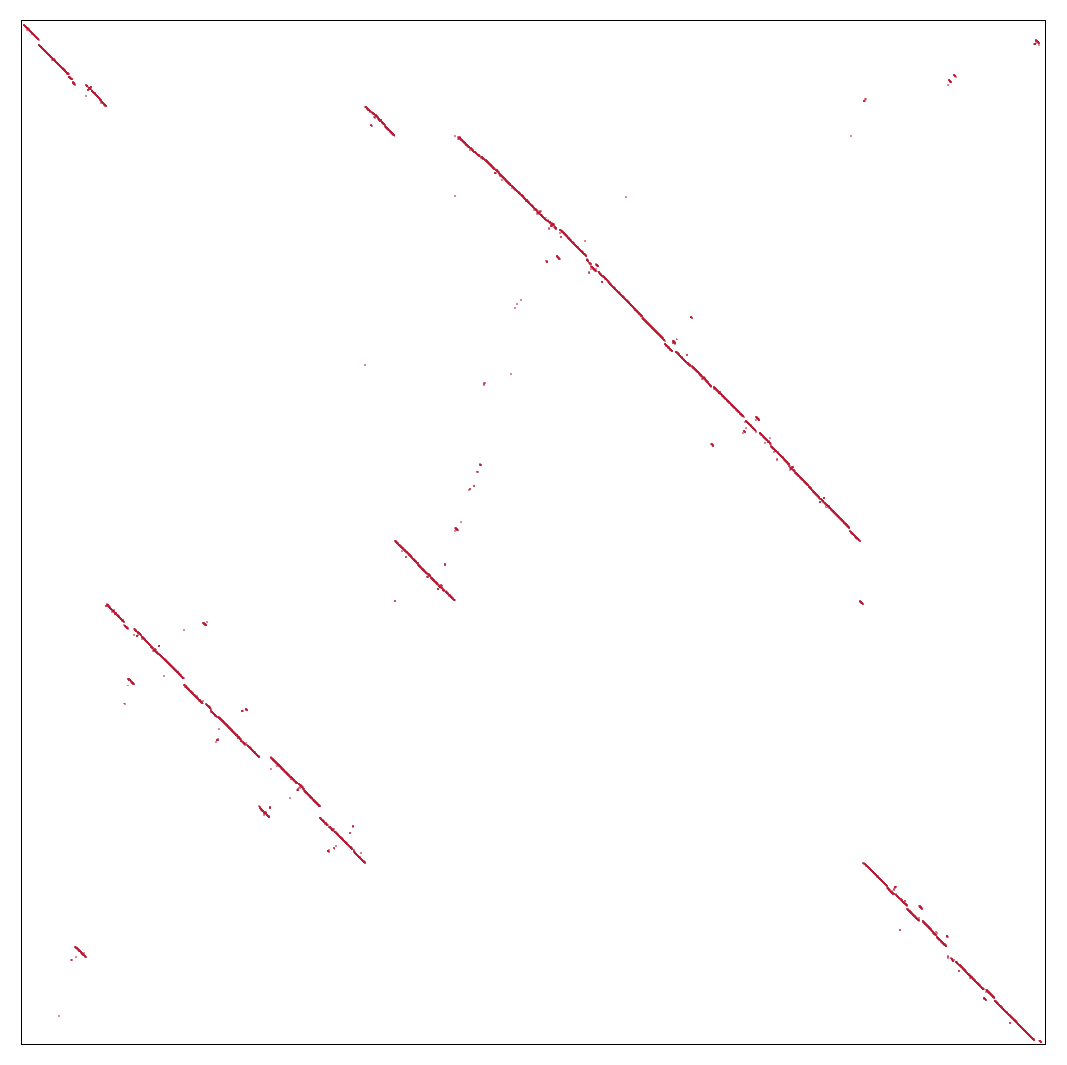}
	\end{minipage}
	\hspace{0.1 cm}
	\begin{minipage}[c]{0.149\textwidth}
		\centering
		\includegraphics[scale=0.202]{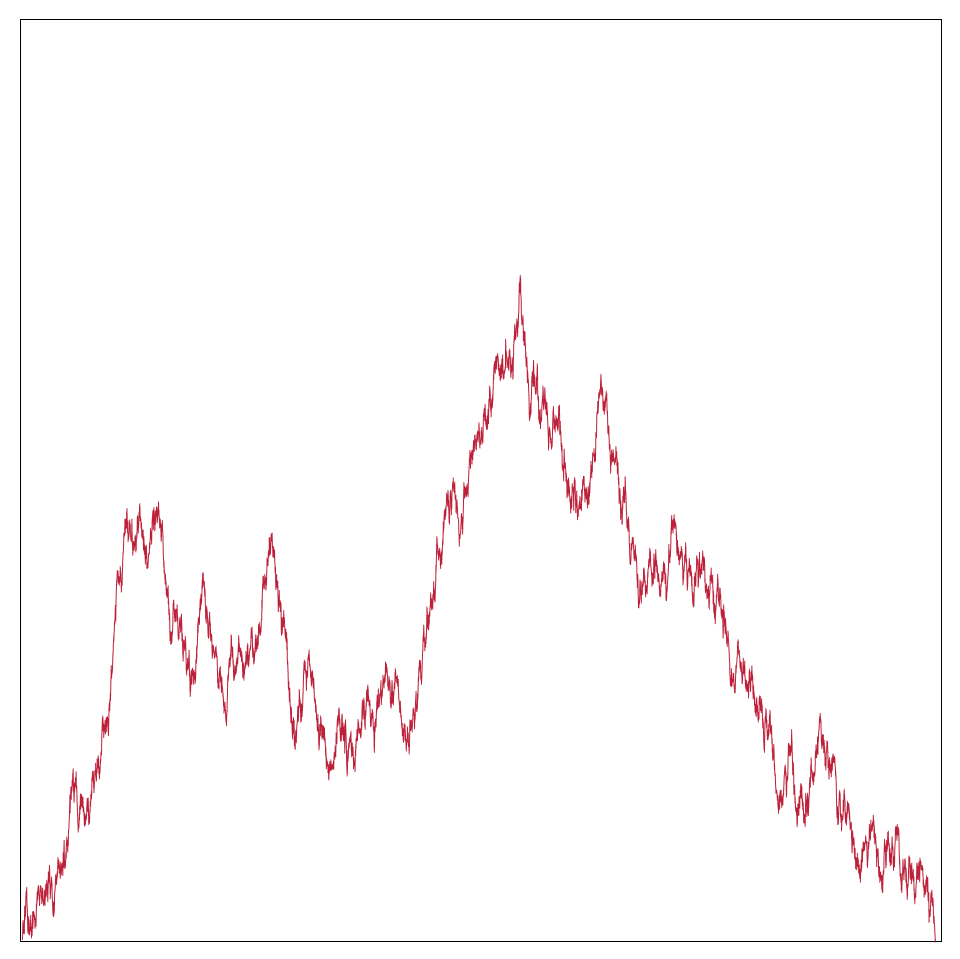}
	\end{minipage}
	\caption{We collect here some simulations of the skew Brownian permuton $\bm \mu_{\rho,q}$ for various values of $(\rho,q)\in(-1,1]\times[0,1]$. In every row, there are five simulations of $\bm \mu_{\rho,q}$ and at the end of each row, there is the corresponding two-dimensional Brownian excursion of correlation $\rho$ in the non-negative quadrant (the specific value of $\rho$ is indicated at the beginning of every row). Each column corresponds to different values of the parameter $q$ (the specific value of $q$ is indicated at the top of every column).
	We highlight that permutons in the same row are driven by \emph{the same} Brownian excursion plotted at the end of the row and so they are coupled. Note that when $\rho=1$ (this is the case of the last row) the corresponding two-dimensional Brownian excursion is simply a one-dimensional Brownian excursion and it is plotted using the standard diagram for real-valued functions. We also remark that the plots of the two-dimensional Brownian excursions $(\conti X_{\rho},\conti Y_{\rho})$ on the right-hand side  have different scales for different values of $\rho$ because $\Var(\conti X_{\rho}+\conti Y_{\rho})$ is of order $1+\rho$.
	}\label{fig:uyievievbeee}
\end{figure}

\subsubsection{Relations with the biased Brownian separable permuton and the Baxter permuton}

As already highlighted in \cref{rem:friv}, the Baxter permuton coincides with the skew Brownian permuton of parameters $\rho=-1/2$ and $q=1/2$. The relation between the skew Brownian permuton and the biased Brownian separable permuton is less trivial and will be investigated in \cref{sect:equiv}, where we will prove the following result.

\begin{thm}\label{thm:Baxt_brow_same}
	For all $p\in[0,1]$, the biased Brownian separable permuton $\bm \mu^S_p$ has the same distribution as the skew Brownian permuton $\bm \mu_{1,1-p}$.
\end{thm} 

As mentioned in the abstract, our new construction of the skew Brownian permuton brings two different limiting objects under the same roof, identifying a new larger universality class and explaining the connection between the Baxter permuton and the biased Brownian separable permuton.

\medskip

The unsatisfactory feature of the instances of the skew Brownian permuton mentioned above is that either $\rho=1$ (biased Brownian separable permuton) or $q=1/2$ (Baxter permuton), and in both cases the SDEs in \cref{eq:flow_SDE_gen} take a simplified form: either the driving process is a one-dimensional Brownian excursion, as in \cref{eq:Tanaka}, or the local time term cancels, as in \cref{eq:flow_SDE} (and so we simply have a perturbed Tanaka equation instead of a skew perturbed Tanaka equation). In a companion paper \cite{borga2021strongBaxter}, we show that this is not the case for the family of  \emph{strong-Baxter permutations}. We recall this result in the next section.

\subsubsection{A model of random permutations converging to the skew Brownian permuton with non-degenerate parameters}

We start by defining \emph{strong-Baxter permutations}.

\begin{defn}
\emph{Strong-Baxter permutations} are permutations avoiding the three vincular patterns $2\underbracket[.5pt][1pt]{41}3$,  $3\underbracket[.5pt][1pt]{14}2$ and $3\underbracket[.5pt][1pt]{41}2$, i.e.\ permutations $\sigma$ such that there are no indices $1\leq i<j<k-1<n$ such that $\sigma(j+1) < \sigma(i) < \sigma(k) < \sigma(j)$ or $\sigma(j) < \sigma(k) < \sigma(i) < \sigma(j+1)$ or $\sigma(j+1) < \sigma(k) < \sigma(i) < \sigma(j)$.
\end{defn}

In \cite{borga2021strongBaxter} we proved the following permuton convergence result.

\begin{thm}[{\cite[Theorem 1.6]{borga2021strongBaxter}}]\label{thm:strong-baxter}
	Let $\bm \sigma_n$ be a uniform strong-Baxter permutation of size $n$. The following convergence in distribution in the permuton sense holds:
	\begin{equation}
		\bm \mu_{\bm \sigma_n}\xrightarrow{d}\bm \mu_{\rho,q},
	\end{equation}  
	where $\rho\approx-0.2151$ is the unique real root of the polynomial 
	\begin{equation}
		1+6\rho+8\rho^2+8\rho^3,
	\end{equation} 
	and $q\approx 0.3008$ is the unique real root of the polynomial
	\begin{equation}
		-1+6q-11q^2+7q^3.
	\end{equation} 
\end{thm}

In the same paper, we also investigated the permuton limit of \emph{semi-Baxter permutations}. We first recall their definition.

\begin{defn}
	\emph{Semi-Baxter permutations} are permutations avoiding the vincular pattern $2\underbracket[.5pt][1pt]{41}3$, i.e.\ permutations $\sigma$ such that there are no indices $1\leq i<j<k-1<n$ such that $\sigma(j+1)<\sigma(i)<\sigma(k)<\sigma(j)$.
\end{defn}

\begin{thm}[{\cite[Theorem 1.8]{borga2021strongBaxter}}]\label{thm:semi-baxter}
	Let $\bm \sigma_n$ be a uniform semi-Baxter permutation of size $n$. The following convergence in distribution in the permuton sense holds:
	\begin{equation}
		\bm \mu_{\bm \sigma_n}\xrightarrow{d}\bm \mu_{\rho,q},
	\end{equation}  
	where
	\begin{equation}
		\rho=-\frac{1+\sqrt 5}{4}\approx -0.8090
		\quad\text{and}\quad
		q=\frac{1}{2}.
	\end{equation}  
\end{thm}

A summarizing table of the various models of random constrained permutations that are currently known to converge to the skew Brownian permuton is given in \cref{fig:Summary_table}. We hope that this table will be enlarged in future research projects (see possible research directions in \cref{sect:open_problems}).

\begin{figure}[htbp]
	\centering
	\includegraphics[scale=0.8]{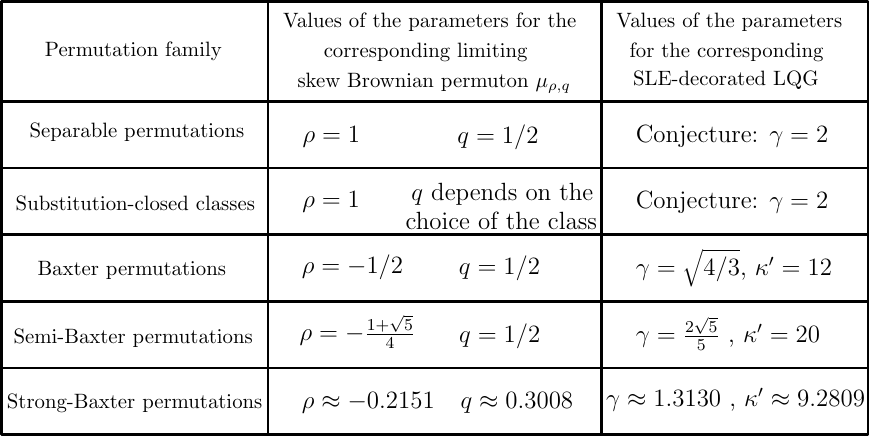}
	\caption{A table with the various models of random constrained permutations that are known to converge to the skew Brownian permuton. The last column shows the values of the parameters $\gamma$ and $\kappa'$ for the corresponding SLE-decorated LQG spheres. These values are obtained using the relations in \cref{eq:paramLQG} and \cref{thm:sbpfromlqg} below. For the case $\rho = 1$ we conjecture that the corresponding $\gamma$-parameter is $\gamma=2$; for further explanations see \cref{rem:paramrel}.
		\label{fig:Summary_table}
	}
\end{figure}

\subsection{The skew Brownian permuton in Liouville quantum gravity theory}\label{sect:efbweifbewoubf}

In the previous section, we saw that the skew Brownian permuton is the permuton limit of various families of constrained permutations.
In this section, we explain how the skew Brownian permuton also arises from a Liouville quantum gravity (LQG) sphere decorated with two coupled Schramm-Loewner evolution (SLE) curves. This result will explain \emph{directly at the continuum level} various connections between decorated planar maps and constrained permutations investigated in the literature \emph{at the discrete level}. More precisely, in the combinatorial literature, several bijections between families of constrained permutations and decorated planar maps have been investigated, see for instance \cite{MR1394948,MR1485138,MR2734180,MR3101730,fusy2021enumeration}.

We highlight that the connection between these two (apparently unrelated) objects has proven to be a relevant tool to establish some results on the skew Brownian permuton using SLE/LQG techniques (see \cite{borga2022baxter,borga2022meanders}). In addition, we believe that building on this new connection, some convergence results available for discrete models of random permutations might be transferred to convergence results for the corresponding discrete models of planar maps (we refer  to \cref{sect:open_problems} for more details on open problems). 

Finally, we mention that our results give a new description of the interactions between two SLE-curves (coupled in the \emph{imaginary geometry sense}; see below for more details) decorating an LQG cone or sphere in terms of the skew perturbed Tanaka equations (see Lemmas \ref{lem:keylemfinvol} and \ref{lem:keylem} below). To the best of our knowledge, this connection with SDEs was not been explicitly exploited so far in the Liouville quantum gravity literature.

\bigskip

We introduce various classical objects related to Liouville quantum gravity. We do not give precise definitions of all these objects, but we provide precise references for each of them.\footnote{Most of the references in this section point to the self-contained survey paper \cite{gwynne2019mating}, where the authors chose to focus (in the finite-volume case) on \emph{quantum disks with boundary length} $\ell$. In the present paper, we need to consider \emph{quantum spheres} instead of \emph{quantum disks}. We remark that a quantum sphere can be viewed as a quantum disk with zero boundary length. If the reader is still not satisfied, we point out that all the definitions used in this section can be also found in \cite{MR4010949}.}

We fix $\gamma\in(0,2)$ and we introduce the following parameters defined in terms of $\gamma$,
\begin{equation}\label{eq:paramLQG}
\kappa=\gamma^2, \qquad \kappa'=16/\gamma^2, \qquad \rho=-\cos(\pi\gamma^2/4), \qquad \chi=2/\gamma-\gamma/2.
\end{equation}
We start by recalling the construction of space-filling SLE$_{\kappa'}$ curves on
$\mathbb{C}\cup\{\infty\}$. One way of constructing them (as proved in \cite{MR3719057}) is from the flow lines of the vector field $e^{i(\hat{\bm h}/\chi+\theta)}$, where $\hat{\bm h}$ is a whole-plane Gaussian free field (modulo a global additive multiple of $2\pi\chi$) and $\theta\in[0,\pi]$. Flow lines of the vector field $e^{i(\hat{\bm h}/\chi+\theta)}$ are constructed in the imaginary geometry sense and are shown to be whole-plane SLE$_\kappa(2-\kappa)$ processes from the starting point to $\infty$. As shown in \cite[Theorem 1.9]{MR3719057}, flow lines of $e^{i(\hat{\bm h}/\chi+\theta)}$ started at different points of $\mathbb{Q}^2$ merge into each other upon intersecting and form a tree. The \emph{space-filling} SLE$_{\kappa'}$ \emph{counterflow line from $\infty$ to $\infty$ generated by} $\hat{\bm{h}}$ \emph{and of angle} $\theta$, denoted $\bm\eta_\theta$, is the Peano curve of this tree, i.e.\ it is the curve which visits the points of $\mathbb{C}$ in chronological order, where a point $x\in\mathbb C$ is visited before a point $y\in \mathbb C$ if the flow line of angle $\theta$ from $x$ merges into the flow line of angle $\theta$ from $y$ on the left side of the latter flow line. 
For $\theta\in[0,\pi]$, we consider the pair $(\bm \eta_0,\bm \eta_\theta)$ of space-filling SLE$_{\kappa'}$ counterflow lines from $\infty$ to $\infty$ generated by the same Gaussian free field $\hat{\bm h}$ of angle $0$ and $\theta$, respectively.\footnote{Our angle conventions differ from those commonly used in the imaginary-geometry literature: we label the space-filling curve $\eta_\theta$ by the angle of the flow-line tree used in its Peano construction: the two boundary flow lines at a typical visited point have angles $\theta$ and $\theta+\pi$. Thus, our subscript $\theta$ is a boundary-flow-line angle, rather than the angle of the corresponding counterflow line. In the usual convention, a counterflow line of angle $\alpha$ has boundary flow-line angles $\alpha+\pi/2$ and $\alpha-\pi/2$.}

\medskip

Let now $({\mathbb C}\cup\{\infty\}, \bm h, \infty)$ be a $\gamma$-LQG sphere, independent of $\hat{\bm h}$, with quantum area one and one marked point at $\infty$ (see \cite[Definition 3.20]{gwynne2019mating}). 
We parametrize the pair $( \bm \eta_0, \bm \eta_\theta)$ by the $ \mu_{\bm h}$-LQG area measure (see \cite[Section 3.3]{gwynne2019mating}) so that $\bm\eta_0(0)=\bm\eta_0(1)=\bm\eta_\theta(0)=\bm\eta_\theta(1)=\infty$ and $ \mu_{\bm h}( \bm\eta_0([0, t]))= \mu_{\bm h}( \bm \eta_\theta([0, t])) = t$ for each $t\in[0, 1]$.

\begin{thm}\label{thm:sbpfromlqg}
	Fix $\gamma\in(0,2)$ and $\theta\in[0,\pi]$. Let $({\mathbb C}\cup\{\infty\}, \bm h, \infty)$ and $(\bm \eta_0,\bm \eta_\theta)$ be the $\gamma$-LQG sphere and the pair of space-filling SLE$_{\kappa'}$ introduced above. 
	For $t\in[0, 1]$, let $\bm\psi_{\gamma,\theta}(t)\in[0, 1]$ denote the first time when $\bm\eta_\theta$ hits the point $\bm\eta_0(t)$.\footnote{We recall that space-filling SLE curves have multiple points. Nevertheless, for each $z\in\mathbb C$, a.s. $z$ is not a multiple point of $\bm\eta_\theta$, i.e., $\bm\eta_\theta$ hits $z$ exactly once. Since $\bm h$ is independent from $(\bm\eta_0,\bm\eta_\theta)$ and $\bm\eta_0$ and $\bm\eta_\theta$ are parametrized by $\mu_{\bm h}$-mass, a.s. the set
	of times $t \in [0, 1]$ such that $\bm\eta_0$ is a multiple point of  $\bm\eta_\theta$ has zero Lebesgue measure.}
	Then the random measure 
	$$(\Id,\bm\psi_{\gamma,\theta})_{*}\Leb$$ 
	is a skew Brownian permuton of parameter $\rho=-\cos(\pi\gamma^2/4)$ and $\overline q=\overline q_\gamma(\theta)\in[0,1]$.
\end{thm}

\begin{rem}\label{rem:paramrel}
	The explicit expression of the function $\overline q_\gamma(\theta)$ is unknown. 
	It holds that
	$\overline q_\gamma(\pi/2)=1/2$. Moreover, for every fixed $\gamma\in(0,2)$, the function $$\overline q_\gamma(\theta):[0,\pi]\to[0,1]$$
	is a homeomorphism and therefore has an inverse function $\theta_\gamma(\overline q)$.
	Finally, for all $\theta\in[0,\pi/2)$, it holds that $\overline q_\gamma(\theta) + \overline q_\gamma(\pi-\theta) = 1$. All these properties are consequences\footnote{The fact that the function $\overline q_\gamma(\theta)$ in \cref{thm:sbpfromlqg} is the same function considered in \cite[Lemma 4.3]{li2017schnyder} follows from \cref{lem:keylemfinvol} below.} of \cite[Lemma 4.3]{li2017schnyder}. We highlight that the fact that $\rho=-\cos(\pi\gamma^2/4)$ and $\overline q_\gamma(\theta)$ is a homeomorphism guarantees that the range of parameters for the skew Brownian permuton appearing in \cref{thm:sbpfromlqg}  is $\rho \in (-1,1)$ and $q\in [0,1]$. That is, all the possible skew Brownian permutons appear in \cref{thm:sbpfromlqg}, except the biased  Brownian separable permuton $\bm \mu_{1,q}$ for $q\in[0,1]$. We conjecture that the latter case can be obtained from the critical $2$-LQG since there are several analogies between the construction of $\bm \mu_{1,q}$ and the results in \cite{aru2021mating}.
\end{rem}

The proof of \cref{thm:sbpfromlqg} is given in \cref{sect:LQGperm} building on results of \cite{duplantier2014liouville} and \cite{gwynne2016joint}. We also invite the curious reader to look at \cite[Lemma 2.8]{borga2022meanders} where a simple consequence of \cref{thm:sbpfromlqg} is derived (giving another description of the skew Brownian permuton in terms of SLEs and LQG) and where it is also shown that our specific choice of $\bm\psi_{\gamma,\theta}(t)\in[0, 1]$ in the statement of \cref{thm:sbpfromlqg} can be replaced by any measurable function $\widetilde{\bm\psi}_{\gamma,\theta}(t)\in[0, 1]$ such that $\bm\eta_\theta(\widetilde{\bm\psi}_{\gamma,\theta}(t))=\bm\eta_0(t)$ for all $t\in [0,1]$.

\subsection{Open problems}\label{sect:open_problems}
In this final section of the introduction, we collect a list of open questions and problems that we plan to address in future research projects.

\begin{enumerate}
	\item \textbf{Intensity measure and Hausdorff dimension of the skew Brownian permuton.}
	The skew Brownian permuton $\bm \mu_{\rho,q}$ is a new fractal random
	measure of the unit square and we plan to investigate some of its properties in the future. For instance, two natural questions are:
	\begin{enumerate}
		\item[(a)] What is the density of the intensity measure $\E[\bm \mu_{\rho,q}]$? How does it depend on the two parameters $\rho$ and $q$?
		\item[(b)] What is the Hausdorff dimension of the support of $\bm \mu_{\rho,q}$? And again, how does it depend on the two parameters $\rho$ and $q$?
	\end{enumerate}
	We highlight that the questions above were answered in \cite{maazoun} for the biased  Brownian separable permuton $\bm \mu_{1,q}$. For instance, it was shown that almost surely, the support of $\bm \mu_{1,q}$
	is totally disconnected, and its Hausdorff
	dimension is 1 (with one-dimensional Hausdorff measure bounded above by $\sqrt 2$). For an expression for the intensity measure see \cite[Theorem 1.7]{maazoun}.
	
	We believe that a key tool to answer the two questions above would be the new connection with SLE-decorated LQG spheres described in \cref{thm:sbpfromlqg}. 
	
	\item \textbf{Properties of the stochastic process $\varphi_{\cnz_{\rho,q}}$.}
	Recall that the stochastic process $\varphi_{\cnz_{\rho,q}}$ was defined in \cref{eq:random_skew_function} as follows
	\begin{equation}
		\varphi_{\cnz_{\rho,q}}(t)\coloneqq
		\Leb\left( \big\{x\in[0,t)|\cnz_{\rho,q}^{(x)}(t)<0\big\} \cup \big\{x\in[t,1]|\cnz_{\rho,q}^{(t)}(x)\geq0\big\} \right), \quad t\in[0,1].
	\end{equation}

	Heuristically speaking, the skew Brownian permuton $\bm \mu_{\rho,q}$ is the graph of the function $\varphi_{\cnz_{\rho,q}}$. Therefore, determining some properties of the process $\varphi_{\cnz_{\rho,q}}$ might be a useful step in the investigation of the permuton $\bm \mu_{\rho,q}$. Our intuition suggests that $\varphi_{\cnz_{\rho,q}}$ might be a particular (conditioned) \emph{fragmentation process} in the sense of \cite{MR3606760}. We also believe that a key step to understanding the process $\varphi_{\cnz_{\rho,q}}$ will be to investigate the joint law of two processes $\cnz_{\rho,q}^{(x)}$ and $\cnz_{\rho,q}^{(y)}$. We remark that similar questions have been investigated in \cite{MR2094439} for pair of solutions to the SDEs
	\begin{equation}
		\begin{cases}
			d\cnz^{(u)}(t) = d \conti B(t)+(2q-1)\cdot d\conti L^{\cnz^{(u)}}(t),& t\in \R_{>u},\\
			\cnz^{(u)}(t)=0,&  t\in [0,u],
		\end{cases}
	\end{equation}
	where $\conti B(t)$ is a standard one-dimensional Brownian motion.

	\item \textbf{Proportions of patterns in the skew Brownian permuton.} We first define the permutation induced by $k$ points in the square $[0,1]^2$. Take a sequence of $k$ points $(X,Y)=((x_1,y_1),\dots, (x_k,y_k))$ in $[0,1]^2$ with distinct $x$ and $y$ coordinates. 
	The \emph{$x$-reordering} of $(X,Y)$  is the unique reordering of the sequence $(X,Y)$ such that
	$x_{(1)}<\cdots<x_{(k)}$, and is denoted by $\left((x_{(1)},y_{(1)}),\dots, (x_{(k)},y_{(k)})\right)$.
	The values $(y_{(1)},\ldots,y_{(k)})$ are then in the same
	relative order as the values of a unique permutation of size $k$, that we call the \emph{permutation induced by} $(X,Y)$.

	Let $\mu$ be a permuton and $((\bm X_i,\bm Y_i))_{i\in \Z_{>0}}$ be an i.i.d.\ sequence with distribution $ \mu$. We denote by $\Perm_k(\mu)$ the random permutation induced by $((\bm X_i,\bm Y_i))_{i \in \llbracket 1,k \rrbracket}$. 
	
	Finally, if $\sigma$ is a permutation of size $n$ and $\pi$ a pattern\footnote{We refer the reader to \cite[Section 1.6.1]{borga2021random} for an introduction to permutation patterns.} of size $k\leq n$, then we denote by $\occ(\pi,\sigma)$ the number of occurrences of $\pi$ in $\sigma$.
	Moreover, we denote by $\pocc(\pi,\sigma)$ the proportion of occurrences of $\pi$ in $\sigma,$ that is,
	$\pocc(\pi,\sigma)\coloneqq\frac{\occ(\pi,\sigma)}{\binom{n}{k}}$.
	
	\begin{quest}\label{prop:ewifweivf}
		Let $\bm \mu_{\rho,q}$ be a skew Brownian permuton with parameters $(\rho,q)\in(-1,1]\times[0,1]$. For all $k\in\Z_{>0}$ and patterns $\pi$ of size $k$ is it possible to determine the limiting distribution of the random variables
		$$\lim_{n\to \infty} \pocc(\pi,\Perm_n(\bm \mu_{\rho,q}))=\P( \Perm_k(\bm \mu_{\rho,q})=\pi|\bm \mu_{\rho,q})?$$
		We point out that the fact that the latter limit exists and satisfies the expression above is a consequence of the fact that $\Perm_n(\bm \mu_{\rho,q})$ converges in distribution in the permuton sense to $\bm \mu_{\rho,q}$ (see \cite[Lemma 2.3 and Theorem 2.5]{bassino2017universal}).
	\end{quest}
	A simpler question is to compute, for all $k\in\Z_{>0}$ and patterns $\pi$ of size $k$, the probabilities
	$$\P( \Perm_k(\bm \mu_{\rho,q})=\pi)=\lim_{n\to \infty} \E\left[\pocc(\pi,\Perm_n(\bm \mu_{\rho,q}))\right].$$
	Our intuition suggests the following statement.
	\begin{conj}\label{eq:conj_pos}
		For all $(\rho,q)\in(-1,1)\times(0,1)$ and for all $k\in\Z_{>0}$ and patterns $\pi$ of size $k$, it holds that 
		$$\P( \Perm_k(\bm \mu_{\rho,q})=\pi)>0.$$
	\end{conj}
	We remark that since $\bm \mu_{1,q}$ can be obtained as the limit of pattern-avoiding permutations then we know that $\P( \Perm_k(\bm \mu_{\rho,q})=\pi)$ is zero for some $q\in(0,1)$ and some pattern $\pi$.
	
	\item \textbf{The length of the longest increasing subsequence in the skew Brownian permuton.}
	
	This question is motivated by the recent work \cite{bassino2021linear}. The authors show that the length of the longest increasing subsequence in a sequence of  permutations converging to the biased Brownian separable permuton has sublinear size. Their proof builds on some self-similarity properties of the biased Brownian separable permuton.
	
	\begin{quest}\label{quest-lis}
	Let $(\bm \sigma_n)_{n\in\Z_{\geq 0}}$ be a sequence of random permutations converging in distribution in the permuton sense to the skew Brownian permuton $\bm \mu_{\rho,q}$. What is the length of the longest increasing subsequence in $\bm \sigma_n$?
	\end{quest}
	We expect, under some regularity conditions on the sequence $\bm \sigma_n$, a formula only depending on $\rho, q$.
	
	\item \textbf{Properties of the parameters $\rho$ and $q$ defining the skew Brownian permuton.} The skew Brownian permuton $\bm \mu_{\rho,q}$ is defined in terms of the two parameters $\rho \in (-1,1]$ and $q\in [0,1]$. It would be interesting to find some natural statistics on permutons (i.e.\ a map from the space of permutons $\mathcal M$ to $\R$ describing some natural quantity) determined by these two parameters. For instance, we expect that $q$ controls the proportion of inversions in $\bm \mu_{\rho,q}$. More precisely, we conjecture the following.
	
	\begin{conj}\label{conj:incresin}
		Set $\P( \Perm_2(\bm \mu_{\rho,q})=21)=f(\rho,q)$. Then  for every fixed $\rho \in (-1,1]$ the function $f(\rho,q)$ is increasing in $q$.
	\end{conj}
	
	It would be even more interesting to derive an explicit expression for $f(\rho,q)$. It would be also remarkable to find some natural statistics $\stc(\cdot)$ on permutons such that $\stc(\bm \mu_{\rho,q})=g(\rho,q)$ a.s., for some function $g$. This would imply that the laws of $\bm \mu_{\rho,q}$ are singular for different values of the parameters $\rho \in (-1,1]$ and $q\in [0,1]$.
	
	\item \textbf{The skew Brownian permuton for $\rho=-1$.} The skew Brownian permuton $\bm{\mu}_{-1,q}$ has not been defined yet, because when $\rho=-1$, it is meaningless to condition a two-dimensional Brownian motion of correlation $\rho=-1$ to stay in the non-negative quadrant. It would be interesting to investigate if there is a natural way of constructing $\bm{\mu}_{-1,q}$. A possible way would be to consider the limit of the permutons $\bm{\mu}_{\rho_n,q}$ for an appropriate sequence $\rho_n\to -1$. Looking at the simulations in \cref{fig:uyievievbeee}, we also believe that a natural candidate model for defining $\bm{\mu}_{-1,q}$ is the Mallows permuton \cite{starr2009thermodynamic, MR3817550}.
	
	\item \textbf{Critical 2-LQG and biased Brownian separable permuton.} \cref{thm:sbpfromlqg} connects the skew Brownian permuton $\bm \mu_{\rho,q}$ to SLE-decorated LQG spheres in the range of parameters $(\rho,q)\in(-1,1)\times[0,1]$.
	It would be interesting to show that the biased  Brownian separable permuton $\bm \mu_{1,q}$ for $q\in[0,1]$ arises from the critical $2$-LQG sphere (in the same spirit of \cref{thm:sbpfromlqg}). Here some new ideas are needed, as it is not clear how to define two SLE curves decorating a $2$-LQG sphere (see \cite{aru2021mating}).
	
	\item \textbf{Models of planar maps associated with semi-Baxter permutations and separable permutations.}
	We believe that there are two natural models of decorated planar maps associated with semi-Baxter permutations and separable permutations that might converge to SLE-decorated LQG spheres. This belief is justified by \cref{thm:sbpfromlqg} and the results in \cite{borga2020scaling}. 
	More precisely,
	\begin{itemize}
		\item We conjecture that \emph{bipolar posets} (introduced in \cite{fusy2021enumeration}) are in bijection with semi-Baxter permutations and converge to a $\gamma$-LQG sphere decorated with two SLE curves of angle $\theta (q)$, where $\gamma$ and $\theta (q)$ are related to the parameters $\rho, q$ appearing in \cref{thm:semi-baxter} through the relations given in  \cref{thm:sbpfromlqg}.
		\item We further conjecture that \emph{rooted series-parallel maps} (studied in \cite[Proposition 6]{MR2734180}, where it is also shown that they are in bijection with separable permutations) converge to the critical $2$-LQG sphere. This would be the first discrete model of planar maps shown to converge to the critical $2$-LQG sphere.
	\end{itemize}

	We point out that the results in this paper give three different possible approaches to answer the questions above: the \emph{permutation approach}, the \emph{SDE approach}, and the \emph{SLE-decorated LQG approach}.\footnote{\emph{Note added in revision:} The SLE-decorated LQG approach has been first used in \cite{borga2022baxter} to compute the intensity measure of the Baxter permuton (giving a partial answer to the Open Problem 1a) and to prove Conjectures \ref{eq:conj_pos} and \ref{conj:incresin}; and recently in \cite{borga2022meanders} to solve the Open Problem 1b and to give a partial answer to \cref{quest-lis}.}
\end{enumerate}

\section*{Acknowledgements}
During the realization of this long-term project, I had the opportunity to exchange various ideas with several people. Here is an alphabetically ordered list of the ones who participated in some discussions and to whom I would like to express my gratitude: thanks to  Mathilde Bouvel, Valentin Féray, Ewain Gwynne,  Hatem Hajri, Nina Holden, Antoine Lejay, Mickaël Maazoun, Grégory Miermont, Vilmos Prokaj, Kilian Raschel, and Xin Sun. We thank the anonymous referees for all their precious and useful comments.

\section{The skew perturbed Tanaka equations}

In this section we focus on the skew perturbed Tanaka equations introduced in \cref{eq:flow_SDE_gen}, p.\ \pageref{eq:flow_SDE_gen}, and \cref{eq:Tanaka}, p.\ \pageref{eq:Tanaka}. In particular, in the following two subsections, we prove \cref{thm:conj_rho_gen} and \cref{thm:fvuwevifview2}.  In both cases, we will first investigate the SDEs in \cref{eq:flow_SDE_gen,eq:Tanaka} when they are driven by a correlated two-dimensional Brownian motion instead of a correlated two-dimensional Brownian excursion, and then we will transfer the results to our specific cases using absolute continuity arguments.

These results are fundamental to then show in \cref{sect:welldef} that the skew Brownian permuton is well-defined, proving  \cref{thm:perm_is_ok}.

\subsection{The skew pure perturbed Tanaka equation}\label{sect:erfkvwijfbwe}

\subsubsection{The Brownian motion case}

Let $(\conti W_{\rho}(t))_{t\in \R_{\geq 0}}=(\conti X_{\rho}(t),\conti Y_{\rho}(t))_{t\in \R_{\geq 0}}$ be a two-dimensional Brownian motion of correlation $\rho\in(-1,1)$ 
and let $q\in[0,1]$ be a parameter. We consider the following SDE
\begin{equation}\label{eq:flow_SDE_gen2}
		d\cnz_{\rho,q}(t) = \idf_{\{\cnz_{\rho,q}(t)> 0\}} d\conti Y_{\rho}(t) - \idf_{\{\cnz_{\rho,q}(t)\leq 0\}} d \conti X_{\rho}(t)+(2q-1)\cdot d\conti L^{\cnz_{\rho,q}}(t),\quad t\in \R_{\geq 0},
\end{equation}
where we recall that $\conti L^{\cnz_{\rho,q}}(t)$ is the symmetric local-time process at zero of $\cnz_{\rho,q}$.
From now on, to simplify notation, we write all the involved processes forgetting the indexes $\rho$ and $q$. We also denote by $\mathcal C(I)$ the set of continuous functions from an interval $I$ of $\R_{\geq 0}$ to  $\R$. Recall also the definition of skew Brownian motion from the beginning of \cref{sect:wievviwevf}.
We prove the following result.

\begin{thm}\label{thm:ex_uni_rho_gen}
	Fix $\rho \in (-1,1)$ and $q\in [0,1]$. Pathwise uniqueness and existence of a strong solution to the SDE in \cref{eq:flow_SDE_gen2} hold. In addition, the  solution is a skew Brownian motion of parameter $q$.
	
	More precisely, if $(\Omega, \mathcal F, (\mathcal F_t)_{t}, \P)$ is a filtered probability space satisfying the usual conditions, and assuming that $\conti W_\rho$ is an $\mathcal F_t$-Brownian motion of correlation $\rho$, 
	\begin{enumerate}
		\item if $\cnz,\widetilde{\cnz}$ are two $\mathcal F_t$-adapted continuous processes that solve \cref{eq:flow_SDE_gen2} a.s., then $\cnz=\widetilde{\cnz}$ a.s.
		\item There exists an  $\mathcal F_t$-adapted continuous process $\cnz$ which solves \cref{eq:flow_SDE_gen2} a.s.
	\end{enumerate}
	In particular, for every $t\in \R_{\geq 0}$ there exists a measurable \emph{solution map} $F_t : \mathcal C([0,t)) \to \mathcal C([0,t))$ such that 
	\begin{enumerate}
		\item[3.] $F_t(\conti W_\rho|_{[0,t)})$ satisfies \cref{eq:flow_SDE_gen2} a.s.\ on the interval $[0,t)$.
		\item[4.] For every $s,t\in\R_{\geq 0}$ with $s \leq t$, then $F_t(\conti W_\rho|_{[0,t)})|_{[0,s)} = F_s (\conti W_\rho|_{[0,s)})$ a.s.
	\end{enumerate} 
\end{thm}

From now on we fix $\rho \in (-1,1)$ and $q\in(0,1)$. The case $q\in\{0,1\}$ will be considered at the end of this section. We introduce the function $r(x)=x/(1-q)\cdot\mathds{1}_{x>0}+x/q\cdot\mathds{1}_{x\leq0}$ and the SDE
\begin{equation}\label{eq:flow_SDE_gen3}
		d\conti R(t) = (1-q)\idf_{\{\conti R(t)> 0\}} d\conti Y(t) - q\idf_{\{\conti R(t)\leq 0\}} d \conti X(t),\quad t\in \R_{\geq 0}. 
\end{equation}
We have the following result.

\begin{prop}\label{prop:equiv_SDE}
	Let $\conti R=(\conti R(t))_{t\in \R_{\geq 0}}$ and $\cnz=(\cnz(t))_{t\in \R_{\geq 0}}$ be two stochastic processes such that $\cnz(t)=r(\conti R(t))$ for all $t\in \R_{\geq 0}$.
	The process $\conti R$ is a strong solution to \cref{eq:flow_SDE_gen3}
	if and only if the process $\cnz$ is a strong solution to \cref{eq:flow_SDE_gen2}. 
\end{prop}

Thanks to \cref{prop:equiv_SDE}, pathwise uniqueness and existence of a strong solution to the SDEs in \cref{eq:flow_SDE_gen2,eq:flow_SDE_gen3} are equivalent.
Thanks to the Yamada--Watanabe theorem\footnote{If further explanations are needed, the reader can look at the discussion at the beginning of \cite[Section 2.1]{MR3882190}.}  (see \cite[Proposition 5.3.20 and Corollary 5.3.23]{MR1121940}), to show pathwise uniqueness and existence of a strong solution to the SDE in \cref{eq:flow_SDE_gen3}, it is enough to show pathwise uniqueness and existence of a weak solution to the SDE in \cref{eq:flow_SDE_gen3}.

\begin{prop}\label{prop:pu_SDE}
	Pathwise uniqueness holds for \cref{eq:flow_SDE_gen3}.
\end{prop}

\begin{prop}\label{prop:ex_SDE}
	There exists a weak solution $\conti R=(\conti R(t))_{t\in \R_{\geq 0}}$ to \cref{eq:flow_SDE_gen3} such that $r(\conti R)$ is a skew Brownian motion of parameter $q$.
\end{prop}

Note that the last two propositions together with \cref{prop:equiv_SDE} prove \cref{thm:ex_uni_rho_gen} for $\rho \in (-1,1)$ and $q\in(0,1)$. We now proceed with the proof of these three propositions.

\bigskip

We start by recalling the \emph{symmetric It\^{o}--Tanaka formula} for convex functions (see for instance \cite[Section 5.1]{MR2280299}) since we will use it repeatedly in this section. Let $f$ be a function from $\R$ to $\R$ which is the difference of two convex functions.
Then 
$$
f'_r(x) \coloneqq \lim_{\varepsilon\to 0,\varepsilon>0}\varepsilon^{-1}
(f(x+\varepsilon)-f(x))
\quad\text{and}\quad 
f'_\ell(x) \coloneqq \lim_{\varepsilon\to 0,\varepsilon<0}\varepsilon^{-1}
(f(x+\varepsilon)-f(x))
$$ 
exist for almost every $x\in\R$. In addition, there exists a signed measure
$\nu$ on $\R$, called the second derivative measure, such that
$$\int_{\R}g(x)\nu(dx)=-\int_{\R}g'(x)f'_\ell(x)dx$$
for every piecewise $\mathcal C^1$ function $g$ with compact support on $\R$. If $f$ has a second
derivative, then $f''$ is the density of $\nu$ with respect to the Lebesgue measure.

Let $\conti M=(\conti M(t))_{t\in\R_{\geq 0}}$ be a real-valued continuous semi-martingale. The \emph{symmetric It\^{o}--Tanaka formula} for $f(\conti M )$ states that
\begin{equation}\label{eq:itotanaka}
	f(\conti M(t))=f(\conti M(0))+\int_0^t \tfrac 1 2 \Big(f'_r(\conti M(s))+f'_\ell(\conti M(s))\Big) d\conti M(s)+\frac{1}{2}\int_{\R} \conti L_x^{\conti M}(t) \nu(dx),\quad t\in\R_{\geq 0},
\end{equation}
where $\conti L_x^{\conti M}(t)$ denotes the symmetric local time process at $x$ of $\conti M$, that is
$$\conti L^{\conti M}_{x}(t)=\lim_{\varepsilon\to 0}\frac{1}{2\varepsilon}\int_0^t\idf_{\{\conti M(s)\in[x-\varepsilon,x+\varepsilon]\}}d[\conti M]_s,$$
where $[\conti M]$ denotes the quadratic variation of $\conti M$.

\bigskip

We can now proceed with the proof of \cref{prop:equiv_SDE}.

\begin{proof}[Proof of \cref{prop:equiv_SDE}]
	We start by recalling that 
	$$r(x)=x/(1-q)\cdot\mathds{1}_{x>0}+x/q\cdot\mathds{1}_{x\leq0}.$$
	We also introduce some additional functions that are used in the proof. We set
	\begin{align}
		&r'(x)=1/(1-q)\cdot\mathds{1}_{x>0}+1/(2q(1-q))\cdot\mathds{1}_{x=0}+1/q\cdot\mathds{1}_{x<0},\\
		&s(x)=(1-q)x\cdot\mathds{1}_{x>0}+qx\cdot\mathds{1}_{x\leq0},\\
		&s'(x)=(1-q)\cdot\mathds{1}_{x>0}+\frac 1 2 \cdot\idf_{x=0}+q\cdot\mathds{1}_{x<0}.
	\end{align}
    Note that $r(s(x))=s(r(x))=x$.
    We first assume that the process $\conti R$ is a strong solution to \cref{eq:flow_SDE_gen3} and we show that the process $\cnz$, defined by $\cnz(t)=r(\conti R(t))$, is a strong solution to \cref{eq:flow_SDE_gen2}. 
	Applying the It\^{o}--Tanaka formula (\cref{eq:itotanaka}) we have that 
	\begin{equation}
		d\cnz(t)=r'(\conti R(t))d\conti R(t)+\frac \gamma 2 d\conti L^{\conti R}(t),
	\end{equation}
	where $\gamma\coloneqq 1/(1-q)-1/q$. Using \cref{eq:flow_SDE_gen3} we obtain that
	\begin{multline}\label{weiubdieuwbdweiobndqenwd}
		d\cnz(t)=\idf_{\{\conti R(t)> 0\}} d\conti Y(t) - \idf_{\{\conti R(t)< 0\}} d \conti X(t)-\frac{1}{2(1-q)}\idf_{\{\conti R(t)= 0\}} d \conti X(t)+\frac \gamma2 d\conti L^{\conti R}(t)\\
		=\idf_{\{\cnz(t)> 0\}} d\conti Y(t) - \idf_{\{\cnz(t)\leq 0\}} d \conti X(t)+\frac \gamma2 d\conti L^{\conti R}(t),
	\end{multline}
	where in the last equality we used that $\cnz(t)=r(\conti R(t))$ and that
	$\int_0^t\idf_{\{\conti R(s)= 0\}} d \conti X(s)$ is identically zero. Indeed, this stochastic integral has zero mean and 
	\begin{equation}\label{eq:vweivfweiuvfewiu}
		\Var\left[\int_0^t\idf_{\{\conti R(s)= 0\}} d \conti X(s)\right]=\E\left[\int_0^t\idf_{\{\conti R(s)= 0\}} ds\right]=0,
	\end{equation}
	since $\int_0^t\idf_{\{\conti R(s)= 0\}} ds=0$ by \cite[Exercise 3.7.10]{MR1121940}.

	It remains to find a relation between $\conti L^{\conti R}(t)$ and $\conti L^{\cnz}(t)$. Define $\widetilde\sgn(x)\coloneqq\idf_{x>0}-\idf_{x<0}$ (note that $\widetilde\sgn(0)=0$ by definition). By It\^{o}--Tanaka formula (\cref{eq:itotanaka}) and \cref{weiubdieuwbdweiobndqenwd},
	\begin{multline}\label{weuhfweuibfipw}
		|\cnz(t)|=\int_0^t\widetilde\sgn(\cnz(s))d\cnz(s)+\conti L^{\cnz}(t)\\
        =\int_0^t\widetilde\sgn(\cnz(s))(\idf_{\{\cnz(s)> 0\}} d\conti Y(s) - \idf_{\{\cnz(s)\leq 0\}} d \conti X(s))
        +\frac \gamma 2\int_0^t\widetilde\sgn(\cnz(s)) d\conti L^{\conti R}(s)
        +\conti L^{\cnz}(t)\\
		=\int_0^t\widetilde\sgn(\cnz(s))(\idf_{\{\cnz(s)> 0\}} d\conti Y(s) - \idf_{\{\cnz(s)\leq 0\}} d \conti X(s))+\conti L^{\cnz}(t),
	\end{multline}
	where in the last equality we used that $\int_0^t\widetilde\sgn(\cnz(s))d\conti L^{\conti R}(s)=\int_0^t\widetilde\sgn(\conti R(s))d\conti L^{\conti R}(s)=0$ because $\widetilde\sgn(0)=0$ and $\conti L^{\conti R}(s)$ increases only when $\conti R(s)=0$.
	Again by It\^{o}--Tanaka formula (\cref{eq:itotanaka}), setting $\gamma'\coloneqq 1/(1-q)+1/q$, we have that
	\begin{multline}\label{fwuibfwe}
		|r(\conti R(t))|=\int_0^t\widetilde\sgn(r(\conti R(s)))r'(\conti R(s))d\conti R(s)+\frac {\gamma'}{2}\conti L^{\conti R}(t)\\
		=\int_0^t\widetilde\sgn(\cnz(s))(\idf_{\{\cnz(s)> 0\}} d\conti Y(s) - \idf_{\{\cnz(s)\leq 0\}} d \conti X(s))+\frac {\gamma'}{2}\conti L^{\conti R}(t),
	\end{multline}
	where in the last equality we used that $\widetilde\sgn(r(\conti R(s)))=\widetilde\sgn(\cnz(s))$ and \cref{eq:flow_SDE_gen3}.
	Comparing \cref{weuhfweuibfipw,fwuibfwe}, we obtain that $\conti L^{\cnz}(t)=\frac {\gamma'}{2}\conti L^{\conti R}(t)$. Substituting the latter expression in \cref{weiubdieuwbdweiobndqenwd} we can conclude that
	$\cnz$ is a strong solution to \cref{eq:flow_SDE_gen2}.

	\medskip
	
	Assume now that the process $\cnz$ is a strong solution to \cref{eq:flow_SDE_gen2}. Recalling that $r(s(x))=s(r(x))=x$, we show that the process $\conti R$, defined by $\conti R(t)=s(\cnz(t))$ is a strong solution to \cref{eq:flow_SDE_gen3}. By It\^{o}--Tanaka formula (\cref{eq:itotanaka}),
	\begin{equation}
		d\conti R(t)=s'(\cnz(t))d\cnz(t)+\frac{1-2q}{2} d\conti L^{\cnz}(t).
	\end{equation}
	From \cref{eq:flow_SDE_gen2}, using similar arguments as before, we obtain that
	\begin{align}
		d\conti R(t)&=s'(\cnz(t))(\idf_{\{\cnz(t)> 0\}} d\conti Y(t) - \idf_{\{\cnz(t)\leq 0\}} d \conti X(t))+\left((2q-1)s'(0)+\frac{1-2q}{2}\right) d\conti L^{\cnz}(t)\\
		&=(1-q)\idf_{\{\conti R(t)> 0\}} d\conti Y(t) - \frac 1 2 \idf_{\{\conti R(t)= 0\}} d \conti X(t)-q \idf_{\{\cnz(t)< 0\}} d \conti X(t).
	\end{align}
	The latter SDE is equivalent to the SDE in \cref{eq:flow_SDE_gen3} because as before $\int_0^t\idf_{\{\conti R(s)= 0\}} d \conti X(s)$ is identically zero.
\end{proof}

We now move to the proof of \cref{prop:pu_SDE}.

\begin{proof}[Proof of \cref{prop:pu_SDE}]
	Our strategy to show that pathwise uniqueness holds for the SDE
	\begin{equation}\label{fbirefiuowebfouiwebnf}
		d\conti R(t) = (1-q)\idf_{\{\conti R(t)> 0\}} d\conti Y(t) - q\idf_{\{\conti R(t)\leq 0\}} d \conti X(t)
	\end{equation}
	is to apply \cite[Theorem 8.1]{MR3055262} that guarantees pathwise uniqueness in the following setting:
	Let $f : \R \to \R$ be a function of finite variation. Consider two continuous local martingales $(\conti M(t))_{t\in\R_{\geq 0}}$ and $(\conti N(t))_{t\in\R_{\geq 0}}$ started at zero, which almost surely satisfy for some constant $c > 0$,
	\begin{equation}
		\langle \conti M,\conti N \rangle (t)=0, \quad   d\langle \conti M \rangle (t) \leq c\cdot d\langle \conti N \rangle (t),\qquad t\in\R_{\geq 0}.
	\end{equation}
	Then pathwise uniqueness holds for the SDE
	$$d\cnz(t) = f(\cnz(t)) d\conti M(t)+d\conti N(t).$$
	The rest of the proof is devoted to rewriting the SDE in \cref{fbirefiuowebfouiwebnf} in the form described above. We first construct the continuous local martingales $(\conti M(t))_{t\in\R_{\geq 0}}$ and $(\conti N(t))_{t\in\R_{\geq 0}}$.
	It is straightforward to check that defining $x$ such that $\sin(2x)=\rho$, then the process 
	$$
	(1-\rho^2)^{-1/2}
	\begin{pmatrix}
		\cos(x) & -\sin(x) \\
		-\sin(x) & \cos(x) 
	\end{pmatrix}
	\begin{pmatrix}
		\conti X(t) \\
		\conti Y(t) 
	\end{pmatrix}$$ 
	is a standard two-dimensional Brownian motion, in particular, it has uncorrelated coordinates.

	\noindent Now, setting\footnote{The parameters $p,q,t,u$ are chosen so that the process $({\conti M}(t),\widetilde{\conti N}(t))$ defined in \cref{eq:weuobfweobf} is a standard two-dimensional Brownian motion and \cref{diuebiebwdow} holds.}  $\Theta=(1 + 2 ( q -1) q - 
		2 (-1 + q) q \sin(2x))^{-1/2}$ and
	\begin{align}
		p&=(-q \cos(x) + ( q - 1) \sin(x))\Theta,\qquad
		t=((q-1) \cos(x) - q \sin(x))\Theta,\\
		q&=(\cos(x) - q \cos(x) + q \sin(x))\Theta,\qquad	
		u=(-q \cos(x) + (q-1) \sin(x))\Theta,
	\end{align}
	it is simple to check that the matrix $\begin{pmatrix}
	p & t \\
	q & u 
	\end{pmatrix}$ is orthogonal. We finally define $A,B,C,D$ such that $\begin{pmatrix}
	A & B \\
	C & D 
	\end{pmatrix}=(1-\rho^2)^{-1/2}\begin{pmatrix}
		p & t \\
		q & u 
	\end{pmatrix}
	\begin{pmatrix}
		\cos(x) & -\sin(x) \\
		-\sin(x) & \cos(x) 
	\end{pmatrix}$.
	With these definitions, the process 
	\begin{equation}\label{eq:weuobfweobf}
		({\conti M}(t),\widetilde{\conti N}(t))\coloneqq
		\begin{pmatrix}
		A & B \\
		C & D 
		\end{pmatrix}
		\begin{pmatrix}
		\conti X(t) \\
		\conti Y(t) 
		\end{pmatrix}
	\end{equation}
	is a standard two-dimensional Brownian motion. Indeed, it is a linear isometric transformation of another standard two-dimensional Brownian motion. Additionally, it holds that 
	\begin{equation}\label{diuebiebwdow}
		A(1-q)=Bq.
	\end{equation}
	Note that setting $S=(AD-BC)^{-1}$ we have that
	\begin{equation}
		\conti X(t)=S(D{\conti M}(t)-B\widetilde{\conti N}(t))\quad \text{and} \quad \conti Y(t)=S(A\widetilde{\conti N}(t)-C{\conti M}(t)).
	\end{equation}
	Therefore the SDE $d\conti R(t) = (1-q)\idf_{\{\conti R(t)> 0\}} d\conti Y(t) - q\idf_{\{\conti R(t)\leq 0\}} d \conti X(t)$ can be written in the following equivalent form
	\begin{multline}
		d\conti R(t) = S(1-q)\idf_{\{\conti R(t)> 0\}} (A\cdot d\widetilde{\conti N}(t)-C\cdot d{\conti M}(t)) - Sq\idf_{\{\conti R(t)\leq 0\}} (D\cdot d{\conti M}(t)-B\cdot d\widetilde{\conti N}(t))\\
		=S (-C(1-q)\idf_{\{\conti R(t)> 0\}}-qD\idf_{\{\conti R(t)\leq 0\}})d{\conti M}(t)+
		qSBd\widetilde{\conti N}(t),		
	\end{multline}
	where in the last equality we used \cref{diuebiebwdow}. Note that we can write the latter SDE as follows
	\begin{equation}
			d\conti R(t) 
		=f(\conti R(t))d{\conti M}(t)+d{\conti N}(t),
	\end{equation}
	where $f(x)=S (-C(1-q)\idf_{\{x> 0\}}-qD\idf_{\{x\leq 0\}})$ is a bounded variation function and ${\conti N}(t)=qSB\widetilde{\conti N}(t)$. In addition, the continuous martingales ${\conti M}(t)$ and ${\conti N}(t)$ satisfy all properties required to apply \cite[Theorem 8.1]{MR3055262}. Therefore we can conclude that pathwise uniqueness holds for the SDE in \cref{fbirefiuowebfouiwebnf}.
\end{proof}

We finally prove \cref{prop:ex_SDE}.

\begin{proof}[Proof of \cref{prop:ex_SDE}]
	Let $(\conti B(t))_{t\in\R_{\geq 0}}$ be a standard one-dimensional Brownian motion. Consider the SDE
	\begin{equation}\label{ewqiyvduqwvbdipbqwdpibqw}
		d\conti R(t) = ((1-q)\idf_{\{\conti R(t)> 0\}} + q\idf_{\{\conti R(t)\leq 0\}})d \conti B(t).
	\end{equation}
	Since we assumed $q\in(0,1)$, according to a result due to S. Nakao \cite{MR326840} (see also \cite{MR770393,MR777514}), the latter SDE has a unique strong solution.  Note that
	\begin{equation}\label{wefivwevbfoubwef}
		d\langle\conti R \rangle(t) = ((1-q)^2\idf_{\{\conti R(t)> 0\}} + q^2\idf_{\{\conti R(t)\leq 0\}})dt.
	\end{equation}
	We now consider two standard one-dimensional Brownian motions $(\widetilde{\conti X}(t))_{t\in\R_{\geq 0}},$ $(\widetilde{\conti Y}(t))_{t\in\R_{\geq 0}}$ such that
	\begin{equation}\label{fbweuubfpeqbnfpqe}
		\langle\conti B,\widetilde{\conti X}\rangle(t)=\langle\conti B,\widetilde{\conti Y}\rangle(t)=\rho t.
	\end{equation}
        This is possible (for instance) by considering two additional independent Brownian motions $\conti B_1$ and $\conti B_2$, both also independent of $\conti B$, and setting $\widetilde{\conti X}= \rho \conti B+ \sqrt{1-\rho^2} \conti B_1$ and  $\widetilde{\conti Y}= \rho \conti B+ \sqrt{1-\rho^2} \conti B_2$.
	We define the following two processes
	\begin{align}
		\overline{\conti X}(t)&\coloneqq\int_0^t\idf_{\{\conti R(s)> 0\}}d \widetilde{\conti Y}(s)  -\frac 1 q\int_0^t\idf_{\{\conti R(s)\leq 0\}} d{\conti R}(s),\\
		\overline{\conti Y}(t)&\coloneqq\frac{1}{1-q}\int_0^t\idf_{\{\conti R(s)> 0\}}d {\conti R}(s)  -\int_0^t\idf_{\{\conti R(s)\leq 0\}} d\widetilde{\conti X}(s).
	\end{align}
	Note that
	\begin{equation}
		 (1-q)\idf_{\{\conti R(t)> 0\}} d\overline{\conti Y}(t) - q\idf_{\{\conti R(t)\leq 0\}} d\overline{\conti X}(t)=d\conti R(t),
	\end{equation}
	and so the triplet $(\conti R,\overline{\conti X},\overline{\conti Y})$ solves \cref{eq:flow_SDE_gen3}. To show that $(\conti R,\overline{\conti X},\overline{\conti Y})$ is a weak solution to \cref{eq:flow_SDE_gen3}, it remains to show that $(\overline{\conti X},\overline{\conti Y})$ is a two-dimensional Brownian motion with correlation $\rho$. By definition, $\overline{\conti X}$ and $\overline{\conti Y}$ are continuous local martingales. In addition, we have that
	\begin{equation}
		\langle\overline{\conti X}\rangle(t)
		=
		\int_0^t\idf_{\{\conti R(s)> 0\}}ds  
		+\frac{1}{ q^2}\int_0^t\idf_{\{\conti R(s)\leq 0\}} d\langle {\conti R} \rangle (s)
		=
		\int_0^t\idf_{\{\conti R(s)> 0\}}ds  
		+\int_0^t\idf_{\{\conti R(s)\leq 0\}}ds=t,
	\end{equation}
	where in the second inequality we used \cref{wefivwevbfoubwef}. Similarly,
	\begin{equation}
		\langle\overline{\conti Y}\rangle(t)
		=
		\tfrac{1}{(1-q)^2}\int_0^t\idf_{\{\conti R(s)> 0\}}d\langle {\conti R} \rangle (s) +\int_0^t\idf_{\{\conti R(s)\leq 0\}} ds
		=
		\int_0^t\idf_{\{\conti R(s)> 0\}}ds +\int_0^t\idf_{\{\conti R(s)\leq 0\}} ds
		=t.
	\end{equation}
	Therefore, by Lévy's characterization theorem (\cite[Theorem 3.3.16]{MR1121940}), $\overline{\conti X}$ and $\overline{\conti Y}$ are standard one-dimensional Brownian motions. It only remains to check that $\overline{\conti X}$ and $\overline{\conti Y}$ have the desired correlation. Note that
	\begin{equation}\label{frbferiobfwbfpwewfw}
		\langle\overline{\conti X}, \overline{\conti Y}\rangle(t)
		=\frac{1}{1-q}\int_0^t\idf_{\{\conti R(s)> 0\}}d \langle\widetilde{\conti Y}, {\conti R}\rangle(s)  +\frac{1}{q}\int_0^t\idf_{\{\conti R(s)\leq 0\}} d\langle\widetilde{\conti X}, {\conti R}\rangle(s).
	\end{equation} 
	From \cref{ewqiyvduqwvbdipbqwdpibqw} and \cref{fbweuubfpeqbnfpqe}, we have that
	\begin{equation}
		d\langle\widetilde{\conti X}, {\conti R}\rangle(s)=d\langle\widetilde{\conti Y}, {\conti R}\rangle(s)=\rho \left((1-q) \idf_{\{\conti R(s)> 0\}} ds+ q \idf_{\{\conti R(s)\leq0\}} ds\right),
	\end{equation}
	and substituting this expression in \cref{frbferiobfwbfpwewfw} we can conclude that $\langle\overline{\conti X}, \overline{\conti Y}\rangle(t)=\rho t$, as desired.
	
	\medskip
	
	The fact that $r(\conti R(t))_{t\in\R_{\geq 0}}$ is a skew Brownian motion of parameter $q$ follows (for instance) using the same arguments as in \cite[Section 5.2]{MR2280299} and recalling that $\conti R$ is a solution to \cref{ewqiyvduqwvbdipbqwdpibqw}.
\end{proof}

	We complete the proof of \cref{thm:ex_uni_rho_gen}, considering the case $\rho \in (-1,1)$ and $q=0$ (the case $q=1$ follows with similar arguments). We consider the functions $g(x)=x\cdot\mathds{1}_{x\geq 0}$ and $g'(x)=\mathds{1}_{x>0}+\frac 1 2 \cdot\idf_{x=0}$. We assume that $\cnz$ is a strong solution to \cref{eq:flow_SDE_gen2} and we define the process $\conti R$ by $\conti R(t)=g(\cnz(t))$. By It\^{o}--Tanaka formula (\cref{eq:itotanaka}) and \cref{eq:flow_SDE_gen2}, we have that
	\begin{equation}
		d \conti R(t)=g'(\cnz(t))d\cnz(t)+\frac 1 2 d\conti L^{\cnz}(t)=
		 \idf_{\{\conti R(t)> 0\}} d\conti Y(t) - \frac 1 2\idf_{\{\conti Z(t)= 0\}} d \conti X(t).
	\end{equation}
	By definition the process $\conti R$ is non-negative, continuous, and started at zero. The last equation shows that $\conti R$ is also a martingale, and so $\conti R$ is a.s.\ identically zero. This implies that a.s.\ $\cnz(t)\leq 0$ for all $t\in\R_{\geq 0}$ and so $\cnz$ solves the SDE
	\begin{equation}
		d\cnz(t) = - \idf_{\{\cnz(t)\leq 0\}} d \conti X(t)-d\conti L^{\cnz}(t),\quad t\in \R_{\geq 0}.
	\end{equation}
	From \cite{MR606993}, we know that the latter SDE has a unique strong solution 
	$\cnz(t)_{t\in\R_{\geq 0}}$ that is a skew Brownian motion of parameter $0$. This completes the proof of \cref{thm:ex_uni_rho_gen}.
	
\subsubsection{The Brownian excursion case}\label{eq;efbwevfouewvfouew}

Building on \cref{thm:ex_uni_rho_gen} it is straightforward to prove \cref{thm:conj_rho_gen} using absolute continuity arguments. We include the following (short) proof for making the article as self-contained as we can, but we highlight that the arguments are similar to the one used in \cite[Theorem 4.6]{borga2020scaling}.

\begin{proof}[Proof of \cref{thm:conj_rho_gen}]
The strategy is to consider the solution mappings $F_t$ defined in \cref{thm:ex_uni_rho_gen} and for all $r\in(u,1)$ to define the following measurable (with regards to $\mathcal F_r^{(u)}$) process $\conti S_{u,r}\in \mathcal C([u,r])$:
\begin{equation}
	\conti S_{u,r}(t) \coloneqq F_{r-u}\Big((\conti E_\rho({u+s}) - \conti E_\rho({u}))_{s\in[0,r-u]}\Big)(t-u),\quad t\in [u,r].
\end{equation}
Since by \cref{prop:abs_cont}, the Brownian excursion $((\conti E_\rho({u+s}) - \conti E_\rho({u}))_{s\in[0,r-u]}$ is absolutely continuous with regards to a two-dimensional Brownian motion of correlation $\rho$ on $[0,r-u]$, using the items 3-4 in \cref{thm:ex_uni_rho_gen} we have that
\begin{enumerate}
	\item $\conti S_{u,r}$ a.s.\ satisfies \cref{eq:flow_SDE_gen} on interval $[u,r]$;
	\item for $r,r'\in(u,1)$, we have $\conti S_{u,r} = (\conti S_{u,r'})|_{[u,r]}$ almost surely.
\end{enumerate}
In addition, by construction, the map $(u,r,\omega)\mapsto \conti S_{u,r}$ is a measurable function. Now, noting that the two statements above hold simultaneously for all $r,r'\in(u,1)\cap\mathbb Q$, we have that a.s.\ there exists a process $\cnz^{(u)} \in \mathcal C([u,1))$ such that $\cnz^{(u)}|_{[u,r]}=\conti S_{u,r}$ for every $r\in\mathbb (u,1)\cap\mathbb Q$. Hence $\cnz^{(u)}$ is $\mathcal F^{(u)}$-adapted and a.s.\ satisfies \cref{eq:flow_SDE_gen}, proving	the existence of a strong solution.

For pathwise uniqueness, let $\cnz^{(u)},\widetilde{\cnz}^{(u)}$ be two $\mathcal F^{(u)}$-adapted solutions to the SDE in \cref{eq:flow_SDE_gen} and $r<1$. There exist two functionals $G,\widetilde{G} : \mathcal C([u,r]) \to \mathcal C([u,r])$ such that a.s.,
\begin{equation}
	\cnz^{(u)} = G\left((\conti E_\rho({s}) - \conti E_\rho({u}))_{s\in[u,r]}\right)
	\quad
	\text{ and }
	\quad\widetilde {\cnz}^{(u)} = \widetilde{G}\left((\conti E_\rho({s}) - \conti E_\rho({u}))_{s\in[u,r]}\right).
\end{equation}	
Using again the absolute continuity (in the other direction) in \cref{prop:abs_cont}, for a two-dimensional Brownian motion $\conti W_\rho$ of correlation $\rho$, the stochastic processes $G(\conti W_\rho)$ and $\widetilde{G}(\conti W_\rho)$ are two solutions to the SDE in \cref{eq:flow_SDE_gen2}. Since by \cref{thm:ex_uni_rho_gen} pathwise uniqueness holds for \cref{eq:flow_SDE_gen2}, $G(\conti W_\rho)=\widetilde{G}(\conti W_\rho)$ a.s. Therefore, by absolute continuity, $\cnz^{(u)} = \widetilde {\cnz}^{(u)}$ a.s. This ends the proof.
\end{proof}

\subsection{The skew Tanaka equation}\label{sect:skew}

The goal of this section is to prove \cref{thm:fvuwevifview2}.
Fix $q\in[0,1]$. We recall that we want to construct solutions to the following SDEs, defined for all $u\in[0,1]$ by
\begin{equation}\label{eq:Tanaka_2}
	\begin{cases}
		d\cnz^{(u)}(t) = \sgn(\cnz^{(u)}(t)) d \bm e(t)+(2q-1)\cdot d\conti L^{\cnz^{(u)}}(t),& t\in [u,1),\\
		\cnz^{(u)}(t)=0,&  t\in [0,u].
	\end{cases} 
\end{equation}
where $\sgn(x)\coloneqq \idf_{\{x>0\}}-\idf_{\{x\leq0\}}$, $(\bm e(t))_{t\in [0,1]}$ is a one-dimensional Brownian excursion on $[0,1]$ and $\conti L^{\cnz^{(u)}}(t)$ is the symmetric local time at zero of $\cnz^{(u)}$.

\medskip

	\cref{eq:Tanaka_2} when $q=1/2$ and $(\bm e(t))_{t\in [0,1]}$ is replaced by a standard one-dimensional Brownian motion $(\conti B(t))_{t\in \R_{\geq 0}}$ is the well-known Tanaka's SDE:
	\begin{equation}\label{eq:fiwvfouwbofwbh}
		\begin{cases}
			d\cnz^{(u)}(t) = \sgn(\cnz^{(u)}(t)) d \conti B(t),&t\in \R_{> u},\\
			\cnz^{(u)}(t)=0,&  t\in[0,u].
		\end{cases} 
	\end{equation}
	The striking feature of this equation is the absence of pathwise
	uniqueness: solutions cannot be measurable functions of the driving
	process $\conti B$ and must also incorporate additional randomness (see for instance \cite[Example 5.3.5]{MR1121940}). 
	
	Similarly, there is absence of pathwise
	uniqueness also for the following SDEs (for a proof see for instance the discussion between Eq.\ (4) and Eq.\ (6) in \cite{MR2835247}):
	\begin{equation}\label{eq:Tanaka3}
		\begin{cases}
			d\cnz^{(u)}_{q}(t) = \sgn(\cnz_q^{(u)}(t)) d \conti B(t)+(2q-1)\cdot d\conti L^{\cnz_q^{(u)}}(t),& t\in\R_{>u},\\
			\cnz_q^{(u)}(t)=0,&  t\in [0,u].
		\end{cases} 
	\end{equation}
	This absence of pathwise uniqueness (both for the SDEs in \cref{eq:fiwvfouwbofwbh} and \cref{eq:Tanaka3}) raises many questions when we want to couple solutions to \cref{eq:fiwvfouwbofwbh} (or \cref{eq:Tanaka3}) for several starting times $u\in\R$. An elegant solution was developed by Le Jan and Raimond (\cite{MR2060298,MR4112725}) using the notion of \textit{stochastic flow of maps}. The same authors proved in \cite[Theorem 2.1]{MR2235172} that there exists a stochastic flow of maps solving \cref{eq:fiwvfouwbofwbh} and explicitly constructed this flow (which was
	also studied in \cite{MR1816931}). Later, Hajri (\cite[Theorem 2]{MR2835247}) extended the ideas of Le Jan and Raimond and explicitly constructed another stochastic flow of maps solving\footnote{We remark that the construction given in \cite[Theorem 2]{MR2835247} is presented in a much more general setting. Specifically, Hajri considered general flow evolving on graphs and some generalized SDEs (mixing Tanaka’s SDE and the skew Brownian motion SDE). Furthermore, he gave a discrete approximation of this flow in \cite{MR2953347}.
	We do not need to consider this general setting in the present paper. We just mention that our specific case corresponds (following the notation in \cite[Theorem 2]{MR2835247}) to $N=p=2$, $\alpha_1=q$ and $\alpha_2=1-q$ (so that $\alpha^{+}=1$).} \cref{eq:Tanaka3}.
	
	Using our notation, the construction of Hajri
	gives the following solution 	$\left\{\cnz^{(u)}_q(t)\right\}_{u\in\R_{\geq 0}}$ to the SDEs in \cref{eq:Tanaka3}. (We suggest comparing the following explanation with \cref{figure:efwuyvew}.)

	Conditional on $\conti B$,  consider an i.i.d.\ sequence $(\bm s(\ell))_{\ell}$, indexed by the local
	minima of $\conti B$, and with distribution $\P(\bm s(\ell)=+1)=q=1-\P(\bm s(\ell)=-1)$.
	
	For $u,t\in[0,1]$ with $u \leq t$, set $ {\bm m}^{(u)}(t) \coloneqq \inf_{[u,t]} \conti B$ and ${\bm \varepsilon}^{(u)}_{q}(t) \coloneqq \bm s\left(\sup\{r\in[u,t]: \conti B(r) = {\bm m}^{(u)}(t)\}\right)$. 
	Then the solutions $\left\{\cnz^{(u)}_q(t)\right\}_{u\in\R_{\geq 0}}$ are defined as follows 
	\begin{equation}\label{eq:def_process}
		\begin{cases}
			\cnz_{q}^{(u)}(t) \coloneqq (\conti B(t) - 
			{\bm m}^{(u)}(t)) {\bm \varepsilon}_q^{(u)}(t),&t\in \R_{> u},\\
			\cnz_{q}^{(u)}(t)=0,&  t\in[0,u].
		\end{cases} 
	\end{equation}
	
	\begin{figure}[h]
		\centering
			\includegraphics[scale=.75]{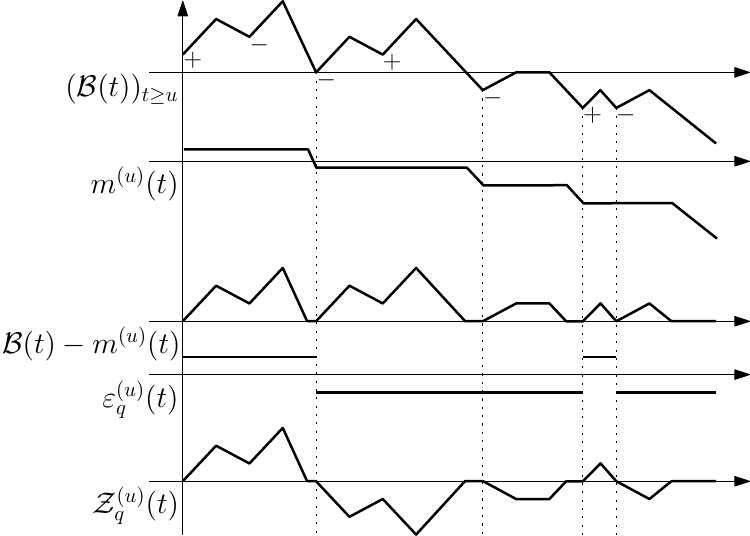}
		\caption{A schema for the processes introduced in \cref{eq:def_process}.\label{figure:efwuyvew}}
	\end{figure}

\begin{rem}
	Note that the solutions $\left\{\cnz^{(u)}_q(t)\right\}_{u\in\R_{\geq 0}}$ defined in \cref{eq:def_process} are constructed using \emph{the same} Brownian motion $\conti B$ and \emph{the same} sequence of i.i.d.\ signs $(\bm s(\ell))_{\ell}$. In particular, there is a coupling between different solutions. 
\end{rem}

\begin{rem}\label{rem:sol_sbm}
	For every fixed $u\in\R_{\geq 0}$, the process $\cnz^{(u)}_q$ defined in \cref{eq:def_process} is a skew Brownian motion of parameter $q$ (and in particular is measurable) as shown in\footnote{We remark again that it is enough to consider the case $N=2$, $\alpha_1=q$ and $\alpha_2=1-q$ in \cite[Proposition 1]{MR2835247} for our purposes.} \cite[Proposition 1]{MR2835247}.
\end{rem}

The fact that the processes $\left\{\cnz_q^{(u)}(t)\right\}_{u\in\R_{\geq 0}}$ defined in \cref{eq:def_process} form a family of solutions to the SDEs in \cref{eq:Tanaka3} is a consequence of the more general result stated in 
\cite[Theorem 2]{MR2835247}, as already mentioned above. Since here we do not need such generality, we include a simple self-contained proof of this result.

\begin{prop}\label{prop:fvuwevifview}
	The family $\left\{\cnz_q^{(u)}(t)\right\}_{u\in\R_{\geq 0}}$ defined in \cref{eq:def_process} is a family of solutions to the SDEs in \cref{eq:Tanaka3}. Moreover, for every $u\in\R_{\geq 0}$, $(\cnz_q^{(u)}(t))_{t\in\R_{\geq u}}$ is a skew Brownian motion of parameter $q$ started at zero at time $u$.
\end{prop}

\begin{proof}
	Fix $u\in \R_{\geq 0}$. Recall from \cref{rem:sol_sbm} that $\cnz^{(u)}_q(t)$ is a skew Brownian motion of parameter $q$. 
	From \cite{MR606993}, there exists a one-dimensional Brownian motion $(\conti W(t))_{t\in \R_{\geq u}}$ started at zero at time $u$ such that
	\begin{equation}\label{eq:djhfevwiufbeqwpidf}
		\cnz_q^{(u)}(t)=\conti W(t)+(2q-1)\conti L^{\cnz_q^{(u)}}(t),\qquad t\in\R_{\geq u}.
	\end{equation}
	From It{\^o}--Tanaka formula (\cref{eq:itotanaka}), we have that
	\begin{align}
		\left|\cnz_q^{(u)}(t)\right|
		&=
		\int_u^t \left(\idf_{\{\cnz_q^{(u)}(s)>0\}}-\idf_{\{\cnz_q^{(u)}(s)<0\}}\right)d\cnz_q^{(u)}(s)
		+\conti L^{\cnz_q^{(u)}}(t)\\
		&=
		\int_u^t\left(\idf_{\{\cnz_q^{(u)}(s)>0\}}-\idf_{\{\cnz_q^{(u)}(s)<0\}}\right)d\conti W(s)
		+\conti L^{\cnz_q^{(u)}}(t),
	\end{align}
	where in the last equality we used \cref{eq:djhfevwiufbeqwpidf} and the fact that 
	$$\int_u^t \left(\idf_{\{\cnz_q^{(u)}(s)>0\}}-\idf_{\{\cnz_q^{(u)}(s)<0\}}\right)d\conti L^{\cnz_q^{(u)}}(s)=0,$$ 
	because  $\conti{L}^{\cnz_q^{(u)}}(s)$ increases only if $\cnz_q^{(u)}(s)=0$.
	Noting that $\int_u^t\idf_{\{\cnz_q^{(u)}(s)=0\}}d\conti W(s)=0$ (this follows for instance using the same arguments as in \cref{eq:vweivfweiuvfewiu}), we obtain that
	\begin{equation}\label{eq:fbiewbfopwenfw}
		\left|\cnz_q^{(u)}(t)\right|=
		\int_u^t\sgn(\cnz_q^{(u)}(s))d\conti W(s)
		+\conti L^{\cnz_q^{(u)}}(t).
	\end{equation}
	Now by construction (see \cref{eq:def_process}) we also have that
	\begin{equation}\label{eq:fhjvewifb}
		\left|\cnz_q^{(u)}(t)\right|= \conti B(t) - {\bm m}^{(u)}(t).
	\end{equation}
    Recall that every continuous semimartingale can be uniquely decomposed into a continuous local martingale and a continuous finite variation process started at zero. Therefore, comparing \cref{eq:fbiewbfopwenfw,eq:fhjvewifb}, and noting that $\conti L^{\cnz_q^{(u)}}(t)$ and $- {\bm m}^{(u)}(t)$ are both increasing (and so with finite variation) continuous processes, we obtain that 
    \begin{equation}
	\conti L^{\cnz_q^{(u)}}(t)= \conti B(u) - {\bm m}^{(u)}(t) \qquad\text{and}\qquad\conti B(t)-\conti B(u)=\int_u^t\sgn(\cnz_q^{(u)}(s))d\conti W(s).
    \end{equation}
    Using the equality on the right-hand side of the last equation, we get
	\begin{equation}
		\int_u^t\sgn(\cnz_q^{(u)}(s))d\conti B(s)=\int_u^t\sgn(\cnz_q^{(u)}(s))^2d\conti W(s)= \int_u^td\conti W(s)=\conti W(t),
	\end{equation}
    where in the last equality we used that $(\conti W(t))_{t\in \R_{\geq u}}$ is a Brownian motion started at zero at time $u$. Substituting the last expression in \cref{eq:djhfevwiufbeqwpidf}, we conclude that $\cnz_q^{(u)}(t)$ is a solution to the SDE in \cref{eq:Tanaka3}.
\end{proof}

The result stated in \cref{thm:fvuwevifview2} follows from \cref{prop:fvuwevifview}
using standard absolute continuity arguments between one-dimensional Brownian excursions and one-dimensional Brownian motions. Since these arguments are very similar to the ones used for the proof of \cref{thm:conj_rho_gen} in \cref{eq;efbwevfouewvfouew}, we skip the details here.

\bigskip

We conclude this section with a quick discussion on uniqueness of the solutions to \cref{eq:Tanaka3} given in \cref{prop:fvuwevifview}.

\begin{rem}\label{rem:uniqueness}
	In the present paper, we decided to avoid using the formalism of stochastic flow of maps and so we did not include any claim in \cref{prop:fvuwevifview} related to uniqueness of the solutions to \cref{eq:Tanaka3} defined in \cref{eq:def_process} (and we made the same choice for \cref{thm:fvuwevifview2} in the introduction).
	Nevertheless, in \cite[Theorem 2]{MR2835247} it is also showed that the stochastic flow of maps corresponding to the processes $\left\{\cnz_q^{(u)}(t)\right\}_{u\in\R_{\geq 0}}$ defined in \cref{eq:def_process} is \emph{unique}, that is, the unique flow adapted to the filtration generated by the driving Brownian motion. We also remark that the proof of this uniqueness result has been later simplified in \cite{MR3314651}.
\end{rem}

\section{The skew Brownian permuton}

In the previous section we investigated solutions to the skew perturbed Tanaka equations driven by a two-dimensional correlated Brownian excursion. Building on these results, we can now deduce that the skew Brownian permuton is well-defined, proving \cref{thm:perm_is_ok}. Later, we will also show that the biased Brownian separable permuton is a particular case of the skew Brownian permuton, proving \cref{thm:Baxt_brow_same}. These are the two goals of the next two sections.

\subsection{The skew Brownian permuton is well-defined}\label{sect:welldef}

We prove here \cref{thm:perm_is_ok}, showing that the skew Brownian permuton $\bm \mu_{\rho,q}$ is well-defined for all $(\rho,q)\in(-1,1]\times[0,1]$. Before doing that we also give a slightly different (but equivalent) definition of the skew Brownian permuton that will be also helpful later to prove \cref{thm:Baxt_brow_same}.

We fix $(\rho,q)\in(-1,1]\times[0,1]$ and consider the continuous coalescent-walk process\footnote{Actually $\cnz_{\rho,q}^{(u)}$ was not defined for $u\in \{0,1\}$ (see \cref{defn:cont_coal_proc}). As what happens on a negligible subset of $[0,1]$ is irrelevant to the arguments to come, this causes no problems.} $\cnz_{\rho,q}=\left\{\cnz^{(u)}_{\rho,q}\right\}_{u\in[0,1]}$. 
We first define a random binary relation $\preccurlyeq_{\cnz_{\rho,q}}$ on $[0,1]^2$ as follows:
\begin{equation}\label{eq:efviwevfiuew}
	\begin{cases}
		t\preccurlyeq_{\cnz_{\rho,q}} t &\text{ for all }t\in [0,1],\\
		t\preccurlyeq_{\cnz_{\rho,q}} s &\text{ for all }0\leq t<s \leq 1\text{ s.t.\ }\ \cnz_{\rho,q}^{(t)}(s)<0,\\
		s\preccurlyeq_{\cnz_{\rho,q}} t,&\text{ for all }0\leq t<s \leq 1\text{ s.t.\ }\ {\cnz}_{\rho,q}^{(t)}(s)\geq0.
	\end{cases}
\end{equation}
We note that the stochastic process $\varphi_{\cnz_{\rho,q}}(t)$ introduce in \cref{eq:random_skew_function} satisfies
\begin{equation}
	\varphi_{\cnz_{\rho,q}}(t)=\Leb\left( \big\{x\in[0,1]|x \preccurlyeq_{\cnz_{\rho,q}} t\big\}\right),\qquad t\in[0,1],
\end{equation}
and we recall that the skew Brownian permuton $\bm \mu_{\rho,q}$ is defined by
\begin{equation}\label{eq:defnskeperm}
	\bm \mu_{\rho,q}(\cdot)\coloneqq(\Id,\varphi_{\cnz_{\rho,q}})_{*}\Leb (\cdot)= \Leb\left(\{t\in[0,1]|(t,\varphi_{\cnz_{\rho,q}}(t))\in \cdot \,\}\right).
\end{equation}
We restrict for the moment to the case $(\rho,q)\in(-1,1)\times[0,1]$ (the case $\rho=1$ will be treated separately in \cref{sect:equiv}).
Using pathwise uniqueness (Item 3 in \cref{thm:conj_rho_gen}) for the SDEs defining the continuous coalescent-walk process $\cnz_{\rho,q}$, it is simple to obtain the following result (for a proof see \cite[Proposition 5.2]{borga2020scaling}).

\begin{prop} \label{prop:orfn3rnfo3rinfi3rnf}
	Fix $(\rho,q)\in(-1,1)\times[0,1]$. There exists a random set $\bm A \subset [0,1]^2$ of a.s.\ zero Lebesgue measure, that is,
	$\P(\Leb(\bm A)=0)=1$, such that the restriction of the relation $\preccurlyeq_{\cnz_{\rho,q}}$ to $[0,1]^2\setminus \bm A$ is a.s.\ a total order.
\end{prop}

We can now prove \cref{thm:perm_is_ok} (using similar ideas as in \cite[Lemma 5.5]{borga2020scaling}).

\begin{proof}[Proof of \cref{thm:perm_is_ok} for  $\rho\neq 1$]
	We start by proving that for $t,s\in(0,1)$ with $t<s$, $\cnz_{\rho,q}^{(t)}(s) \neq 0$ almost surely. Let $\varepsilon>0$ be such that $s<1-\varepsilon<1$. Thanks to \cref{thm:ex_uni_rho_gen}, $(\cnz_{\rho,q}^{(t)}(t+r))_{r\in[0,1-t-\varepsilon]}$ is absolutely continuous with regards to a skew Brownian motion of parameter $q$ on $[0,1-t-\varepsilon]$. Since the time spent at zero by a skew Brownian motion a.s.\ has zero Lebesgue measure, we can conclude that $\cnz_{\rho,q}^{(t)}(s) \neq 0$ a.s.
	
	We can now prove that $\bm \mu_{\rho,q}$ is a permuton. By definition, $\bm \mu_{\rho,q}$ is a random probability measure on the unit square, and its first marginal is a.s.\ uniform (see \cref{eq:defnskeperm}). Therefore, it is enough to verify that also its second marginal is a.s.\ uniform, i.e.\ that  $(\varphi_{\cnz_{\rho,q}})_{*}\Leb=\Leb$ a.s.

	Let $(\bm U_i)_{i\in \Z_{>0}}$ be a sequence of i.i.d.\ uniform random variables on $[0,1]$. We define for $k\in\Z_{\geq 2}$,
	\begin{equation}
		\bm U_{1,k}
		\coloneqq
		\tfrac 1 {k-1}{\#\left\{i\in \llbracket 2,k \rrbracket\middle|\bm U_i\preccurlyeq_{\cnz_{\rho,q}} \bm U_1 \right\}}=\tfrac 1 {k-1}
		\sum_{i\in \llbracket 2,k \rrbracket}\mathds{1}_{\left\{\bm U_i\preccurlyeq_{\cnz_{\rho,q}} \bm U_1\right\}}.
	\end{equation}
	The random variables $\left(\mathds{1}_{\left\{\bm U_i\preccurlyeq_{\cnz_{\rho,q}} \bm U_1\right\}}\right)_{i\in \Z_{\geq 2}}$ are i.i.d. Bernoulli random variables of parameter $\varphi_{\cnz_{\rho,q}}(\bm U_1)$, conditionally on $(\conti E_\rho,\bm U_1)$. Therefore, from the law of large numbers, $\bm U_{1,k}$ a.s.\ converges to 
	$\varphi_{\cnz_{\rho,q}}(\bm U_1)$ as $k$ tends to infinity.
	
	On the other hand, by the exchangeability of the random variables $(\bm U_i)_{i\in \Z_{>0}}$, and using that $\cnz_{\rho,q}^{(t)}(s) \neq 0$ a.s., the random variable $\bm U_{1,k}$ is uniform in $\big\{\frac{0}{k-1},\dots, \frac{k-1}{k-1}\big\},$ conditionally on $\conti E_\rho$. Therefore, $\bm U_{1,k}$ converges in distribution to a uniform random variable on $[0,1]$ as $k$ tends to infinity. 
	Thus we can conclude that $\varphi_{\cnz_{\rho,q}}(\bm U_1)$ is uniform on $[0,1]$ conditionally on $\conti E_\rho$. This concludes the proof.
\end{proof}

We state and prove a final lemma useful for later purposes.

\begin{lem}\label{lem:fbewobfoihf}
	Fix $(\rho,q)\in(-1,1)\times[0,1]$. Almost surely, for almost every $0\leq t<s\leq 1$, 
	either
	$\cnz_{\rho,q}^{(t)}(s)>0$ and $\varphi_{\cnz_{\rho,q}}(s)<\varphi_{\cnz_{\rho,q}}(t)$, or  $\cnz_{\rho,q}^{(t)}(s)<0$ and $\varphi_{\cnz_{\rho,q}}(s)>\varphi_{\cnz_{\rho,q}}(t)$.
\end{lem}

\begin{proof}
	Consider two independent uniform random variables $\bm U$ and $\bm V$ on $[0,1]$, also independent of $\conti E_{\rho}$. It is immediate from \cref{prop:orfn3rnfo3rinfi3rnf} that if $\bm U < \bm V$ and $\cnz_{\rho,q}^{(\bm U)}(\bm V)> 0$ then $\varphi_{\cnz_{\rho,q}}(\bm U)\geq\varphi_{\cnz_{\rho,q}}(\bm V)$ a.s. Similarly, if $\bm U < \bm V$ and $\cnz_{\rho,q}^{(\bm U)}(\bm V)<0$ then $\varphi_{\cnz_{\rho,q}}(\bm U)\leq\varphi_{\cnz_{\rho,q}}(\bm V)$ a.s. The case $\cnz_{\rho,q}^{(\bm U)}(\bm V) = 0$ is a.s.\ excluded by the first part of the previous proof, and the case  $\varphi_{\cnz_{\rho,q}}(\bm U) = \varphi_{\cnz_{\rho,q}}(\bm V)$ is a.s.\ excluded by the fact that $\varphi_{\cnz_{\rho,q}}(\bm U)$ and $\varphi_{\cnz_{\rho,q}}(\bm V)$ are two independent uniform random variables, thanks to the second part of the previous proof. This is enough to complete the proof of the lemma.
\end{proof}

\subsection{The biased Brownian separable permuton is a particular case of the skew Brownian permuton}\label{sect:equiv}

Recalling the construction of the biased Brownian separable permuton from \cref{sect:fvuuvf} and the construction of the skew Brownian permuton presented in \cref{sect:welldef}, to prove \cref{thm:Baxt_brow_same} (and also to complete the proof of \cref{thm:perm_is_ok} when $\rho=1$) it is enough to prove the following result.
\begin{prop}
	Fix a one-dimensional Brownian excursion $(\bm e(t))_{t\in [0,1]}$ on $[0,1]$ and a parameter $p\in[0,1]$. Conditional on $\bm e$, consider an i.i.d.\ sequence $(\bm s(\ell))_{\ell}\in\{+1,-1\}^{\Z_{>0}}$ indexed by the local	minima of $\bm e$ and with distribution $\P(\bm s(\ell)=+1)=p=1-\P(\bm s(\ell)=-1)$. Let $\vartriangleleft_{\widetilde{\bm e},p}$ be the relation defined in \cref{eq:exc_to_perm} constructed from the pair $(\bm e(t),(\bm s(\ell))_{\ell})$ and $\preccurlyeq_{\cnz_{1,1-p}}$ be the relation defined in \cref{eq:efviwevfiuew}, p.\ \pageref{eq:efviwevfiuew}, constructed from the pair $(\bm e(t),(-\bm s(\ell))_{\ell})$ (note that the two relations $\vartriangleleft_{\widetilde{\bm e},p}$ and $\preccurlyeq_{\cnz_{1,1-p}}$ now have a specific coupling).
	
	Then there exists a random set $\bm A \subset [0,1]^2$ of a.s.\ zero Lebesgue measure, i.e.\ $\P(\Leb(\bm A)=0)=1$, such that the relations $\vartriangleleft_{\widetilde{\bm e},p}$ and $\preccurlyeq_{\cnz_{1,1-p}}$ restricted to $[0,1]^2\setminus \bm A$ are the same total order.
\end{prop}

\begin{proof}
	Recall from \cref{sect:fvuuvf} that there exists a random set $\bm A \subset [0,1]^2$ of a.s.\ zero Lebesgue measure such that for every $x,y\in[0,1]^2\setminus \bm A$ with $x<y$ then $\min_{[x,y]}\bm e$ is reached at a unique point which is a strict local minimum. In addition, $\vartriangleleft_{\widetilde{\bm e},p}$ is a random total order on $[0,1]^2\setminus \bm A$.
	  
	Fix now $x,y\in[0,1]^2\setminus \bm A$ with $x<y$ and assume that  $\bm s(\ell)$ is the sign corresponding to the unique local minimum where $\min_{[x,y]}\bm e$ is reached.
	By definition (see \cref{eq:exc_to_perm}, p.\ \pageref{eq:exc_to_perm}),
	\begin{equation}\label{eq:wfvbwegfweohbfowe}
		\begin{cases}
			x\vartriangleleft_{\widetilde{\bm e},p} y, &\quad\text{if}\quad \bm s(\ell)=+1,\\
			y\vartriangleleft_{\widetilde{\bm e},p} x, &\quad\text{if}\quad \bm s(\ell)=-1.\\
		\end{cases}
	\end{equation}
	Now note that since the excursion $\bm e$ has a unique local minimum on the interval $[x,y]$, say at time $t$, then by construction (see \cref{eq:def_process2}) the process $\cnz_{1,1-p}^{(x)}$ is positive at time $y$ if $-s(\ell)=+1$ and negative if $-s(\ell)=-1$, therefore by definition (see \cref{eq:efviwevfiuew}, p.\ \pageref{eq:efviwevfiuew}),
	\begin{equation}\label{eq:wfvbwegfweohbfowe2}
		\begin{cases}
			x\preccurlyeq_{\cnz_{1,1-p}} y, &\quad\text{if}\quad s(\ell)=+1,\\
			y\preccurlyeq_{\cnz_{1,1-p}} x, &\quad\text{if}\quad s(\ell)=-1.\\
		\end{cases}
	\end{equation}  
	Comparing \cref{eq:wfvbwegfweohbfowe,eq:wfvbwegfweohbfowe2} we can conclude that $\vartriangleleft_{\widetilde{\bm e},p}$ and $\preccurlyeq_{\cnz_{1,1-p}}$ restricted to $[0,1]^2\setminus \bm A$ are the same total order.
\end{proof}

\section{The skew Brownian permuton and SLE-decorated Liouville quantum gravity spheres}\label{sect:LQGperm}

This section is devoted to establishing the connection between the skew Brownian permuton and SLE-decorated LQG spheres described in \cref{thm:sbpfromlqg}. 

\medskip

Recall that $({\mathbb C}\cup\{\infty\}, \bm h, \infty)$ and  $(\bm \eta_0, \bm \eta_\theta)$ denote the $\gamma$-LQG sphere and the pair of space-filling SLE$_{\kappa'}$ introduced in \cref{sect:efbweifbewoubf}. 
\cref{thm:sbpfromlqg} immediately follows from the result in \cref{prop:huewvuwevf} below. To state this proposition we need another construction.

For $t\in[0,1]$, let $\conti X_\rho(t)$ be
the $\nu_{\bm h}$-LQG length measure (see \cite[Section 3.3]{gwynne2019mating}) of the left outer boundary of $\bm \eta_0([0, t])$ and $\conti Y_\rho(t)$ be
the $\nu_{\bm h}$-LQG length measure
of the right outer boundary of $\bm\eta_0([0, t])$.
From \cite[Theorem 1.1]{MR4010949} (see also \cite[Theorem 4.10]{gwynne2019mating}) the process $\conti E_\rho(t)=(\conti X_\rho (t),\conti Y_\rho (t))$ defined above has  (up to time reparametrization)  the law of a two-dimensional Brownian excursion of correlation $\rho$ in the non-negative quadrant.
In addition, the process $(\conti E_\rho(t))_{t\in [0,1]}$ a.s.\ determines $(({\mathbb C}\cup\{\infty\}, \bm h, \infty),\bm \eta_0)$ as a curve-decorated quantum surface. 

\begin{prop}\label{prop:huewvuwevf}
	Fix $\gamma\in(0,2)$ and $\theta\in[0,\pi]$. Recall that $\rho=-\cos(\pi\gamma^2/4)$. Let $\conti E_\rho(t)=(\conti X(t),\conti Y(t))$ be the two-dimensional Brownian excursion of correlation $\rho$ defined above from the curve-decorated quantum surface $(({\mathbb C}\cup\{\infty\}, \bm h, \infty),\bm \eta_0)$. Recall that, for $t\in[0, 1]$, $\bm\psi_{\gamma,\theta}(t)\in[0, 1]$ denotes the first time when $\bm\eta_\theta$ hits the point $\bm\eta_0(t)$.
	
	Let $\cnz_{\rho,q}=\left\{\cnz_{\rho,q}^{(u)}\right\}_{u\in\R}$ be the continuous coalescent-walk process driven by $(\conti E_\rho,q)$ and $\varphi_{\cnz_{\rho,q}}$ be the associated stochastic process defined in \cref{eq:random_skew_function}.
	There exist a constant $\overline q=\overline q_\gamma(\theta)\in[0,1]$ such that a.s.\ for almost all $t\in(0,1)$
	\begin{equation}
	\bm\psi_{\gamma,\theta}(t)=\varphi_{\cnz_{\rho,\overline q}}(t).
	\end{equation}
\end{prop}

The proof of \cref{prop:huewvuwevf} builds on the next lemma (see \cref{lem:keylemfinvol} below) whose proof is postponed to the end of the section. We need some additional notation.

For $t\in[0,1]$, let $\hat{\bm\phi}_\theta^{(t)}$ denote the flow line of the vector field $e^{i(\hat{\bm h}/\chi+\theta)}$ started at $\bm \eta_0(t)$. We point out that there is a unique flow line of the vector field $e^{i(\hat{\bm h}/\chi+\theta)}$ started from $\bm \eta_0(t)$ for almost all $t\in[0,1]$ and this is enough for our purposes.\footnote{In particular, there will be two possible choices of the flow line of the vector field $e^{i(\hat{\bm h}/\chi+\theta)}$ started at $\bm \eta_0(t)$ if $\bm \eta_0(t)$ is a double point of $\bm \eta_\theta$.} Finally, we denote by $\bm\phi_\Theta^{(t)}$ the union of the flow lines $\hat{\bm\phi}_\theta^{(t)}$ and $\hat{\bm\phi}_{\theta+\pi}^{(t)}$ (followed in the same direction of $\hat{\bm\phi}_\theta^{(t)}$).
In \cref{fig:schemaLQG} we show the various curves that we are considering and we explain what we mean when we refer to \emph{left} and \emph{right} in \cref{lem:keylemfinvol}.

\begin{figure}[h]
	\centering
	\includegraphics[scale=.63]{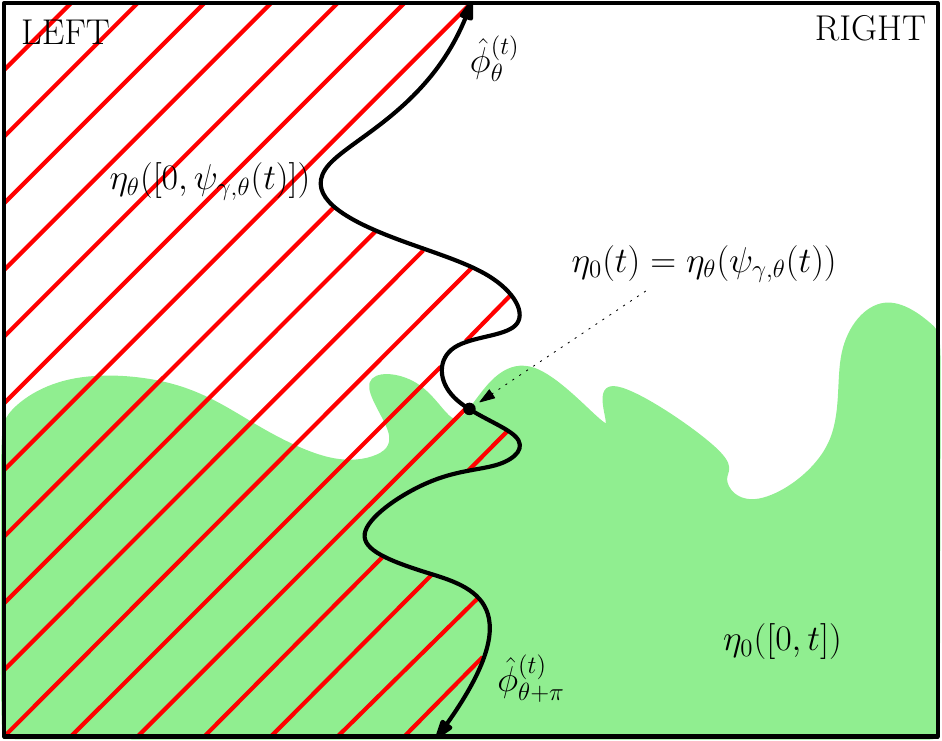}
	\caption{A chart for the various curves considered in this section. In particular, the region covered by $\bm\eta_0([0,t])$ is highlited in green and the region covered by $\bm\eta_\theta([0,\bm\psi_{\gamma,\theta}(t)])$ is highlited in dashed red. By construction, the boundaries of $\bm\eta_\theta([0,\bm\psi_{\gamma,\theta}(t)])$ are given by $\hat{\bm\phi}_\theta^{(t)}$ and $\hat{\bm\phi}_{\theta+\pi}^{(t)}$.
	\emph{Left} and \emph{right} in \cref{lem:keylemfinvol} are defined with respect to the flow line $\bm\phi_\theta^{(t)}=\hat{\bm\phi}_\theta^{(t)}\cup\hat{\bm\phi}_{\theta+\pi}^{(t)}$ followed in the direction of $\hat{\bm\phi}_\theta^{(t)}$. 
	\label{fig:schemaLQG}}
\end{figure}

\begin{lem}\label{lem:keylemfinvol}
	
	There exists a constant $\overline q=\overline q_\gamma(\theta)\in[0,1]$ such that a.s.\ for almost all $t\in(0,1)$
	\begin{align}
		\left\{x\in[t,1):\cnz_{\rho,\overline q}^{(t)}(x)\geq0\right\}&=\left\{x\in[t,1):\bm \eta_0(x) \text{ is weakly on the left of } {\bm\phi}_\theta^{(t)}\right\},\\
		\left\{x\in[t,1):\cnz_{\rho,\overline q}^{(t)}(x)\leq 0\right\}&=\left\{x\in[t,1): \bm \eta_0(x) \text{ is weakly on the right of } {\bm\phi}_\theta^{(t)} \right\}.
	\end{align}
\end{lem}

We can now prove  \cref{prop:huewvuwevf} using \cref{lem:keylemfinvol}.

\begin{proof}[Proof of \cref{prop:huewvuwevf}]
	Recall that $\bm\psi_{\gamma,\theta}(t)\in[0, 1]$ denotes the first time when $\bm\eta_\theta$ hits the point $\bm\eta_0(t)$. Fix $t\in(0,1)$ such that the relations in \cref{lem:keylemfinvol} hold (this is true for almost all $t\in(0,1)$). We want to show that a.s.\	
	$$\bm\psi_{\gamma,\theta}(t)=\Leb\left( \big\{x\in[0,t)|\cnz_{\rho,q}^{(x)}(t)<0\big\} \cup \big\{x\in[t,1]|\cnz_{\rho,q}^{(t)}(x)\geq0\big\} \right).$$
	By definition and time parametrization of $\bm\eta_0$ and $\bm\eta_\theta$, it holds that
	\begin{equation}
		\bm\psi_{\gamma,\theta}(t)=\Leb\left(\left\{x\in[0,1]:\bm \eta_0(x) \text{ is weakly on the left of } {\bm\phi}_\theta^{(t)} \right\}\right).
	\end{equation}
	Noting that for $x\in[0,t)$ the point $\bm \eta_0(x)$ is weakly on the left of $\bm\phi_{\theta}^{(t)}$ if and only if $\bm \eta_0(t)$ is weakly on the right of ${\bm\phi}_{\theta}^{(x)}$, we obtain that	a.s.\
	\begin{align}
		\bm\psi_{\gamma,\theta}(t)
		=&\Leb\left(\left\{x\in[t,1]:\bm \eta_0(x) \text{ is weakly on the left of } {\bm\phi}_\theta^{(t)} \right\}\right)\\
		+& \Leb\left(\left\{x\in[0,t):\bm \eta_0(t) \text{ is weakly on the right of } \bm\phi_{\theta}^{(x)} \right\}\right)\\
		=&\Leb\left(\left\{x\in[t,1]:\cnz_{\rho,\overline q}^{(t)}(x)\geq0\right\}\right)
		+\Leb\left(\left\{x\in[0,t):\cnz_{\rho,\overline q}^{(x)}(t)\leq0 \right\}\right),
	\end{align}
	where in the last equality we used \cref{lem:keylemfinvol}. Finally, from \cref{lem:fbewobfoihf} it holds that
	$$\Leb\left(\left\{x\in[0,t):\cnz_{\rho,\overline q}^{(x)}(t)=0 \right\}\right)=0$$
	and so we can conclude the proof.
\end{proof}

It remains to prove \cref{lem:keylemfinvol}. We first state and prove a result similar to \cref{lem:keylemfinvol} (see \cref{lem:keylem} below) when the pair of space-filling SLEs $( \bm \eta_0, \bm \eta_\theta)$ is parameterized by an LQG cone instead of an LQG sphere.

Let $(\mathbb C,\widetilde{\bm h},0,\infty)$ be a $\gamma$-quantum cone (\cite[Definition 3.10]{gwynne2019mating}) independent of $\hat{\bm h}$. We denote by  $( \widetilde{\bm \eta}_0, \widetilde{\bm \eta}_\theta)$ the pair $( \bm \eta_0, \bm \eta_\theta)$  parametrize  by the $ \mu_{\widetilde{\bm h}}$-LQG area measure so that $\widetilde{\bm\eta}_0(0)=\widetilde{\bm\eta}_\theta(0)=0$ and $\mu_{\widetilde{\bm h}}( \widetilde{\bm\eta}_0([s, t]))= \mu_{\widetilde{\bm h}}(\widetilde{\bm\eta}_\theta([s, t])) = t-s$, for each $s,t\in\R$ with $s<t$. 

For $t\in\R_+$ (resp.\ $t\in\R_-$), let $\widetilde{\conti X}_{\rho}(t)$ be
the net change of the $\nu_{\widetilde{\bm h}}$-LQG length measure of the left outer boundary of $\widetilde{\bm\eta}_0([0, t])$ (resp. $\widetilde{\bm\eta}_0([t,0])$) relative to time 0 and $\widetilde{\conti Y}_{\rho}(t)$ be
the net change of the $\nu_{\widetilde{\bm h}}$-LQG length measure
of the right outer boundary of $\widetilde{\bm\eta}_0([0, t])$ (resp. $\widetilde{\bm\eta}_0([t,0])$)  relative to 0 (see \cite[Section 4.2.1]{gwynne2019mating}).
From \cite[Theorems 1.9 and 1.11]{duplantier2014liouville} (see also \cite[Theorem 4.6]{gwynne2019mating}) the process $\widetilde{\conti W}_\rho(t)=(\widetilde{\conti X}_{\rho} (t),\widetilde{\conti Y}_{\rho}(t))$ defined above has (up to time reparametrization) the law of a two-dimensional Brownian motion of correlation $\rho$.
In addition, the process $(\widetilde{\conti W}_\rho(t))_{t\in \R}$ a.s.\ determines $(({\mathcal C}, \widetilde{\bm h}, -\infty),\widetilde{\bm \eta}_0)$ as a curve-decorated quantum surface. 

For $u\in\R$, let $\widetilde{\bm\phi}_\theta^{(u)}$ denote the union of the two flow lines of the vector fields $e^{i(\hat{\bm h}/\chi+\theta)}$ and $e^{i(\hat{\bm h}/\chi+\theta+\pi)}$ started at $\widetilde{\bm \eta}_0(u)$ (followed in the direction of the flow line of $e^{i(\hat{\bm h}/\chi+\theta)}$). As before these flow lines are unique for almost all $u\in\R$.

\begin{lem}\label{lem:keylem}
	Let $\widetilde{\cnz}_{\rho,q}=\left\{\widetilde{\cnz}_{\rho,q}^{(u)}\right\}_{u\in\R}$ denote the collection of (strong) solutions to the following SDEs indexed by $u\in\R$ and driven by $\widetilde{\conti W}_\rho$,	
		\begin{equation}\label{eq:infcolsde}
			\begin{cases}
				d\widetilde{\cnz}_{\rho,q}^{(u)}(t) = \idf_{\{\widetilde{\cnz}_{\rho,q}^{(u)}(t)> 0\}} d\widetilde{\conti Y}_\rho (t) - \idf_{\{\widetilde{\cnz}_{\rho,q}^{(u)}(t)\leq 0\}} d \widetilde{\conti X}_{\rho}(t)+(2q-1)\cdot d\conti L^{\widetilde{\cnz}_{\rho,q}^{(u)}}(t),& t\in \R_{> u}\\
				\widetilde{\cnz}_{\rho,q}^{(u)}(u)=0.&  
			\end{cases} 
		\end{equation}
	There exists a constant $\overline q=\overline q_\gamma(\theta)\in[0,1]$ such that a.s.\ for almost all $u\in\R$
	\begin{align}
			\left\{t\in \R_{\geq u}:\widetilde{\cnz}_{\rho,\overline q}^{(u)}(t)\geq0\right\}&=\left\{t\in \R_{\geq u}:\widetilde{\bm \eta}_0(t) \text{ is weakly on the left of } \widetilde{\bm\phi}_\theta^{(u)} \right\},\\
			\left\{t\in \R_{\geq u}:\widetilde{\cnz}_{\rho,\overline q}^{(u)}(t)\leq 0\right\}&=\left\{t\in \R_{\geq u}:\widetilde{\bm \eta}_0(t) \text{ is weakly on the right of } \widetilde{\bm\phi}_\theta^{(u)} \right\}.
	\end{align}
\end{lem}

\begin{proof}
	We start by recalling that existence and uniqueness of solutions to the SDE in \cref{eq:infcolsde} are guaranteed by \cref{thm:ex_uni_rho_gen}.
	
	We now fix $u=0$ and we set $\widetilde{\bm\phi}_\theta\coloneqq\widetilde{\bm\phi}_\theta^{(0)}$. We assume that the two flow lines of the vector fields $e^{i(\hat{\bm h}/\chi+\theta)}$ and $e^{i(\hat{\bm h}/\chi+\theta+\pi)}$ started at $\widetilde{\bm \eta}_0(0)$ are unique. The proof for $u\neq 0$ is similar. For $t\in\R_{\geq0}$, we consider the event
	\begin{equation}
		E_t=\left\{\widetilde{\bm \eta}_0(t) \text{ is on the left of } \widetilde{\bm\phi}_\theta\right\},
	\end{equation}
	and the random variable
	\begin{align}
		\widetilde{\bm \tau_t}&=\sup\left\{s\in \R_{\leq t}:  \widetilde{\bm \eta}_0 \text{ crosses }\widetilde{\bm\phi}_\theta \text{ at time } s\right\}.
	\end{align}
 	We also consider the process
 	\begin{equation}\label{wefivwfwevfiouewfwee}
 		\conti Q(t)\coloneqq(\widetilde{\conti Y}_{\rho}(t)-\widetilde{\conti Y}_{\rho}(\widetilde{\bm \tau_t}))\mathds{1}_{E_t}-(\widetilde{\conti X}_{\rho}(t)-\widetilde{\conti X}_{\rho}(\widetilde{\bm \tau_t}))\mathds{1}_{E_t^c},\quad t\in\R_{\geq 0},
 	\end{equation}
	where we recall that  $\widetilde{\conti W}_\rho(t)=(\widetilde{\conti X}_{\rho} (t),\widetilde{\conti Y}_{\rho}(t))$ is the two-dimensional Brownian motion of correlation $\rho$ encoding $(({\mathcal C}, \widetilde{\bm h}, -\infty),\widetilde{\bm \eta}_0)$. In \cite[Proposition 3.2]{gwynne2016joint} it was shown that there exists a constant $\overline q=\overline q_\gamma(\theta)\in[0,1]$ such that $\conti Q(s)$ is a skew Brownian motion of parameter $\overline q$ and 
	\begin{align}
		\left\{t\in\R_{\geq 0}:\bm Q(t)\geq0\right\}&=\left\{t\in\R_{\geq 0}:\widetilde{\bm \eta}_0(t) \text{ is weakly on the left of } \widetilde{\bm\phi}_\theta\right\},\\
		\left\{t\in\R_{\geq 0}:\bm Q(t)\leq0\right\}&=\left\{t\in\R_{\geq 0}:\widetilde{\bm \eta}_0(t) \text{ is weakly on the right of } \widetilde{\bm\phi}_\theta\right\}.\label{ewdbiwuydvuweibvd}
	\end{align}
	Therefore in other to complete the proof it is enough to show that $\conti Q(s)$ solves the SDE in \cref{eq:infcolsde} for $u=0$. Indeed, thanks to pathwise uniqueness (\cref{thm:ex_uni_rho_gen}), then we have that $\conti Q=\widetilde{\cnz}_{\rho,q}^{(0)}$ a.s.
	
	Since $\conti Q(t)$ is a a skew Brownian motion of parameter $\overline q$ then, as shown in \cite{MR606993}, there exists a standard one-dimensional Brownian motion $(\conti B(t))_{t\in\R_{\geq 0}}$ such that 
	\begin{equation}\label{edvwevfouwefboiweo}
		\conti Q(t)=\conti B(t)+(2\overline q-1)\conti L^{\conti Q}(t),\quad t\in\R_{\geq 0}.
	\end{equation}
	Setting 
	\begin{equation}\label{weivfowevfbeopwfew}
		\conti U (t)\coloneqq (1-\overline q)\conti Q(t)\mathds{1}_{\{\conti Q(t)>0\}}+\overline q \conti Q(t)\mathds{1}_{\{ \conti Q(t)\leq0\}},
	\end{equation}
	then from \cite[Equation (11)]{MR606993} we have that
	\begin{equation}\label{fuvfuvwrfvewoufwe}
		d\conti U(t)=(1-\overline q)\mathds{1}_{\{\conti U(t)>0\}}d\conti B(t)+\overline q \mathds{1}_{\{ \conti U(t)\leq0\}}d\conti B(t).
	\end{equation}
	We also introduce the following stochastic process $\overline{\conti W}_{\rho}(t)=(\overline{\conti X}_{\rho}(t),\overline{\conti Y}_{\rho}(t))$, defined for all $t\in\R_{\geq 0}$ by
	\begin{align}\label{efuigwehfweojfew}
		\overline{\conti X}_{\rho}(t)&\coloneqq\int_0^t\mathds{1}_{\{\conti U(s)>0\}} d \widetilde{\conti X}_{\rho}(s)-\int_0^t\mathds{1}_{\{\conti U(s)\leq0\}} d \conti B(s),\\
		\overline{\conti Y}_{\rho}(t)&\coloneqq\int_0^t\mathds{1}_{\{\conti U(s)>0\}} d \conti B(s)+\int_0^t\mathds{1}_{\{\conti U(s)\leq0\}} d \widetilde{\conti Y}_{\rho}(s).
	\end{align}
	Note that 
	\begin{equation}
		(1-\overline q)\mathds{1}_{\{\conti U(t)>0\}}d\overline{\conti Y}_{\rho}(t)-\overline q \mathds{1}_{\{ \conti U(t)\leq0\}}d\overline{\conti X}_{\rho}(t)=(1-\overline q)\mathds{1}_{\{\conti U(t)>0\}}d\conti B(t)+\overline q \mathds{1}_{\{ \conti U(t)\leq0\}}d\conti B(t)=d\conti U(t),
	\end{equation}
	where in the last equality we used \cref{fuvfuvwrfvewoufwe}. Setting, $r(x)=x/(1-\overline q)\cdot\mathds{1}_{x>0}+x/\overline q\cdot\mathds{1}_{x\leq0}$, and using \cref{weivfowevfbeopwfew}, we have that $r(\conti U(t))=\conti Q(t)$. Therefore, from \cref{prop:equiv_SDE}, we obtain that $\conti Q(t)$ satisfies
	\begin{equation}
		d\conti Q(t) = \idf_{\{\conti Q(t)> 0\}} d\overline{\conti Y}_\rho (t) - \idf_{\{\conti Q(t)\leq 0\}} d \overline{\conti X}_{\rho}(t)+(2\overline q-1)\cdot d\conti L^{\conti Q}(t).
	\end{equation}
	Note that if we show that $\overline{\conti X}_\rho=\widetilde{\conti X}_\rho$ a.s.\ and $\overline{\conti Y}_\rho=\widetilde{\conti Y}_\rho$ a.s., i.e.\ $\overline{\conti W}_\rho=\widetilde{\conti W}_\rho$ a.s., then we can conclude the proof.
	From \cref{efuigwehfweojfew,weivfowevfbeopwfew} we have that
	\begin{equation}\label{ebcuiewrcbouierwbcew}
		\overline{\conti X}_\rho(t)-\overline{\conti X}_\rho(\widetilde{\bm \tau_t})=\int_{\widetilde{\bm \tau_t}}^t\mathds{1}_{\{\conti Q(s)>0\}} d \widetilde{\conti X}_{\rho}(s)-\int_{\widetilde{\bm \tau_t}}^t\mathds{1}_{\{\conti Q(s)\leq0\}} d \conti B(s).
	\end{equation}
	Since $\int_{0}^t\mathds{1}_{\{\conti Q(s)=0\}} d \conti B(s)$ is identically zero (this follows for instance using the same arguments as in \cref{eq:vweivfweiuvfewiu}) and $d \conti L^{\conti Q}(s)=0$ when $\conti Q(s)<0$, using \cref{edvwevfouwefboiweo} we have that 
	$$\int_{\widetilde{\bm \tau_t}}^t\mathds{1}_{\{\conti Q(s)\leq0\}} d \conti B(s)=\int_{\widetilde{\bm \tau_t}}^t\mathds{1}_{\{\conti Q(s)<0\}} d \conti Q(s).$$ 
	In addition, noting that $\{\conti Q(s)<0\}\subseteq E_s^c$ (this follows form \cref{ewdbiwuydvuweibvd}) and using \cref{wefivwfwevfiouewfwee} we obtain that 
	$$\int_{\widetilde{\bm \tau_t}}^t\mathds{1}_{\{\conti Q(s)<0\}} d \conti Q(s)=-\int_{\widetilde{\bm \tau_t}}^t\mathds{1}_{\{\conti Q(s)<0\}} d \widetilde{\conti X}_{\rho}(s).$$ 
	Substituting the latter two expressions in \cref{ebcuiewrcbouierwbcew} we conclude that for all $t\in\R_{\geq 0},$
	\begin{equation}\label{eq:webfwevfiuoewvfoew}
		\overline{\conti X}_\rho(t)-\overline{\conti X}_\rho(\widetilde{\bm \tau_t})=\widetilde{\conti X}_\rho(t)-\widetilde{\conti X}_\rho(\widetilde{\bm \tau_t}).
	\end{equation}
	Similarly, $\overline{\conti Y}_\rho(t)-\overline{\conti Y}_\rho(\widetilde{\bm \tau_t})=\widetilde{\conti Y}_\rho(t)-\widetilde{\conti Y}_\rho(\widetilde{\bm \tau_t})$ for all $t\in\R_{\geq 0}$.
	From \cref{efuigwehfweojfew} and Lévy's characterization theorem (\cite[Theorem 3.3.16]{MR1121940}), we have that both $\overline{\conti X}_\rho(t)$ and $\overline{\conti Y}_\rho(t)$ are standard one-dimensional Brownian motions. Therefore, to conclude that $\overline{\conti W}_\rho=\widetilde{\conti W}_\rho$ a.s., it is enough to show that for $t>s$
	\begin{equation}\label{eq:erergger}
		\E\left[\overline{\conti W}_\rho(t)-\widetilde{\conti W}_\rho(t)\middle |\mathcal{F}_s\right]=\overline{\conti W}_\rho(s)-\widetilde{\conti W}_\rho(s),
	\end{equation}
	where $\mathcal{F}_s\coloneqq \sigma((\overline{\conti W}_\rho,\widetilde{\conti W}_\rho)|_{[0,s]})$.
	Indeed, the latter equation implies that $\overline{\conti W}_\rho-\widetilde{\conti W}_\rho$ is a $\mathcal{F}_t$-martingale and then in \cite[Lemma 3.17]{gwynne2016joint} it is shown that if \cref{eq:webfwevfiuoewvfoew} holds then the quadratic variation of $\overline{\conti W}_\rho-\widetilde{\conti W}_\rho$ must be zero and so $\overline{\conti W}_\rho=\widetilde{\conti W}_\rho$ a.s.
	
	\medskip

We proceed with the proof of \cref{eq:erergger} by showing that both $\overline{\conti W}_\rho$ and $\widetilde{\conti W}_\rho$ are martingales with respect to $(\mathcal F_t)_{t\ge0}$.

We set $\widetilde{\mathfrak{W}}_\rho(t) \coloneqq \widetilde{\conti W}_\rho(t) - \widetilde{\conti W}_\rho(\widetilde{\bm \tau_t})$. From \cite[Proposition 3.4 and Lemma 3.16]{gwynne2016joint}, we have that:
\begin{itemize}
	\item the processes $\widetilde{\mathfrak{W}}_\rho|_{[0,s]}$ and $\widetilde{\conti W}_\rho|_{[0,s]}$ a.s.\ determine each other;
	\item $\widetilde{\mathfrak{W}}_\rho|_{[0,s]}$ a.s.\ determines $\conti Q|_{[0,s]}$.
\end{itemize}

Let $(\mathcal G_t)_{t\ge0}$ denote the usual augmentation of the filtration generated by $\widetilde{\conti W}_\rho$. We first show that $\conti Q$ is a skew Brownian motion with parameter $q$ with respect to this filtration.
	
Fix $s\ge0$ and set
\[
	T_s\coloneqq\inf\{t\ge s:\conti Q(t)=0\}.
\]
Since $\conti Q$ is continuous and $\mathcal G_t$-adapted, $T_s$ is a $\mathcal G_t$-stopping time. By the definition of $\conti Q$, for $s\le t\le T_s$,
\[
	\conti Q(t)-\conti Q(s)
	=
	\mathbf 1_{\{\conti Q(s)>0\}}
	\bigl(
	\widetilde{\conti Y}_\rho(t)
	-
	\widetilde{\conti Y}_\rho(s)
	\bigr)
	-
	\mathbf 1_{\{\conti Q(s)<0\}}
	\bigl(
	\widetilde{\conti X}_\rho(t)
	-
	\widetilde{\conti X}_\rho(s)
	\bigr).
\]
Thus, conditionally on $\mathcal G_s$, the process $\conti Q$ up to $T_s$ has the law of a standard Brownian motion started from $\conti Q(s)$ and stopped upon hitting zero. Moreover, \cite[Lemma 3.15]{gwynne2016joint}, as used in the proof of \cite[Lemma 3.16]{gwynne2016joint}, implies that $\conti Q(T_s+\cdot)$ is independent of $\mathcal G_{T_s}$ and has the unconditional law of $\conti Q$. Consequently, conditionally on $\mathcal G_s$, $\conti Q(s+\cdot)$ has the law of a skew Brownian motion with parameter $q$ started from $\conti Q(s)$.
	
By additivity of local time and \cref{edvwevfouwefboiweo}, it follows that the process $\conti B$ appearing there satisfies
\[
	\mathbb E\!\left[
	\conti B(t)-\conti B(s)
	\,\middle|\,
	\mathcal G_s
	\right]=0,
	\qquad 0\le s<t.
\]
Since $\conti B$ is continuous and has quadratic variation $[\conti B]_t=t$, L\'evy's characterization shows that $\conti B$ is a Brownian motion with respect to $(\mathcal G_t)_{t\ge0}$.

Now from the definition of $\overline{\conti W}_{\rho}$ in \cref{efuigwehfweojfew} and the relations in \cref{edvwevfouwefboiweo,weivfowevfbeopwfew}, we also have that:
\begin{itemize}
	\item the process $\overline{\conti W}_{\rho}|_{[0,s]}$ is a.s.\ determined by $(\widetilde{\conti W}_{\rho},\conti Q)|_{[0,s]}$. Therefore, from the items above, $\overline{\conti W}_{\rho}|_{[0,s]}$ is a.s.\ determined by $\widetilde{\conti W}_{\rho}|_{[0,s]}$. In particular, $\mathcal F_s=\mathcal G_s$ up to completion.
\end{itemize}

The stochastic integrals defining $\overline{\conti W}_{\rho}$ in \cref{efuigwehfweojfew} have bounded predictable integrands and Brownian integrators with respect to the common filtration $(\mathcal G_t)_{t\ge0}$. Hence both $\overline{\conti W}_{\rho}$ and $\widetilde{\conti W}_{\rho}$ are square-integrable $\mathcal G_t$-martingales. Since $\mathcal F_s=\mathcal G_s$ up to completion, we obtain
\[
	\mathbb E\!\left[
	\overline{\conti W}_{\rho}(t)
	-
	\widetilde{\conti W}_{\rho}(t)
	\,\middle|\,
	\mathcal F_s
	\right]
	=
	\overline{\conti W}_{\rho}(s)
	-
	\widetilde{\conti W}_{\rho}(s),
	\qquad 0\le s<t.
\]
This proves \cref{eq:erergger}. This ends the proof.
\end{proof}

It remains to deduce \cref{lem:keylemfinvol} from \cref{lem:keylem}.

\begin{proof}[Proof of \cref{lem:keylemfinvol}]
	Note that the only difference between the pairs $( \widetilde{\bm \eta}_0, \widetilde{\bm \eta}_\theta)$ and $( \bm \eta_0, \bm \eta_\theta)$ is in their time parametrization. More precisely, 
	\begin{itemize}
		\item $( \widetilde{\bm \eta}_0, \widetilde{\bm \eta}_\theta)$ are parametrized using the $\gamma$-quantum cone $(\mathbb C,\widetilde{\bm h},0,\infty)$ and therefore the left and right boundary measures of $\widetilde{\bm \eta}_0$ are encoded by the Brownian motion $\widetilde{\conti W}_\rho(t)=(\widetilde{\conti X}_{\rho} (t),\widetilde{\conti Y}_{\rho}(t))$ of correlation $\rho$;	
		
		\item $(\bm \eta_0, \bm \eta_\theta)$ are parametrized using the $\gamma$-LQG sphere $({\mathbb C}\cup\{\infty\}, \bm h, \infty)$  and therefore the left and right boundary measures of $\bm \eta_0$ are encoded by the  Brownian excursion $\conti E_\rho(t)=(\conti X_{\rho} (t),\conti Y_{\rho}(t))$  of correlation $\rho$.
	\end{itemize} 
	For any $\varepsilon > 0$, we consider the following  curve-decorated quantum surfaces:
	\begin{itemize}
		\item the curve-decorated quantum surface $(\widetilde{\mathcal S}_\eps,\widetilde{\bm \eta}_0|_{[\eps , 1-\eps]})$, where $\widetilde{\mathcal S}_\eps$ denotes the quantum surface obtained by restricting the quantum cone field $\widetilde{\bm h}$ to $\widetilde{\bm \eta}_0([\eps , 1-\eps])$.
		\item the curve-decorated quantum surface $(\mathcal S_\eps,\bm \eta_0|_{[\eps , 1-\eps]})$, where $\mathcal S_\eps$ denotes the quantum surface obtained by restricting the quantum sphere field $\bm h$ to $\bm \eta_0([\eps , 1-\eps])$.
	\end{itemize}

	From \cite[Lemma 3.12]{gwynne2016joint} the curve-decorated quantum surface $(\widetilde{\mathcal S}_\eps,\widetilde{\bm \eta}_0|_{[\eps , 1-\eps]})$ is a.s.\ determined by $(\widetilde{\conti W}_\rho (\eps+t) - \widetilde{\conti W}_\rho(\eps))_{0\leq t\leq 1-2\eps}$, while the curve-decorated quantum surface $(\mathcal S_\eps,\bm \eta_0|_{[\eps , 1-\eps]})$ is a.s.\ determined by $(\conti E_\rho (\eps+t) - \conti E_\rho(\eps))_{0\leq t\leq 1-2\eps}$.

	The law of $(\conti E_\rho (\eps+t) - \conti E_\rho(\eps))_{0\leq t\leq 1-2\eps}$ is absolutely continuous w.r.t.\ the law of $(\widetilde{\conti W}_\rho (\eps+t) - \widetilde{\conti W}_\rho(\eps))_{0\leq t\leq 1-2\eps}$ (see \cref{prop:abs_cont}).
	This implies that curve-decorated quantum surface $(\mathcal S_\eps,\bm \eta_0|_{[\eps , 1-\eps]})$ is absolutely continuous w.r.t.\ the  curve-decorated quantum surface $(\widetilde{\mathcal S_\eps},\widetilde{\bm \eta}_0|_{[\eps , 1-\eps]})$.
	
	From \cite[Lemma 3.10]{gwynne2016joint} it follows that the flow lines $\left\{\widetilde{\bm \phi}^{(t)}_\theta\right\}_{t\in[\eps,1-\eps]}$ (resp.\ $\left\{{\bm \phi}^{(t)}_\theta\right\}_{t\in[\eps,1-\eps]}$) run until they exit $\widetilde{\bm \eta}_0([\eps , 1-\eps])$ (resp.\ $\bm \eta_0([\eps , 1-\eps])$) are a.s.\ determined by $\widetilde{\mathcal S}_\eps$ (resp. $\mathcal S_\eps$).
	
	Finally, since by \cite[Lemma 3.6]{gwynne2016joint}, $\widetilde{\bm \eta}_0$ (resp.\ $\bm \eta_0$) hits points on $\widetilde{\bm \phi}^{(t)}_\theta$ (resp.\ $\bm \phi^{(t)}_\theta$) in chronological order, $\widetilde{\bm \phi}^{(t)}_\theta$ (resp.\ $\bm \phi^{(t)}_\theta$) cannot revisit $\widetilde{\bm \eta}_0([\eps , 1-\eps])$ (resp.\ $\bm \eta_0([\eps , 1-\eps])$) after exiting this region. Hence it is possible to determine from $\widetilde{\mathcal S}_\eps$ (resp. $\mathcal S_\eps$) what points of $\widetilde{\bm \eta}_0$ (resp.\ $\bm \eta_0$) are to the left or right of $\widetilde{\bm \eta}_\theta$ (resp.\ $\bm \eta_\theta$). 
	
	\medskip
	
	On the other hand, thanks to \cref{thm:conj_rho_gen} (resp.\ \cref{thm:ex_uni_rho_gen}) the processes $\left\{\cnz_{\rho,\overline q}^{(t)}|_{[t,1-\eps]}\right\}_{t\in[\eps,1-\eps]}$ in the statement of \cref{lem:keylemfinvol} (resp. $\left\{\widetilde{\cnz}_{\rho,\overline q}^{(t)}|_{[t,1-\eps]}\right\}_{t\in[\eps,1-\eps]}$ in the statement of \cref{lem:keylem})  are a.s.\ determined -- through the same solution map $F_{1-\varepsilon}$ -- by $(\conti E_\rho (\eps+t) - \conti E_\rho(\eps))_{0\leq t\leq 1-2\eps}$ (resp.\ $(\widetilde{\conti W}_\rho (\eps+t) - \widetilde{\conti W}_\rho(\eps))_{0\leq t\leq 1-2\eps}$).
	
	\medskip
	
	Therefore, by absolute continuity, we can deduce from \cref{lem:keylem} that a.s.\ for almost all $t\in[\eps,1-\eps]$
	\begin{align}
		\left\{x\in[t,1-\eps]:\cnz_{\rho,\overline q}^{(t)}(x)\geq0\right\}&=\left\{x\in[t,1-\eps]:\bm \eta_0(x) \text{ is weakly on the left of } \bm\phi_\theta^{(t)}\right\},\\
		\left\{x\in[t,1-\eps]:\cnz_{\rho,\overline q}^{(t)}(x)\leq 0\right\}&=\left\{x\in[t,1-\eps]: \bm \eta_0(x) \text{ is weakly on the right of } \bm\phi_\theta^{(t)} \right\}.
	\end{align}
	Since $\varepsilon > 0$ can be made arbitrarily small, this proves \cref{lem:keylemfinvol}.
\end{proof}

\appendix

\section{Absolute continuity between correlated Brownian excursions in cones and correlated Brownian motions}\label{sect:abs_cont}

Let $\conti{W}_\rho$ be a two-dimensional Brownian motion of correlation $\rho\in(-1,1)$ and $\conti{E}_\rho$ a two-dimensional Brownian excursion of correlation $\rho$.

 \begin{prop}\label{prop:abs_cont}
	For every $\eps>0$, the distribution of $(\conti{E}_\rho (\eps+t) - \conti{E}_\rho(\eps))_{t\in[0, 1-2\eps]}$ is absolutely continuous w.r.t.\ the distribution of $(\conti W_{\rho}(t))_{t\in[0,1-2\eps]}$. 
	In particular, for every $0<\eps<1/2$ and for every integrable function  $h:\mathcal{C}([0,1-2\eps],\R^2)\to \R$,
	\begin{equation}\label{eq:fwiobfuowebf}
		\E\left[h\left((\conti{E}_\rho (\eps+t) - \conti{E}_\rho(\eps))_{t\in[0, 1-2\eps]}\right)\right]
		= \E\left[
		h\left((\conti W_{\rho}(t))_{t\in[0,1-2\eps]}\right)
		\alpha_{\eps}\left(-\inf_{[0,1-2\eps]} \conti W_\rho\; ,\; \conti W_\rho(1-2\eps)\right)\right],
	\end{equation}
	where $\alpha_{\eps}$ is a bounded positive continuous function on $(\R_+)^2\times \R^2$. 
	
	In addition, since $\alpha_{\eps}>0$, we have that the two measures are equivalent.
\end{prop}

The theorem above was proved in \cite[Proposition A.1]{borga2020scaling} in the specific case when $\rho=-1/2$ building on some specific results on a family of discrete two-dimensional walks called \emph{tandem walk} \cite{MR4219068}. In what follows we prove the general case $\rho\in(-1,1)$ building on the more general results of \cite{MR3342657,MR4102254}.

\begin{proof}
	We prove the proposition by considering a two-dimensional random walk  $(\bm W_n)_{n\in\Z_{\geq 0}}$.
	Setting $(\bm X,\bm Y)=\bm W_1-\bm W_0$, we assume that $\bm X $ and $\bm Y$ have finite moments, $\E[\bm X] =\E[\bm Y]= 0$, $\E[\bm X^2] =\E[\bm Y^2]= 1$ and $\Cov(\bm X, \bm Y) = \rho$. Under these assumptions, \cite[Theorem 4]{MR4102254} guarantees\footnote{The results in \cite{MR4102254} are stated under the assumption that $\Cov(\bm X, \bm Y) =0$. This does not restrict the generality of the results and they remain valid (with the obvious adaptations) when $\Cov(\bm X, \bm Y) =\rho$. See the bottom part of page 996 in \cite{MR3342657} for more explanations on this fact.}
	 that setting $Q\coloneqq \Z_{\geq 0}^2$,
	\begin{equation}\label{fkbewivfoueqwbfwe}
		\P_x\left(\left(\tfrac 1 {\sqrt{n}} \bm W_{\lfloor nt \rfloor}\right)_{t\in[0, 1]} \in \cdot \;
		\middle| \; \bm W_{[0,n]}\subset Q, \bm W_n = y\right)
		\xrightarrow[n\to\infty]{} \P((\conti{E}_\rho(t))_{t\in[0, 1]} \in \cdot),
	\end{equation}
	 for all $x,y\in Q$.
	 In \cite[Lemma A.2]{borga2020scaling} it was proved\footnote{Note that the proof of \cite[Lemma A.2]{borga2020scaling} does not use any specific property of the tandem walk $(\bm W_n)_{n\in\Z_{\geq 0}}$.} that if  $h:(\Z^2)^{n-2m+1}\to\R$ is a bounded measurable function, $x,y\in Q$ and $1\leq m<n/2$, then
	 \begin{multline}\label{fhuvfihwevofubwfwe}
	 	\E_{x}\left[h((\bm W_{i+m}-\bm W_{m})_{0\leq i \leq n-2m}) \mid \bm W_{[0,n]}\subset Q, \bm W_{n} = y\right]\\
	 	= \E_{0}\left[
	 	h(\bm W_{i})_{0\leq i \leq n-2m}\cdot
	 	\alpha_{n,m}^{x,y}\left(-\inf_{0\leq i \leq n-2m} \bm W_{i}\; ,\;\bm W_{n-2m}\right)
	 	\right],
	 \end{multline}
	 where
	 \begin{equation}\label{eq:bfwowbfwubfowiebf}
	 	\alpha_{n,m}^{x,y}(a,b) \coloneqq 
	 	\sum_{z\in Q \colon z-a \in Q} \frac{
	 		\P_x\left(\bm W_{m} = z, \bm{W}_{[0,m]}\subset Q\right)
	 		\P_{\widehat y}\left(\bm W_{m} = \wzwb, \bm{W}_{[0,m]}\subset Q\right)
	 	}{\P_{x}\left(\bm W_{n} = y, \bm{W}_{[0,n]}\subset Q\right)}
	 \end{equation}
 	and $\widehat{(i,j)} = (j,i)$.
 	Note that if we show that for fixed $x,y\in Q$ and for all $\eps\in(0,1/2)$, there exists a bounded positive continuous function $\alpha_\eps$ on $(\R_+)^2\times \R^2$ such that
 	\begin{equation}\label{hvrfvwjobfowebfpiwe}
 		\lim_{n\to\infty} \sup_{a\in\Z^2_{\geq 0} ,b\in \Z^2} \left\lvert\alpha^{x,y}_{n,\lfloor n\eps\rfloor}(a,b) - \alpha_\eps\left(\tfrac a{\sqrt {n}},\tfrac b{\sqrt {n}}\right)\right\rvert = 0,
 	\end{equation}
 	 then \cref{eq:fwiobfuowebf} follows from \cref{fkbewivfoueqwbfwe,fhuvfihwevofubwfwe,hvrfvwjobfowebfpiwe}. Therefore it just remains to prove \cref{hvrfvwjobfowebfpiwe}.
 	 
 	 \medskip
 	 
 	 Fix $x\in Q$. From \cite{MR3342657}, there exists a positive function $V$ on $Q$ such that as $n\to\infty$ the following asymptotics hold
 	 \begin{gather}
 	 	\mathbb{P}_x\left(\bm W_{[0,n]} \subset Q\right) \sim c_1 V(x) n^{-p / 2}\text{ as }n \rightarrow \infty, \label{eq:LLTconditioning}\\
	 	\delta_1(x,n) \coloneqq \sup _{y \in Q}\left|n^{p/2+1} \cdot \mathbb{P}_x\left(\bm W_n=y, \bm W_{[0,n]} \subset Q\right)-c_2 V(x) g\left(\frac{y}{\sqrt{2n}}\right)\right| \rightarrow 0, 	\label{eq:iwrfbwruobfipenfwipnefw}\\
 	 	\mathbb{P}_x\left(\bm W_n=y, \bm W_{[0,n]} \subset Q\right) \sim c_3\cdot \frac{V(x) V(\widehat y)}{n^{p+1}}, \label{eq:obewfbweoifbweipnfipwen}
 	 \end{gather}
 	 where $c_1,c_2,c_3$ are three constants, $g$ is a positive bounded integrable function on $\R^2_{\geq 0}$ and $p$ is a parameter depending only on $\rho$ and $Q$. (We highlight that all these constants and the function $g$ can be explicitly computed in specific cases, see for instance \cite[Lemma A.5]{borga2020scaling}.) More precisely, \cref{eq:LLTconditioning} is \cite[Theorem 1]{MR3342657}, \cref{eq:iwrfbwruobfipenfwipnefw} is \cite[Theorem 5]{MR3342657}, and \cref{eq:obewfbweoifbweipnfipwen} is \cite[Theorem 6]{MR3342657}.

	In what follows, $m = \lfloor n \eps \rfloor$ for some $\eps>0$. Let us consider $\alpha_{n,m}^{x,y}(a,b)$ defined in \cref{eq:bfwowbfwubfowiebf}.
	By \cref{eq:obewfbweoifbweipnfipwen}, the denominator (which is independent of $a,b$) is of order $n^{-p-1}$.
	
	We now look at the numerator of $\alpha_{n,m}^{x,y}(a,b)$. We first cut the sum at $t\sqrt{n}$ for some $t>0$ and bound the rest of the sum. Using \cref{eq:iwrfbwruobfipenfwipnefw} for the first factor (recall that $g$ is bounded) and \cref{eq:LLTconditioning} for the second one, we can guarantee that there exists a constant $C>0$ depending only on $x,y$ such that
	\begin{align}
		R_n\coloneqq &\sum_{|z|>t\sqrt{n}} 
		\P_x(\bm W_{m} = z, \bmwzm \subset Q)
		\P_{\widehat y}(\bm W_m = \wzwb, \bmwzm\subset Q)\\
		&\leq C n^{-p/2}n^{-p/2-1}\sum_{|z|>t\sqrt n} \P_{x}(\bm W_{m} = z \mid \bmwzm \subset Q)\\
		&= C n^{-p-1}\P_{x}(|\bm W_{ \lfloor n \eps \rfloor}| > t\sqrt{n} \mid \bm W_{[0, \lfloor n \eps \rfloor]} \subset Q).
	\end{align}
	Using \cite[Theorem 3]{MR3342657}, it is possible to find a function $\dxyent$ independent of $a,b$ such that
	\begin{equation}\label{eq:eibdiuwebdowebdbiew}
		n^{p+1}R_n \leq \dxyent\qquad \text{and}\qquad 
		\lim_{t\to\infty} \limsup_{n\to\infty} \dxyent = 0 .
	\end{equation}
	Now set $S_{n,m} \coloneqq \sum_{z:z-a\in Q ,|z|\leq t\sqrt n} 
		\P_x(\bm W_m = z, \bmwzm \subset Q)
		\cdot \P_{\widehat y}(\bm W_m = \wzwb, \bmwzm \subset Q).$
	Using \cref{eq:iwrfbwruobfipenfwipnefw}, we have for fixed $x$ and $y$ that
	\begin{multline}
		S_{n,m} = m^{-p-2}\cdot c_2^2\cdot V(x)V(\widehat y)\sum_{z:z-a\in Q ,|z|\leq t\sqrt n}  
		g\left(\frac{z}{\sqen}\right)g\left(\frac{\widehat z+\widehat b}{\sqen}\right) \\
		+ O(1) \left(t\sqrt n\right)^2 m^{-p-2}(\delta_{1}( x , m ) + \delta_{1}(\widehat y , m )).
	\end{multline}	
	Putting together the latter estimate for the numerator of \cref{eq:bfwowbfwubfowiebf} with the estimate in \cref{eq:obewfbweoifbweipnfipwen} for the denominator, both uniform in $(a,b)$, we have
	\begin{align}
		&\alpha^{x,y}_{n,\lfloor \eps n \rfloor}(a,b) \\
		&= O(1)n^{p+1}R_n + o(1) + \frac{1}{\eps^{p+2}}\left(\frac{c_2^2}{c_3}+o(1)\right) \frac 1 {(\sqrt n)^2} \sum_{z \geq a,|z|\leq t\sqrt n}  g\left(\frac{z}{\sqen}\right)g\left(\frac{\widehat z+\widehat b}{\sqen}\right)\\
		&= O(1)n^{p+1}R_n + o(1) + \frac{1}{\eps^{p+2}}\left(\frac{c_2^2}{c_3}+o(1)\right)\int_{w\geq \frac a {\sqrt n}, |w|\leq t} g\left(\frac{w}{\sqrt{2\eps }}\right)g\left(\frac{\widehat w+ \widehat b/\sqrt n}{\sqrt{2\eps }}\right)dw + o(1), 
	\end{align}
	where the final $o(1)$ error term comes from the summation approximation. We highlight that all the error terms are uniform in $a$ and $b$.	
	Finally, setting 
	\begin{equation}
		\alpha_\eps\left(a,b\right)
		\coloneqq
		\frac{1}{\eps^{p+2}}\cdot\frac{c_2^2}{c_3}\cdot\int_{w:w- a\in\R^2_{\geq 0}} g\left(\frac{w}{\sqrt{2\eps}}\right)g\left(\frac{\widehat w+ \widehat b}{\sqrt{2\eps}}\right)dw,
	\end{equation}
	we obtain that
	\begin{align}
		\left|\alpha_{n,\lfloor n\eps\rfloor}(a,b) - \alpha_\eps\left(\tfrac a{\sqrt n},\tfrac b{\sqrt n}\right)\right|
		&= O(1)n^{p+1}R_n + o(1) + O(1)\int_{|w|> t} g(w/\sqrt{2 \eps}).
	\end{align}
	Since $g$ is integrable and we have the bounds in \cref{eq:eibdiuwebdowebdbiew}, we can conclude that the latter term tends to zero first taking $n\to\infty$ and then $t\to\infty$.
\end{proof}

\section{Simulations of the skew Brownian permuton}\label{sect:simul}

We briefly explain how we obtained the simulations of the skew Brownian permuton $\bm{\mu}_{\rho,q}$ given in \cref{fig:uyievievbeee}, p.\ \pageref{fig:uyievievbeee}.

\begin{itemize}
	\item The first step is to sample an approximation of a two-dimensional Brownian excursion $\conti{E}_\rho$ with correlations $\rho$. To do that, 
	\begin{itemize}
		\item we start with (any) two-dimensional walk in the non-negative quadrant, started and ended at $(0,0)$,  with (say) 10 points and linear interpolating among these points. Then we run the following Glauber dynamics: we resample each point of the walk in such a way that the conditional law of each point given the neighboring points is that of the Gaussian distribution with correct mean and variance to correspond to the desired Brownian excursion, then conditioned to be in the non-negative quadrant.
		\item To get a new walk with double points, we consider the walk obtained in the previous step and then we add a new point in the middle point of each linear segment. Then we run again the same  Glauber dynamics used in the previous step. 
		\item Iterating the previous step, we always get a new walk with double points, which is a better approximation of two-dimensional Brownian excursion $\conti{E}_\rho$.
	\end{itemize}

	\item The second step is to construct approximations of the solutions of the SDEs in \cref{eq:flow_SDE_gen}, p.\ \pageref{eq:flow_SDE_gen}, driven by the approximation of $\conti{E}_\rho$ obtained in the previous step. This can be done by using some well-chosen \emph{discrete coalescent-walk processes} (see \cite[Section 2.2]{borga2021strongBaxter}) whose walks converge to the solutions of the SDEs in \cref{eq:flow_SDE_gen} (in the same spirit of \cite[Proposition 4.8]{borga2021strongBaxter}).
	
	\item The final step is the simplest one: once we have the approximated solutions of the SDEs in \cref{eq:flow_SDE_gen}, it is simple to construct the process $\varphi_{\cnz_{\rho,q}}$ defined in \cref{eq:random_skew_function}. Finally, the skew Brownian permuton $\bm{\mu}_{\rho,q}$ is well approximated by the graph of the function of the process $\varphi_{\cnz_{\rho,q}}$.
\end{itemize}

\bibliography{mybib}

@article { borga2020scaling,
    AUTHOR = {Borga, Jacopo and Maazoun, Micka\"{e}l},
     TITLE = {Scaling and local limits of {B}axter permutations and bipolar
              orientations through coalescent-walk processes},
   JOURNAL = {Ann. Probab.},
  FJOURNAL = {The Annals of Probability},
    VOLUME = {50},
      YEAR = {2022},
    NUMBER = {4},
     PAGES = {1359--1417},
      ISSN = {0091-1798},
   MRCLASS = {60C05 (05A05 60J67)},
  MRNUMBER = {4420422},
       DOI = {10.1214/21-aop1559},
       URL = {https://doi.org/10.1214/21-aop1559},
}

@article { bassino2017universal,
    AUTHOR = {Bassino, Fr\'{e}d\'{e}rique and Bouvel, Mathilde and F\'{e}ray, Valentin
              and Gerin, Lucas and Maazoun, Micka\"{e}l and Pierrot, Adeline},
     TITLE = {Universal limits of substitution-closed permutation classes},
   JOURNAL = {J. Eur. Math. Soc. (JEMS)},
  FJOURNAL = {Journal of the European Mathematical Society (JEMS)},
    VOLUME = {22},
      YEAR = {2020},
    NUMBER = {11},
     PAGES = {3565--3639},
      ISSN = {1435-9855},
   MRCLASS = {60C05 (05A05 05A15 60F05)},
  MRNUMBER = {4167015},
       DOI = {10.4171/jems/993},
       URL = {https://doi.org/10.4171/jems/993},
}

@article {MR3882946,
	AUTHOR = {Bouvel, Mathilde and Guerrini, Veronica and Rechnitzer, Andrew
	and Rinaldi, Simone},
	TITLE = {Semi-{B}axter and strong-{B}axter: two relatives of the
	{B}axter sequence},
	JOURNAL = {SIAM J. Discrete Math.},
	FJOURNAL = {SIAM Journal on Discrete Mathematics},
	VOLUME = {32},
	YEAR = {2018},
	NUMBER = {4},
	PAGES = {2795--2819},
	ISSN = {0895-4801},
	MRCLASS = {05A15 (05A16)},
	MRNUMBER = {3882946},
	MRREVIEWER = {Malvina George Vamvakari},
	DOI = {10.1137/17M1126734},
	URL = {https://doi.org/10.1137/17M1126734},
}

@article{gwynne2016joint,
  title={Joint scaling limit of a bipolar-oriented triangulation and its dual in the peanosphere sense},
  author={Gwynne, Ewain and Holden, Nina and Sun, Xin},
  journal={arXiv preprint:1603.01194},
  year={2016}
}

@article {MR4102254,
    AUTHOR = {Duraj, Jetlir and Wachtel, Vitali},
     TITLE = {Invariance principles for random walks in cones},
   JOURNAL = {Stochastic Process. Appl.},
  FJOURNAL = {Stochastic Processes and their Applications},
    VOLUME = {130},
      YEAR = {2020},
    NUMBER = {7},
     PAGES = {3920--3942},
      ISSN = {0304-4149},
   MRCLASS = {60G50 (60F17 60G40)},
  MRNUMBER = {4102254},
MRREVIEWER = {Doncho S. Donchev},
       DOI = {10.1016/j.spa.2019.11.004},
       URL = {https://doi.org/10.1016/j.spa.2019.11.004},
}

@article {MR3342657,
    AUTHOR = {Denisov, Denis and Wachtel, Vitali},
     TITLE = {Random walks in cones},
   JOURNAL = {Ann. Probab.},
  FJOURNAL = {The Annals of Probability},
    VOLUME = {43},
      YEAR = {2015},
    NUMBER = {3},
     PAGES = {992--1044},
      ISSN = {0091-1798},
   MRCLASS = {60G50 (60F17 60G40)},
  MRNUMBER = {3342657},
MRREVIEWER = {Kilian Raschel},
       DOI = {10.1214/13-AOP867},
       URL = {https://doi.org/10.1214/13-AOP867},
}

@article {MR2280299,
    AUTHOR = {Lejay, Antoine},
     TITLE = {On the constructions of the skew {B}rownian motion},
   JOURNAL = {Probab. Surv.},
  FJOURNAL = {Probability Surveys},
    VOLUME = {3},
      YEAR = {2006},
     PAGES = {413--466},
   MRCLASS = {60J60 (60H10 60J55 60J65)},
  MRNUMBER = {2280299},
MRREVIEWER = {Daniel Ocone},
       DOI = {10.1214/154957807000000013},
       URL = {https://doi.org/10.1214/154957807000000013},
}

@article {MR4115736,
	AUTHOR = {Borga, Jacopo and Bouvel, Mathilde and F\'{e}ray, Valentin and
	Stufler, Benedikt},
	TITLE = {A decorated tree approach to random permutations in
	substitution-closed classes},
	JOURNAL = {Electron. J. Probab.},
	FJOURNAL = {Electronic Journal of Probability},
	VOLUME = {25},
	YEAR = {2020},
	PAGES = {Paper No. 67, 52},
	MRCLASS = {60C05 (05A05)},
	MRNUMBER = {4115736},
	DOI = {10.1214/20-ejp469},
	URL = {https://doi.org/10.1214/20-ejp469},
}

@book {MR1121940,
    AUTHOR = {Karatzas, Ioannis and Shreve, Steven E.},
     TITLE = {Brownian motion and stochastic calculus},
    SERIES = {Graduate Texts in Mathematics},
    VOLUME = {113},
   EDITION = {Second},
 PUBLISHER = {Springer-Verlag, New York},
      YEAR = {1991},
     PAGES = {xxiv+470},
      ISBN = {0-387-97655-8},
   MRCLASS = {60J65 (35K99 35R60 60G44 60H10 60J60)},
  MRNUMBER = {1121940},
       DOI = {10.1007/978-1-4612-0949-2},
       URL = {https://doi.org/10.1007/978-1-4612-0949-2},
}

@article {MR3055262,
    AUTHOR = {Fernholz, E. Robert and Ichiba, Tomoyuki and Karatzas, Ioannis
              and Prokaj, Vilmos},
     TITLE = {Planar diffusions with rank-based characteristics and
              perturbed {T}anaka equations},
   JOURNAL = {Probab. Theory Related Fields},
  FJOURNAL = {Probability Theory and Related Fields},
    VOLUME = {156},
      YEAR = {2013},
    NUMBER = {1-2},
     PAGES = {343--374},
      ISSN = {0178-8051},
   MRCLASS = {60H10 (60G44 60J55 60J60)},
  MRNUMBER = {3055262},
MRREVIEWER = {Pavel Chigansky},
       DOI = {10.1007/s00440-012-0430-7},
       URL = {https://doi.org/10.1007/s00440-012-0430-7},
}

@article {MR326840,
    AUTHOR = {Nakao, Shintaro},
     TITLE = {On the pathwise uniqueness of solutions of one-dimensional
              stochastic differential equations},
   JOURNAL = {Osaka Math. J.},
  FJOURNAL = {Osaka Mathematical Journal},
    VOLUME = {9},
      YEAR = {1972},
     PAGES = {513--518},
      ISSN = {0388-0699},
   MRCLASS = {60H10},
  MRNUMBER = {326840},
MRREVIEWER = {Richard S. Ellis},
       URL = {http://projecteuclid.org/euclid.ojm/1200693939},
}

@incollection {MR770393,
    AUTHOR = {Le Gall, Jean-Fran\c{c}ois},
     TITLE = {Applications du temps local aux \'{e}quations diff\'{e}rentielles
              stochastiques unidimensionnelles},
 BOOKTITLE = {Seminar on probability, {XVII}},
    SERIES = {Lecture Notes in Math.},
    VOLUME = {986},
     PAGES = {15--31},
 PUBLISHER = {Springer, Berlin},
      YEAR = {1983},
   MRCLASS = {60H10 (60J55)},
  MRNUMBER = {770393},
MRREVIEWER = {Toshio Yamada},
       DOI = {10.1007/BFb0068296},
       URL = {https://doi.org/10.1007/BFb0068296},
}

@incollection {MR777514,
    AUTHOR = {Le Gall, Jean-Fran\c{c}ois},
     TITLE = {One-dimensional stochastic differential equations involving
              the local times of the unknown process},
 BOOKTITLE = {Stochastic analysis and applications ({S}wansea, 1983)},
    SERIES = {Lecture Notes in Math.},
    VOLUME = {1095},
     PAGES = {51--82},
 PUBLISHER = {Springer, Berlin},
      YEAR = {1984},
   MRCLASS = {60H10 (60J55)},
  MRNUMBER = {777514},
MRREVIEWER = {Toshio Yamada},
       DOI = {10.1007/BFb0099122},
       URL = {https://doi.org/10.1007/BFb0099122},
}

@article {MR606993,
    AUTHOR = {Harrison, John M. and Shepp, Larry A.},
     TITLE = {On skew {B}rownian motion},
   JOURNAL = {Ann. Probab.},
  FJOURNAL = {The Annals of Probability},
    VOLUME = {9},
      YEAR = {1981},
    NUMBER = {2},
     PAGES = {309--313},
      ISSN = {0091-1798},
   MRCLASS = {60J60 (60H10 60J65)},
  MRNUMBER = {606993},
MRREVIEWER = {J. Walsh},
       URL =
              {http://links.jstor.org/sici?sici=0091-1798(198104)9:2<309:OSBM>2.0.CO;2-W&origin=MSN},
}

@article {MR3882190,
	AUTHOR = {\c{C}a\u{g}lar, Mine and Hajri, Hatem and Karaku\c{s}, Abdullah Harun},
	TITLE = {Correlated coalescing {B}rownian flows on {$\mathbb R$} and the
	circle},
	JOURNAL = {ALEA Lat. Am. J. Probab. Math. Stat.},
	FJOURNAL = {ALEA. Latin American Journal of Probability and Mathematical
	Statistics},
	VOLUME = {15},
	YEAR = {2018},
	NUMBER = {2},
	PAGES = {1447--1464},
	MRCLASS = {60H10 (60J60)},
	MRNUMBER = {3882190},
	MRREVIEWER = {Olga V. Aryasova},
	DOI = {10.30757/alea.v15-54},
	URL = {https://doi.org/10.30757/alea.v15-54},
}

@article {MR2235172,
	AUTHOR = {Le Jan, Yves and Raimond, Olivier},
	TITLE = {Flows associated to {T}anaka's {SDE}},
	JOURNAL = {ALEA Lat. Am. J. Probab. Math. Stat.},
	FJOURNAL = {ALEA. Latin American Journal of Probability and Mathematical
	Statistics},
	VOLUME = {1},
	YEAR = {2006},
	PAGES = {21--34},
	MRCLASS = {60H10 (60H40)},
	MRNUMBER = {2235172},
	MRREVIEWER = {David Nualart},
}

@article {MR4112725,
	AUTHOR = {Le Jan, Yves and Raimond, Olivier},
	TITLE = {Flows, coalescence and noise. {A} correction},
	JOURNAL = {Ann. Probab.},
	FJOURNAL = {The Annals of Probability},
	VOLUME = {48},
	YEAR = {2020},
	NUMBER = {3},
	PAGES = {1592--1595},
	ISSN = {0091-1798},
	MRCLASS = {60H10 (60G51 60H40)},
	MRNUMBER = {4112725},
	DOI = {10.1214/19-AOP1394},
	URL = {https://doi.org/10.1214/19-AOP1394},
}

@article {MR2060298,
	AUTHOR = {Le Jan, Yves and Raimond, Olivier},
	TITLE = {Flows, coalescence and noise},
	JOURNAL = {Ann. Probab.},
	FJOURNAL = {The Annals of Probability},
	VOLUME = {32},
	YEAR = {2004},
	NUMBER = {2},
	PAGES = {1247--1315},
	ISSN = {0091-1798},
	MRCLASS = {60H10 (60G51 60H40 76F05 76M35)},
	MRNUMBER = {2060298},
	MRREVIEWER = {Luis G. Gorostiza},
	DOI = {10.1214/009117904000000207},
	URL = {https://doi.org/10.1214/009117904000000207},
}

@article {maazoun,
	AUTHOR = {Maazoun, Micka\"{e}l},
	TITLE = {On the {B}rownian separable permuton},
	JOURNAL = {Combin. Probab. Comput.},
	FJOURNAL = {Combinatorics, Probability and Computing},
	VOLUME = {29},
	YEAR = {2020},
	NUMBER = {2},
	PAGES = {241--266},
	ISSN = {0963-5483},
	MRCLASS = {60C05 (05A05 60G18 60G57)},
	MRNUMBER = {4079636},
	DOI = {10.1017/s0963548319000300},
	URL = {https://doi.org/10.1017/s0963548319000300},
}

@article { alon2021runsort,
    AUTHOR = {Alon, Noga and Defant, Colin and Kravitz, Noah},
     TITLE = {The runsort permuton},
   JOURNAL = {Adv. in Appl. Math.},
  FJOURNAL = {Advances in Applied Mathematics},
    VOLUME = {139},
      YEAR = {2022},
     PAGES = {Paper No. 102361, 18},
      ISSN = {0196-8858},
   MRCLASS = {05A05 (60F05)},
  MRNUMBER = {4420996},
       DOI = {10.1016/j.aam.2022.102361},
       URL = {https://doi.org/10.1016/j.aam.2022.102361},
}

@article{borga2021random,
	title={Random Permutations -- A geometric point of view},
	author={Borga, Jacopo},
	journal={arXiv preprint:2107.09699 (Ph.D. Thesis)},
	year={2021}
}

@incollection {MR2732835,
	AUTHOR = {Presutti, Cathleen B. and Stromquist, Walter},
	TITLE = {Packing rates of measures and a conjecture for the packing
	density of 2413},
	BOOKTITLE = {Permutation patterns},
	SERIES = {London Math. Soc. Lecture Note Ser.},
	VOLUME = {376},
	PAGES = {287--316},
	PUBLISHER = {Cambridge Univ. Press, Cambridge},
	YEAR = {2010},
	MRCLASS = {05A05 (05D99 28A99)},
	MRNUMBER = {2732835},
	MRREVIEWER = {Vincent R. Vatter},
	DOI = {10.1017/CBO9780511902499.015},
	URL = {https://doi.org/10.1017/CBO9780511902499.015},
}

@article {hoppen2013limits,
	AUTHOR = {Hoppen, Carlos and Kohayakawa, Yoshiharu and Moreira, Carlos
	Gustavo and R\'{a}th, Bal\'{a}zs and Menezes Sampaio, Rudini},
	TITLE = {Limits of permutation sequences},
	JOURNAL = {J. Combin. Theory Ser. B},
	FJOURNAL = {Journal of Combinatorial Theory. Series B},
	VOLUME = {103},
	YEAR = {2013},
	NUMBER = {1},
	PAGES = {93--113},
	ISSN = {0095-8956},
	MRCLASS = {05A05 (60B10 60C05)},
	MRNUMBER = {2995721},
	MRREVIEWER = {Sevak Mkrtchyan},
	DOI = {10.1016/j.jctb.2012.09.003},
	URL = {https://doi.org/10.1016/j.jctb.2012.09.003},
}

@article {MR3279390,
	AUTHOR = {Glebov, Roman and Grzesik, Andrzej and Klimo\v{s}ov\'{a}, Tereza and
	Kr\'{a}l', Daniel},
	TITLE = {Finitely forcible graphons and permutons},
	JOURNAL = {J. Combin. Theory Ser. B},
	FJOURNAL = {Journal of Combinatorial Theory. Series B},
	VOLUME = {110},
	YEAR = {2015},
	PAGES = {112--135},
	ISSN = {0095-8956},
	MRCLASS = {05C80 (05A05)},
	MRNUMBER = {3279390},
	MRREVIEWER = {Lyuben R. Mutafchiev},
	DOI = {10.1016/j.jctb.2014.07.007},
	URL = {https://doi.org/10.1016/j.jctb.2014.07.007},
}

@article {starr2009thermodynamic,
	AUTHOR = {Starr, Shannon},
	TITLE = {Thermodynamic limit for the {M}allows model on {$S_n$}},
	JOURNAL = {J. Math. Phys.},
	FJOURNAL = {Journal of Mathematical Physics},
	VOLUME = {50},
	YEAR = {2009},
	NUMBER = {9},
	PAGES = {095208, 15},
	ISSN = {0022-2488},
	MRCLASS = {82B31 (60K35 60K40 82B20)},
	MRNUMBER = {2566888},
	MRREVIEWER = {Arnaud Le Ny},
	DOI = {10.1063/1.3156746},
	URL = {https://doi.org/10.1063/1.3156746},
}

@article{ dauvergne2018archimedean,
  title={The Archimedean limit of random sorting networks},
  author={Dauvergne, Duncan},
  journal={Journal of the American Mathematical Society},
  year={2021}
}

@article {MR2266895,
	AUTHOR = {Romik, Dan},
	TITLE = {Permutations with short monotone subsequences},
	JOURNAL = {Adv. in Appl. Math.},
	FJOURNAL = {Advances in Applied Mathematics},
	VOLUME = {37},
	YEAR = {2006},
	NUMBER = {4},
	PAGES = {501--510},
	ISSN = {0196-8858},
	MRCLASS = {05A05 (05D99 05E10)},
	MRNUMBER = {2266895},
	MRREVIEWER = {Rebecca Nicole Smith},
	DOI = {10.1016/j.aam.2005.08.008},
	URL = {https://doi.org/10.1016/j.aam.2005.08.008},
}

@article {bassino2018separable,
	AUTHOR = {Bassino, Fr\'{e}d\'{e}rique and Bouvel, Mathilde and F\'{e}ray, Valentin
	and Gerin, Lucas and Pierrot, Adeline},
	TITLE = {The {B}rownian limit of separable permutations},
	JOURNAL = {Ann. Probab.},
	FJOURNAL = {The Annals of Probability},
	VOLUME = {46},
	YEAR = {2018},
	NUMBER = {4},
	PAGES = {2134--2189},
	ISSN = {0091-1798},
	MRCLASS = {60C05 (05A05)},
	MRNUMBER = {3813988},
	MRREVIEWER = {Kent E. Morrison},
	DOI = {10.1214/17-AOP1223},
	URL = {https://doi.org/10.1214/17-AOP1223},
}

@article { bassino2019scaling,
    AUTHOR = {Bassino, Fr\'{e}d\'{e}rique and Bouvel, Mathilde and F\'{e}ray, Valentin
              and Gerin, Lucas and Maazoun, Micka\"{e}l and Pierrot, Adeline},
     TITLE = {Scaling limits of permutation classes with a finite
              specification: {A} dichotomy},
   JOURNAL = {Adv. Math.},
  FJOURNAL = {Advances in Mathematics},
    VOLUME = {405},
      YEAR = {2022},
     PAGES = {Paper No. 108513},
      ISSN = {0001-8708},
   MRCLASS = {60C05 (05A05)},
  MRNUMBER = {4437614},
       DOI = {10.1016/j.aim.2022.108513},
       URL = {https://doi.org/10.1016/j.aim.2022.108513},
}

@article{gwynne2019mating,
	title={Mating of trees for random planar maps and {L}iouville quantum gravity: a survey},
	author={ Gwynne, Ewain and Holden, Nina and Sun, Xin
},
	journal={arXiv preprint: 1910.04713},
	year={2019}
}

@article {MR4010949,
    AUTHOR = {Miller, Jason and Sheffield, Scott},
     TITLE = {Liouville quantum gravity spheres as matings of
              finite-diameter trees},
   JOURNAL = {Ann. Inst. Henri Poincar\'{e} Probab. Stat.},
  FJOURNAL = {Annales de l'Institut Henri Poincar\'{e} Probabilit\'{e}s et
              Statistiques},
    VOLUME = {55},
      YEAR = {2019},
    NUMBER = {3},
     PAGES = {1712--1750},
      ISSN = {0246-0203},
   MRCLASS = {60J67 (28C20)},
  MRNUMBER = {4010949},
       DOI = {10.1214/18-aihp932},
       URL = {https://doi.org/10.1214/18-aihp932},
}

@article {MR3719057,
    AUTHOR = {Miller, Jason and Sheffield, Scott},
     TITLE = {Imaginary geometry {IV}: interior rays, whole-plane
              reversibility, and space-filling trees},
   JOURNAL = {Probab. Theory Related Fields},
  FJOURNAL = {Probability Theory and Related Fields},
    VOLUME = {169},
      YEAR = {2017},
    NUMBER = {3-4},
     PAGES = {729--869},
      ISSN = {0178-8051},
   MRCLASS = {60J67 (60K35)},
  MRNUMBER = {3719057},
MRREVIEWER = {Ben Dyhr},
       DOI = {10.1007/s00440-017-0780-2},
       URL = {https://doi.org/10.1007/s00440-017-0780-2},
}

@article { duplantier2014liouville,
    AUTHOR = {Duplantier, Bertrand and Miller, Jason and Sheffield, Scott},
     TITLE = {Liouville quantum gravity as a mating of trees},
   JOURNAL = {Ast\'{e}risque},
  FJOURNAL = {Ast\'{e}risque},
    NUMBER = {427},
      YEAR = {2021},
     PAGES = {viii+257},
      ISSN = {0303-1179},
      ISBN = {978-2-85629-941-8},
   MRCLASS = {60-02 (60D05 60J67 82B20 82B41 83C45)},
  MRNUMBER = {4340069},
       DOI = {10.24033/ast},
       URL = {https://doi.org/10.24033/ast},
}

@article {MR3098074,
	AUTHOR = {Prokaj, Vilmos},
	TITLE = {The solution of the perturbed {T}anaka-equation is pathwise
	unique},
	JOURNAL = {Ann. Probab.},
	FJOURNAL = {The Annals of Probability},
	VOLUME = {41},
	YEAR = {2013},
	NUMBER = {3B},
	PAGES = {2376--2400},
	ISSN = {0091-1798},
	MRCLASS = {60H20 (60G44 60H10 60J55 60J65)},
	MRNUMBER = {3098074},
	MRREVIEWER = {Martin Ondrej\'{a}t},
	DOI = {10.1214/11-AOP716},
	URL = {https://doi.org/10.1214/11-AOP716},
}

@article {MR0184217,
    AUTHOR = {Baxter, Glen E.},
     TITLE = {On fixed points of the composite of commuting functions},
   JOURNAL = {Proc. Amer. Math. Soc.},
  FJOURNAL = {Proceedings of the American Mathematical Society},
    VOLUME = {15},
      YEAR = {1964},
     PAGES = {851--855},
   MRCLASS = {54.85},
  MRNUMBER = {0184217},
MRREVIEWER = {E. Dyer},
}

@article {MR0250516,
    AUTHOR = {Boyce, William M.},
     TITLE = {Generation of a class of permutations associated with
              commuting functions},
   JOURNAL = {Math. Algorithms 2 (1967), 19--26; addendum, ibid.},
    VOLUME = {3},
      YEAR = {1967},
     PAGES = {25--26},
   MRCLASS = {68.00 (47.00)},
  MRNUMBER = {0250516},
MRREVIEWER = {F. J. Murray},
}

@article {MR491652,
    AUTHOR = {Chung, Fan-Rong K. and Graham, Ronald L. and Hoggatt, Jr., Verner Emil and
              Kleiman, Mark},
     TITLE = {The number of {B}axter permutations},
   JOURNAL = {J. Combin. Theory Ser. A},
  FJOURNAL = {Journal of Combinatorial Theory. Series A},
    VOLUME = {24},
      YEAR = {1978},
    NUMBER = {3},
     PAGES = {382--394},
      ISSN = {0097-3165},
   MRCLASS = {05A15 (05A05 54H25)},
  MRNUMBER = {491652},
MRREVIEWER = {L. H. Harper},
       DOI = {10.1016/0097-3165(78)90068-7},
       URL = {https://doi.org/10.1016/0097-3165(78)90068-7},
}

@article {MR555815,
    AUTHOR = { Mallows, Colin L.},
     TITLE = {Baxter permutations rise again},
   JOURNAL = {J. Combin. Theory Ser. A},
  FJOURNAL = {Journal of Combinatorial Theory. Series A},
    VOLUME = {27},
      YEAR = {1979},
    NUMBER = {3},
     PAGES = {394--396},
   MRCLASS = {05A15 (05A05)},
  MRNUMBER = {555815},
MRREVIEWER = {Fan Chung Graham},
}

@article {MR2763051,
    AUTHOR = {Felsner, Stefan and Fusy, \'{E}ric and Noy, Marc and Orden, David},
     TITLE = {Bijections for {B}axter families and related objects},
   JOURNAL = {J. Combin. Theory Ser. A},
  FJOURNAL = {Journal of Combinatorial Theory. Series A},
    VOLUME = {118},
      YEAR = {2011},
    NUMBER = {3},
     PAGES = {993--1020},
      ISSN = {0097-3165},
   MRCLASS = {05A05 (05A15 05A19 05C10 05C30)},
  MRNUMBER = {2763051},
MRREVIEWER = {Paul Klingsberg},
       DOI = {10.1016/j.jcta.2010.03.017},
       URL = {https://doi.org/10.1016/j.jcta.2010.03.017},
}

@article {MR1620935,
	AUTHOR = {Bose, Prosenjit and Buss, Jonathan F. and Lubiw, Anna},
	TITLE = {Pattern matching for permutations},
	JOURNAL = {Inform. Process. Lett.},
	FJOURNAL = {Information Processing Letters},
	VOLUME = {65},
	YEAR = {1998},
	NUMBER = {5},
	PAGES = {277--283},
	ISSN = {0020-0190},
	MRCLASS = {68Q25 (68R05)},
	MRNUMBER = {1620935},
	DOI = {10.1016/S0020-0190(97)00209-3},
	URL = {https://doi.org/10.1016/S0020-0190(97)00209-3},
}

@article {MR3359905,
	AUTHOR = {Albert, Michael and Homberger, Cheyne and Pantone, Jay},
	TITLE = {Equipopularity classes in the separable permutations},
	JOURNAL = {Electron. J. Combin.},
	FJOURNAL = {Electronic Journal of Combinatorics},
	VOLUME = {22},
	YEAR = {2015},
	NUMBER = {2},
	PAGES = {Paper 2.2, 18},
	MRCLASS = {05A05 (05A15 05A17 05A19)},
	MRNUMBER = {3359905},
	MRREVIEWER = {Robin DaPao Zhou},
	DOI = {10.37236/4797},
	URL = {https://doi.org/10.37236/4797},
}

@article {MR624050,
	AUTHOR = {Avis, David and Newborn, Monroe},
	TITLE = {On pop-stacks in series},
	JOURNAL = {Utilitas Math.},
	FJOURNAL = {Utilitas Mathematica. An International Journal of Discrete and
	Combinatorial Mathematics, and Statistical Design},
	VOLUME = {19},
	YEAR = {1981},
	PAGES = {129--140},
	ISSN = {0315-3681},
	MRCLASS = {68B15 (68E05)},
	MRNUMBER = {624050},
	MRREVIEWER = {I. Kaufmann},
}

@article {MR1093199,
	AUTHOR = {Shapiro, Louis and Stephens, A. B.},
	TITLE = {Bootstrap percolation, the {S}chr\"{o}der numbers, and the
	{$N$}-kings problem},
	JOURNAL = {SIAM J. Discrete Math.},
	FJOURNAL = {SIAM Journal on Discrete Mathematics},
	VOLUME = {4},
	YEAR = {1991},
	NUMBER = {2},
	PAGES = {275--280},
	ISSN = {0895-4801},
	MRCLASS = {05B20 (82C43)},
	MRNUMBER = {1093199},
	DOI = {10.1137/0404025},
	URL = {https://doi.org/10.1137/0404025},
}

@article{li2017schnyder,
	title={Schnyder woods, {S}{L}{E} (16), and {L}iouville quantum gravity},
	author={Li, Yiting and Sun, Xin and Watson, Samuel S},
	journal={arXiv preprint:1705.03573},
	year={2017}
}

@article{ borga2021strongBaxter,
  title={The permuton limit of strong-{B}axter and semi-{B}axter permutations is the skew {B}rownian permuton},
  author={Borga, Jacopo},
  journal={Electronic Journal of Probability},
  volume={27},
  pages={1--53},
  year={2022},
  publisher={The Institute of Mathematical Statistics and the Bernoulli Society}
}

@article{aru2021mating,
	title={Mating of trees for critical {L}iouville quantum gravity},
	author={Aru, Juhan and Holden, Nina and Powell, Ellen and Sun, Xin},
	journal={arXiv preprint:2109.00275},
	year={2021}
}

@article {MR3606760,
    AUTHOR = {Bertoin, Jean},
     TITLE = {Markovian growth-fragmentation processes},
   JOURNAL = {Bernoulli},
  FJOURNAL = {Bernoulli. Official Journal of the Bernoulli Society for
              Mathematical Statistics and Probability},
    VOLUME = {23},
      YEAR = {2017},
    NUMBER = {2},
     PAGES = {1082--1101},
      ISSN = {1350-7265},
   MRCLASS = {60J80 (60G18 60G51)},
  MRNUMBER = {3606760},
MRREVIEWER = {Matthias Winkel},
       DOI = {10.3150/15-BEJ770},
       URL = {https://doi.org/10.3150/15-BEJ770},
}

@incollection{ fusy2021enumeration,
  title={On the enumeration of plane bipolar posets and transversal structures},
  author={Fusy, Eric and Narmanli, Erkan and Schaeffer, Gilles},
  booktitle={Extended Abstracts EuroComb 2021},
  pages={560--566},
  year={2021},
  publisher={Springer}
}

@article {MR2734180,
    AUTHOR = {Bonichon, Nicolas and Bousquet-M\'{e}lou, Mireille and Fusy, \'{E}ric},
     TITLE = {Baxter permutations and plane bipolar orientations},
   JOURNAL = {S\'{e}m. Lothar. Combin.},
  FJOURNAL = {S\'{e}minaire Lotharingien de Combinatoire},
    VOLUME = {61A},
      YEAR = {2009/11},
     PAGES = {Art. B61Ah, 29},
   MRCLASS = {05A15 (05A05 05A19 05C30)},
  MRNUMBER = {2734180},
MRREVIEWER = {Timothy Y. Chow},
}

@article {MR3101730,
    AUTHOR = {Kitaev, Sergey and Salimov, Pavel and Severs, Christopher and
              Ulfarsson, Henning},
     TITLE = {Restricted non-separable planar maps and some pattern avoiding
              permutations},
   JOURNAL = {Discrete Appl. Math.},
  FJOURNAL = {Discrete Applied Mathematics. The Journal of Combinatorial
              Algorithms, Informatics and Computational Sciences},
    VOLUME = {161},
      YEAR = {2013},
    NUMBER = {16-17},
     PAGES = {2514--2526},
      ISSN = {0166-218X},
   MRCLASS = {05C10 (05A05)},
  MRNUMBER = {3101730},
       DOI = {10.1016/j.dam.2013.01.004},
       URL = {https://doi.org/10.1016/j.dam.2013.01.004},
}

@article {MR1394948,
    AUTHOR = {Dulucq, Serge and Gire, Sophie and West, Julian},
     TITLE = {Permutations with forbidden subsequences and nonseparable
              planar maps},
 BOOKTITLE = {Proceedings of the 5th {C}onference on {F}ormal {P}ower
              {S}eries and {A}lgebraic {C}ombinatorics ({F}lorence, 1993)},
   JOURNAL = {Discrete Math.},
  FJOURNAL = {Discrete Mathematics},
    VOLUME = {153},
      YEAR = {1996},
    NUMBER = {1-3},
     PAGES = {85--103},
      ISSN = {0012-365X},
   MRCLASS = {05C30 (05A05 05C10 68P10 68Q20 68R05)},
  MRNUMBER = {1394948},
MRREVIEWER = {W. G. Brown},
       DOI = {10.1016/0012-365X(95)00130-O},
       URL = {https://doi.org/10.1016/0012-365X(95)00130-O},
}

@article {MR1485138,
    AUTHOR = {B{\'o}na, Mikl{\'o}s},
     TITLE = {Exact enumeration of {$1342$}-avoiding permutations: a close
              link with labeled trees and planar maps},
   JOURNAL = {J. Combin. Theory Ser. A},
  FJOURNAL = {Journal of Combinatorial Theory. Series A},
    VOLUME = {80},
      YEAR = {1997},
    NUMBER = {2},
     PAGES = {257--272},
      ISSN = {0097-3165},
   MRCLASS = {05A15 (05C05)},
  MRNUMBER = {1485138},
MRREVIEWER = {Giuseppe Pellegrino},
       DOI = {10.1006/jcta.1997.2800},
       URL = {https://doi.org/10.1006/jcta.1997.2800},
}

@article {MR4149526,
    AUTHOR = {Borga, Jacopo and Slivken, Erik},
     TITLE = {Square permutations are typically rectangular},
   JOURNAL = {Ann. Appl. Probab.},
  FJOURNAL = {The Annals of Applied Probability},
    VOLUME = {30},
      YEAR = {2020},
    NUMBER = {5},
     PAGES = {2196--2233},
      ISSN = {1050-5164},
   MRCLASS = {60C05 (05A05)},
  MRNUMBER = {4149526},
MRREVIEWER = {Kent E. Morrison},
       DOI = {10.1214/19-AAP1555},
       URL = {https://doi.org/10.1214/19-AAP1555},
}

@article{borga2021almost,
  title={Almost square permutations are typically square},
  author={Borga, Jacopo and Duchi, Enrica and Slivken, Erik},
  journal={Annales de l'Institut Henri Poincar{\'e}, Probabilit{\'e}s et Statistiques},
  volume={57},
  number={4},
  pages={1834--1856},
  year={2021},
  organization={Institut Henri Poincar{\'e}}
}

@article {MR3817550,
    AUTHOR = {Starr, Shannon and Walters, Meg},
     TITLE = {Phase uniqueness for the {M}allows measure on permutations},
   JOURNAL = {J. Math. Phys.},
  FJOURNAL = {Journal of Mathematical Physics},
    VOLUME = {59},
      YEAR = {2018},
    NUMBER = {6},
     PAGES = {063301, 28},
      ISSN = {0022-2488},
   MRCLASS = {60B15 (60F10 60K35 81R12 82B26)},
  MRNUMBER = {3817550},
       DOI = {10.1063/1.5017924},
       URL = {https://doi.org/10.1063/1.5017924},
}

@book {lovasz2012large,
    AUTHOR = {Lov\'{a}sz, L\'{a}szl\'{o}},
     TITLE = {Large networks and graph limits},
    SERIES = {American Mathematical Society Colloquium Publications},
    VOLUME = {60},
 PUBLISHER = {American Mathematical Society, Providence, RI},
      YEAR = {2012},
     PAGES = {xiv+475},
      ISBN = {978-0-8218-9085-1},
   MRCLASS = {05-02 (05C60 05C80 05C82 05D40)},
  MRNUMBER = {3012035},
MRREVIEWER = {Anant P. Godbole},
       DOI = {10.1090/coll/060},
       URL = {https://doi.org/10.1090/coll/060},
}

@article{bassino2021linear,
  title={Linear-sized independent sets in random cographs and increasing subsequences in separable permutations},
  author={Bassino, Fr{\'e}d{\'e}rique and Bouvel, Mathilde and Drmota, Michael and F{\'e}ray, Valentin and Gerin, Lucas and Maazoun, Micka{\"e}l and Pierrot, Adeline},
  journal={arXiv preprint:2104.07444},
  year={2021}
}

@article {MR4219068,
    AUTHOR = {Bousquet-M\'{e}lou, Mireille and Fusy, \'{E}ric and Raschel, Kilian},
     TITLE = {Plane bipolar orientations and quadrant walks},
   JOURNAL = {S\'{e}m. Lothar. Combin.},
  FJOURNAL = {S\'{e}minaire Lotharingien de Combinatoire},
    VOLUME = {81},
      YEAR = {2020},
     PAGES = {Art. B81l, 64},
   MRCLASS = {05C10 (05C30 60F17 60G50)},
  MRNUMBER = {4219068},
}

@article {MR2835247,
    AUTHOR = {Hajri, Hatem},
     TITLE = {Stochastic flows related to {W}alsh {B}rownian motion},
   JOURNAL = {Electron. J. Probab.},
  FJOURNAL = {Electronic Journal of Probability},
    VOLUME = {16},
      YEAR = {2011},
     PAGES = {no. 58, 1563--1599},
   MRCLASS = {60H25 (60H20 60J60)},
  MRNUMBER = {2835247},
MRREVIEWER = {S. Ramasubramanian},
       DOI = {10.1214/EJP.v16-924},
       URL = {https://doi.org/10.1214/EJP.v16-924},
}

@article {MR1816931,
    AUTHOR = {Watanabe, S.},
     TITLE = {The stochastic flow and the noise associated to {T}anaka's
              stochastic differential equation},
   JOURNAL = {Ukra\"{\i}n. Mat. Zh.},
  FJOURNAL = {Nats\={\i}onal\cprime na Akadem\={\i}ya Nauk Ukra\"{\i}ni. \={I}nstitut Matematiki.
              Ukra\"{\i}ns\cprime ki\u{\i} Matematichni\u{\i} Zhurnal},
    VOLUME = {52},
      YEAR = {2000},
    NUMBER = {9},
     PAGES = {1176--1193},
      ISSN = {0041-6053},
   MRCLASS = {60H10},
  MRNUMBER = {1816931},
MRREVIEWER = {Olivier Raimond},
       DOI = {10.1023/A:1010319800544},
       URL = {https://doi.org/10.1023/A:1010319800544},
}

@article {MR3314651,
    AUTHOR = {Hajri, Hatem},
     TITLE = {On flows associated to {T}anaka's {SDE} and related works},
   JOURNAL = {Electron. Commun. Probab.},
  FJOURNAL = {Electronic Communications in Probability},
    VOLUME = {20},
      YEAR = {2015},
     PAGES = {no. 16, 12},
   MRCLASS = {60H25 (60J60)},
  MRNUMBER = {3314651},
MRREVIEWER = {Diego Sebastian Ledesma},
       DOI = {10.1214/ECP.v20-4058},
       URL = {https://doi.org/10.1214/ECP.v20-4058},
}

@article {MR2094439,
    AUTHOR = {Burdzy, Krzysztof and Kaspi, Haya},
     TITLE = {Lenses in skew {B}rownian flow},
   JOURNAL = {Ann. Probab.},
  FJOURNAL = {The Annals of Probability},
    VOLUME = {32},
      YEAR = {2004},
    NUMBER = {4},
     PAGES = {3085--3115},
      ISSN = {0091-1798},
   MRCLASS = {60J65 (60J55)},
  MRNUMBER = {2094439},
MRREVIEWER = {Robert G. Smits},
       DOI = {10.1214/009117904000000711},
       URL = {https://doi.org/10.1214/009117904000000711},
}

@incollection {MR2953347,
    AUTHOR = {Hajri, Hatem},
     TITLE = {Discrete approximations to solution flows of {T}anaka's {SDE} related to {W}alsh {B}rownian motion},
 BOOKTITLE = {S\'{e}minaire de {P}robabilit\'{e}s {XLIV}},
    SERIES = {Lecture Notes in Math.},
    VOLUME = {2046},
     PAGES = {167--190},
 PUBLISHER = {Springer, Heidelberg},
      YEAR = {2012},
   MRCLASS = {60H10 (60J65)},
  MRNUMBER = {2953347},
MRREVIEWER = {Mario Abundo},
       DOI = {10.1007/978-3-642-27461-9\_8},
       URL = {https://doi.org/10.1007/978-3-642-27461-9_8},
}

@article{borga2022baxter,
  title={Baxter permuton and {L}iouville quantum gravity},
  author={Borga, Jacopo and Holden, Nina and Sun, Xin and Yu, Pu},
  journal={Probability Theory and Related Fields},
  pages={1--49},
  year={2023},
  publisher={Springer}
}

@article{borga2022meanders,
  title={Permutons, meanders, and {S}{L}{E}-decorated {L}iouville quantum gravity},
  author={Borga, Jacopo and Gwynne, Ewain and Sun, Xin},
  journal={arXiv preprint: 2207.02319},
  year={2022}
}
\bibliographystyle{alpha}

\end{document}